\documentclass{article}
\usepackage{braket,amsfonts}
\usepackage{array}
\usepackage[caption=false]{subfig}
\usepackage[top=1in,bottom=1in,left=1in,right=1in]{geometry}

\usepackage{algorithmic}
\usepackage{graphicx,epstopdf}

\usepackage{bold-extra}
\usepackage{amsthm,amsmath}
\usepackage{mathrsfs}
\usepackage{comment}
\usepackage{appendix}
\usepackage{bbm}
\usepackage{algorithm,algorithmic}
\usepackage[colorlinks=true, allcolors=blue]{hyperref}

\usepackage{dsfont}
\usepackage{mathtools}
\usepackage{booktabs}
\newcommand{\R}{\mathbb{R}}

\newcommand{\E}{\mathbb{E}}

\newcommand{\dif}{\mathop{}\!\mathrm{d}}

\theoremstyle{plain}
\newtheorem{theorem}{Theorem}[section]

\newtheorem{corollary}[theorem]{Corollary}

\theoremstyle{definition}
\newtheorem{definition}[theorem]{Definition}

\theoremstyle{remark}
\newtheorem{remark}[theorem]{Remark}

\numberwithin{equation}{section}

\title{Multi-Agent Relative Investment Games in a Jump Diffusion Market with Deep Reinforcement Learning Algorithm}

\author{Liwei Lu\thanks{Department of Mathematical Science, Tsinghua University, Beijing, 100084, China, \em{llw20@mails.tsinghua.edu.cn}.} \and Ruimeng Hu\thanks{Department of Mathematics, and Department of Statistics and Applied Probability, University of California, Santa Barbara, CA 93106-3080, USA, \em{rhu@ucsb.edu}.} \and Xu Yang\thanks{Department of Mathematics, University of California, Santa Barbara, CA 93106-3080, USA, \em{xuyang@math.ucsb.edu}.} \and Yi Zhu\thanks{Yau Mathematical Sciences Center, Tsinghua University, Beijing, 100084, China and Beijing Institute of Mathematical Sciences and Applications, Beijing, 101408, China, \em{yizhu@tsinghua.edu.cn}.}}

\begin{document}
\maketitle

\begin{abstract}
    This paper focuses on multi-agent stochastic differential games for jump-diffusion systems. On one hand, we study the multi-agent game for optimal investment in a jump-diffusion market. We derive constant Nash equilibria and provide sufficient conditions for their existence and uniqueness for exponential, power, and logarithmic utilities, respectively. On the other hand, we introduce a computational framework based on the actor-critic method in deep reinforcement learning to solve the stochastic control problem with jumps. We extend this algorithm to address the multi-agent game with jumps and utilize parallel computing to enhance computational efficiency. We present numerical examples of the Merton problem with jumps, linear quadratic regulators, and the optimal investment game under various settings to demonstrate the accuracy, efficiency, and robustness of the proposed method. In particular, neural network solutions numerically converge to the derived constant Nash equilibrium for the multi-agent game.

\end{abstract}


\section{Introduction}

Real-world financial markets often experience sudden price movements due to abrupt changes in market sentiment, liquidity constraints, or trading activities during specific periods. These discontinuous behaviors can occur across various time scales \cite{rama2004financial}. Modeling them with jump processes not only allows us to explicitly capture these sudden price changes, leading to a more accurate representation of market behavior, but also enables us to better capture the heavy-tailed behavior of asset price returns \cite{levy2000renewal}, volatility clustering observed in financial markets, and other phenomena. All of these motivate us to employ jump models in studying financial problems. 

In this paper, we focus on analyzing multi-agent investment games within a jump-diffusion market, alongside developing efficient deep reinforcement learning (RL) algorithms. In the first part, building on the framework and notion of \cite{lacker2019mean}, we consider the game where  $n$-agents invest in a jump-diffusion market with each trading between an individual risky asset and a common riskless asset. The objective is to maximize their utility, which depends on both their own states and their peer's performance. We analyze the game under both constant absolute risk aversion (CARA) and constant relative risk aversion (CRRA),  deriving a semi-explicit Nash equilibrium in each case.
In the second part, our goal is to develop efficient algorithms for solving games where the states are modeled by general It\^o-L\'evy processes, and both the diffusion and jump terms are controllable. Recognizing the absence of such algorithms in the literature even for stochastic control problems under this setting, we first introduce an actor-critic RL algorithm to solve the control problem. We then extend it to the game setting by combining it with fictitious play \cite{brown1949some,brown1951iterative} and parallel computing techniques \cite{zhang2020deep, babaeizadeh2016reinforcement}.

\smallskip
\noindent\textbf{Related literature.} In recent years, multi-agent investment games have attracted significant attention. Since the seminal paper  \cite{lacker2019mean}, which provided an explicit Nash equilibrium in a log-normal market, numerous extensions have been explored. These include forward utility \cite{anthropelos2022competition,dos2022forward}, stochastic volatility \cite{kraft2020dynamic}, investment with consumption \cite{lacker2020many,dos2022forward}, using a probabilistic approach \cite{fu2020mean,tanana2023relative}, general incomplete markets with random coefficients \cite{hu2022n}, non-exponential discounting utility \cite{liang2023time}, graphon games \cite{tangpi2024optimal}, financial markets with delay dynamics \cite{balkin2023stochastic}, and arbitrage-free semimartingale financial markets \cite{bauerle2023nash}, to list a few. Among them, the most closely related works are \cite{bo2022approximating}, which considers consumption in jump-diffusion markets; and \cite{bo2024mean}, which models jump risk using a mutually exciting Hawkes process. In comparison to \cite{bo2022approximating,bo2024mean}, our stylized model, to the best of our knowledge, is the first one to provide a unique semi-explicit constant Nash equilibrium in jump-diffusion markets at the $n$-agent level.  

With the rapid development of graphic processing units (GPU), deep learning-based algorithms have advanced significantly, demonstrating remarkable ability to represent and approximate high-dimensional functions across various domains. In the direction of solving high-dimensional continuous-time stochastic control problems when the system is subject to Brownian noise, these include direct parameterization methods \cite{han2016deep,han2020rnn}, partial differential equation (PDE)-based methods \cite{han2017deep,sirignano2018dgm,al2019applications,raissi2019physics}, and backward stochastic differential equations (BSDE)-based methods \cite{han2017deep,HaJeE:18,hure2020deep,ji2020three,beck2019machine}. For finite-player games and mean-field games, a non-exhaustive list includes \cite{Hu2:19,han2020deep,xuan2020ams, carmona2019convergence2,lauriere2024deep}. 
When combined with reinforcement learning ideas, researchers have explored policy gradient methods such as those in \cite{munos2006policy,zhou2021actor} and convergence analysis in  \cite{zhou2023policy,zhou2024solving}.  For specific applications, \cite{wang2020continuousmeanvar,wang2020reinforcement} address mean-variance portfolio problems and linear-quadratic problems, while \cite{jia2021policy,jia2021policy2} offer insights into policy-based and value function-based RL methods for generic continuous-time and space problems. These methods have been extended in various directions, including variance reduction techniques~\cite{kobeissi2022variance}, risk-aware problems~\cite{jaimungal2022robustrl} and mean-field games~\cite{angiuli2023deep,lauriere2022learning}. For a survey of recent developments in this field we refer readers to \cite{hu2023recent}.

While the literature on machine learning solvers for problems under Brownian noise is quite rich, problems involving jumps have received only limited attention. Several recent attempts \cite{gnoatto2022deep,castro2022deep,boussange2023deep,frey2022deep,frey2022convergence,alasseur2024deep} have been made to solve the corresponding partial integro-differential equations (PIDEs) or BSDEs, as extensions of methods like Deep BSDE \cite{han2017deep}, DBDP \cite{hure2020deep}, and deep splitting methods \cite{beck2021deep}. However, the PIDEs or BSDEs treated correspond to linear or semilinear problems (problems with control in the drift term only), and therefore are not tailed to fully nonlinear control problems with controlled diffusion and jump terms, let alone the game setting, or consideration of RL ideas.  In contrast to the aforementioned literature on PIDEs or BSDEs, this paper focuses on stochastic control and games involving control in both the diffusion and jump terms, which correspond to solving fully nonlinear PIDEs. Instead of addressing the PIDE directly or indirectly through the corresponding BSDE using the Feynman-Kac formula, we approach the control problem using deep reinforcement learning (RL) techniques. Specifically, we approximate both the value and control functions using deep neural networks and develop an actor-critic method consisting of two steps: policy evaluation, which updates the value function associated with a given control, and policy improvement, which finds the control that maximizes this value function. This alternating scheme proves to be highly efficient and stable. We then extend this method to the game setting. For further related studies, see the recent works \cite{georgoulis2024deep, gennaro2024delegated, mastrolia2025optimal}.

\smallskip
\noindent\textbf{Our contribution.} In this paper, we focus on the multi-agent optimal investment game in a jump-diffusion market, as well as develop efficient deep RL algorithms for games where the states are modeled by general It\^o-L\'evy processes. Our main contributions are as follows: 
\begin{enumerate}
  \item We analyze an optimal investment game among $n$-agents, each trading between a common bond and an individual stock subject to Poisson jumps, with the goal of maximizing their utility relative to their peers. We develop semi-explicit Nash equilibria under exponential, power, and logarithmic utilities, and provide sufficient conditions for the existence and uniqueness of constant Nash equilibrium. These semi-explicit results not only enhance our understanding of this stylized model, but also serve as benchmarks to deep RL solvers proposed in the second part. \label{item.contribution1}
  
  \item We address the lack of algorithms, even for fully nonlinear control problems. We begin by proposing a framework based on the actor-critic method in deep reinforcement learning to numerically solve general stochastic control problems with jumps, particularly focusing on cases with control in both the diffusion and jump terms. We then conduct a rigorous numerical comparison under different settings in terms of accuracy, running time, and stability.
  \label{item.contribution2}
  
  \item Additionally, we develop numerical methods based on the framework presented in \ref{item.contribution2} to solve multi-agent games, where the states are modeled by general Itô-Lévy processes. The implementation of parallel computing significantly improves computational efficiency, and we provide numerical demonstrations of the robustness of our deep RL method in computing the Nash equilibrium in the multi-agent game presented in the first part of the paper.
\end{enumerate}

The rest of the paper is organized as follows. In Section~\ref{sec.game}, we introduce the $n$-player optimal investment games under relative performance criteria and derive the semi-explicit Nash equilibria when the agents possess CARA or CRRA utilities and trade in the stock market subject to jump risks. We provide conditions for the existence and uniqueness of such equilibria. Section \ref{sec.algorithm} first develops an actor-critic type deep reinforcement learning framework for solving stochastic control problems with jumps. Then combining it with fictitious play, we generalize the framework to find Nash equilibrium in stochastic differential games. In Section \ref{sec.numerical}, we conduct extensive numerical experiments on stochastic control and multi-agent games to demonstrate the efficiency and accuracy of the proposed algorithms. Finally, Section \ref{sec.conclusion} provides conclusions and outlines possible directions for future work.


\section{The multi-agent portfolio game} \label{sec.game}
We consider an $n$-agent game in which they trade in a jump diffusion market with a common investment horizon. The agents are heterogeneous, having CARA or CRRA risk preferences with different parameters, and aim to maximize their expected utility based on the relative performance of their peers. This section will focus on deriving semi-explicit Nash equilibria under various utility functions.

In this $n$-agent game, all agents, indexed by $i \in \mathcal{I}:= \{1, 2, \cdots, n\}$, trade two assets, a riskless bond and a stock $S_t^i$. Without loss of generality, the riskless bond is assumed to have zero interest rate. For tractability purposes, we consider the stock $S_t^i$ driven by Possion jumps:
\begin{equation} \label{eq.stock}
  \dif S_t^i = S_{t-}^i (\mu_i\dif t+\nu_i\dif W_t^i+\sigma_i\dif B_t+\alpha_i\dif M_t^i+\beta_i\dif M_t^0), \quad i \in \mathcal{I} :=\{1, 2, \cdots, n\},
\end{equation}
with the constant $\mu_i>0$, $\nu_i\geq0$, $\sigma_i\geq0$, $\nu_i^2+\sigma_i^2>0$, $\alpha_i\in\R$, $\beta_i\in\R$. Here, the Brownian motions $B_t, W_t^1, \cdots, W_t^n$ are independent and defined on the filtered probability space $(\Omega,\mathcal{F},\mathbb{F} =(\mathcal{F}_t)_{t\geq0},\mathbb{P})$ with $\mathbb{F}$ being the natural filtration; $N_t^0, N_t^1, \cdots, N_t^n$ are adapted independent Poisson process with intensities $\lambda_0, \lambda_1, \cdots, \lambda_n$; and $M_t^i=N_t^i-\lambda_i t$, $i \in \mathcal{I}\cup\{0\}$ are the corresponding compensated Poisson processes. The terms $B_t$ and $M_t^0$ appearing in all stocks $\{S_t^i\}_{i \in \mathcal{I}}$ reflect the correlation between different stocks and can be viewed as \emph{common noise}.
When $\mu_i=\mu$, $\nu_i=\alpha_i=0$, $\sigma_i=\sigma$, $\beta_i=\beta$, the game degenerates into the scenario where $n$ agents are investing in the same stock.

Each agent has her own utility function $U_i$, which depends on both her terminal wealth and the average of her peers'. In the sequel, we consider two commonly used types of utilities: CARA (exponential) and CRRA (power and logarithmic), and present the derivations in Sections~\ref{sec.exp}--\ref{sec.log}, respectively. 

\subsection{CARA: Exponential utility} \label{sec.exp}
For the $i$-th agent, let $\{\pi_t^i\}_{0\leq t\leq T}$  be the dollar amount of wealth she puts into her own stock $S_t^i$. Assume the portfolio is self-financing, then her wealth process $X_t^i\in\R$ satisfies:
\begin{equation} \label{eq.dXt_exp}
  \dif X_t^i = \frac{\pi_t^i}{S_{t-}^i} \dif S_t^i = \pi_t^i(\mu_i\dif t+\nu_i\dif W_t^i+\sigma_i\dif B_t+\alpha_i\dif M_t^i+\beta_i\dif M_t^0), \quad X_0^i = x_0^i, \; i \in \mathcal{I}. 
\end{equation}
For the exponential case, agent $i$'s utility function $U_i$ depends on her own position and the arithmetic average of others':
\begin{equation} \label{eq.utility_exp}
  U_i(x_1,x_2,\cdots,x_n) = -e^{-\frac{1}{\delta_i}(x_i-\theta_i\frac{1}{n}\sum_{k=1}^nx_k)},
\end{equation}
where $\delta_i>0$ denotes the absolute risk tolerance and $\theta_i \in (0,1)$ models the agent's interaction towards her peers. The expected utility for $i$-th agent, following a given set of trading strategies $(\pi^1, \cdots, \pi^n)$, is defined as:
\begin{equation} \label{eq.payoff}
  J_i(t, x_1, \cdots, x_n; \pi^1, \cdots, \pi^n) = \E^{t,x_1,\cdots,x_n} \left[U_i(X_T^1,\cdots,X_T^n)\right],
\end{equation}
and she aims to maximize $J_i$ over all possible $\pi^i$. Note that this is not only influenced by her own strategy but also by others'. Therefore, it is natural to consider the following notion of solutions in multi-agent games.
\begin{definition}[Nash equilibrum] \label{def.Nash}
  A tuple of strategies $(\pi_1^*,\cdots,\pi_n^*)$ defined on [0,T] is called a Nash equilibrium if, under the initial condition $(X_0^1,\cdots,X_0^n)=(x_0,\cdots,x_n):=\overrightarrow{x}$, for any agent $i \in \mathcal{I}$ and any strategy $\pi_i$, we have
  \begin{equation*}
    J_i(0,\overrightarrow{x};\pi_1^*,\cdots,\pi_{i-1}^*,\pi_i^*,\pi_{i+1}^*,\cdots,\pi_n^*) \geq J_i(0,\overrightarrow{x};\pi_1^*,\cdots,\pi_{i-1}^*,\pi_i,\pi_{i+1}^*,\cdots,\pi_n^*).
  \end{equation*}
  In particular, a constant Nash equilibrium refers to a Nash equilibrium that is independent of time $t$ and state $\overrightarrow{x}$.
\end{definition}

We next present the main result of this subsection, concerning the constant Nash equilibrium in the case of exponential utility.

\begin{theorem} \label{thm.exp}
  Assume $\mu_i>0$, $\nu_i\geq0$, $\sigma_i\geq0$, $\nu_i^2+\sigma_i^2>0$, $\delta_i>0$ and $\theta_i \in (0,1)$, then any constant Nash equilibrium $(\pi_1^*,\cdots,\pi_n^*)$ satisfies the coupled system:
  \begin{equation} \label{eq.thm_exp}
    \begin{aligned}
      \mu_i + \frac{1}{\delta_i}\sigma_i\theta_i\widehat{\pi^*\sigma} & - \frac{1}{\delta_i}(1-\frac{\theta_i}{n})(\nu_i^2+\sigma_i^2)\pi^*_i - \lambda_0\beta_i - \lambda_i\alpha_i \\ 
      & + \lambda_0\beta_i e^{-\frac{1}{\delta_i}\left((1-\frac{\theta_i}{n})\pi_i^*\beta_i-\theta_i\widehat{\pi^*\beta}\right)} + \lambda_i\alpha_i e^{-\frac{1}{\delta_i}(1-\frac{\theta_i}{n})\pi_i^*\alpha_i} = 0, \quad \forall i \in \mathcal{I},
    \end{aligned}
  \end{equation}
  where
    $\widehat{\pi^*\sigma} := \frac{1}{n} \sum_{k\neq i} \pi^*_k \sigma_k$ and 
    $\widehat{\pi^*\beta} := \frac{1}{n} \sum_{k\neq i} \pi^*_k \beta_k$. If system \eqref{eq.thm_exp} has a unique solution, the constant Nash equilibrium is unique.
\end{theorem}

\begin{proof}
  Consider the wealth process of the $i$-th agent $X_t^i$ and the sum of the remaining agents' $Y_t^i :=\frac{1}{n}\sum_{k\neq i} X_t^k$, assuming that all other agents ($k\neq i$) use constant strategies $\pi_k$:
  \begin{equation}\label{eq.dXtdYt_exp}
  \begin{aligned}
    \dif X_t^i &= \pi_i(\mu_i\dif t+\nu_i\dif W_t^i+\sigma_i\dif B_t+\alpha_i\dif M_t^i+\beta_i\dif M_t^0), \\
    \dif Y_t^i &= \widehat{\pi\mu}\dif t + \frac{1}{n}\sum_{k\neq i} \pi_k\nu_k\dif W_t^k + \widehat{\pi\sigma}\dif B_t + \frac{1}{n}\sum_{k\neq i} \pi_k\alpha_k \dif M_t^k + \widehat{\pi\beta}\dif M_t^0,
  \end{aligned}
\end{equation}
  where
    $\widehat{\pi\mu} := \frac{1}{n} \sum_{k\neq i} \pi_k \mu_k$, 
    $\widehat{\pi\sigma} := \frac{1}{n} \sum_{k\neq i} \pi_k \sigma_k$, and 
   $ \widehat{\pi\beta} := \frac{1}{n} \sum_{k\neq i} \pi_k \beta_k$.
Meanwhile, the utility function $U_i$ can also be rewritten in terms of $X_T^i$ and $Y_T^i$:
  \begin{equation*}
    U_i(\overrightarrow{X}_T) = -e^{-\frac{1}{\delta_i}\left(X_T^i-\theta_i\frac{1}{n}\sum_{k=1}^n X_T^k\right)} = -e^{-\frac{1}{\delta_i}\left((1-\frac{\theta_i}{n})X_T^i-\theta_i\frac{1}{n}\sum_{k\neq i} X_T^k\right)} := \widetilde{U}_i(X_T^i,Y_T^i),
  \end{equation*}
  where 
    $\widetilde{U}_i(x,y) = -e^{-\frac{1}{\delta_i}((1-\frac{\theta_i}{n})x-\theta_i y)}$.
 
For agent $i$'s utility maximization problem, define the value function $v^i$ by
  \begin{equation*}
    v^i(t,x,y) = \sup_{\pi_i} \E^{t,x,y} \left[\widetilde{U}_i(X_T^i,Y_T^i)\right].
  \end{equation*}
Given the strict convexity of $e^{-x}$ and the linearity of the wealth process $X_t^i$ with respect to the dollar amount of wealth process $\pi_i$, one can deduce that $\E^{t,x,y} \left[\widetilde{U}_i(X_T^i,Y_T^i)\right]$ is strictly concave with respect to $\pi_i$. Therefore, the constant optimal strategy $\pi_i^*$, given others using fixed $\pi_k$, $k \neq i$, is unique. By dynamic programming (\cite[Sections 3.3--3.4]{pham2009continuous}), it satisfies the following HJB equation:
  \begin{equation} \label{eq.HJB_exp}
    \left\{
    \begin{aligned}
        & \frac{\partial v^i}{\partial t}(t,x,y) + \sup_{\pi_i} \mathcal{L}^i v^i(t,x,y) = 0, \ t \in [0, T), \\
        & v^i(T, x, y) = \widetilde{U}_i(x,y),
    \end{aligned}
    \right.
  \end{equation}
  with $\mathcal{L}^i$ being the generator of $(X_t^i,Y_t^i)$ defined in \eqref{eq.dXtdYt_exp}:
  \begin{equation}
    \begin{aligned}
      \mathcal{L}^i  v^i& (t,x,y) \\
      = & \left[\pi_i\mu_iv_x^i + \widehat{\pi\mu} v_y^i + \frac{1}{2}(\nu_i^2+\sigma_i^2)\pi_i^2v_{xx}^i + \frac{1}{2}\left(\widehat{\pi\sigma}^2+\frac{1}{n}\widehat{\pi^2\nu^2}\right)v_{yy}^i + \pi_i\sigma_i\widehat{\pi\sigma}v_{xy}^i\right]_{(t,x,y)} \\ 
      & + \lambda_i\left(v^i(t,x+\pi_i\alpha_i,y)-v^i(t,x,y)-\pi_i\alpha_i v^i_x(t,x,y)\right) \\ 
      & + \sum_{k\neq i} \lambda_k\left(v^i(t,x,y+\frac{\pi_k\alpha_k}{n})-v^i(t,x,y)-\frac{\pi_k\alpha_k}{n}v^i_y(t,x,y)\right) \\ 
      & + \lambda_0\left(v^i(t,x+\pi_i\beta_i,y+\widehat{\pi\beta})-v^i(t,x,y)-\pi_i\beta_i v^i_x(t,x,y)-\widehat{\pi\beta}v^i_y(t,x,y)\right),
    \end{aligned}
  \end{equation}
  where
    $\widehat{\pi^2\nu^2} := \frac{1}{n} \sum_{k\neq i}\pi_k^2\nu_k^2$.

To solve \eqref{eq.HJB_exp}, one makes the ansatz $v^i(t,x,y) = f(t)\cdot\widetilde{U}_i(x,y)$ and it gives  
\begin{equation}
    f'(t)+kf(t)=0, \quad f(T)=1, \quad \text{with}
  \end{equation}
\begin{equation}
    \begin{aligned}
      k = & -\frac{1}{\delta_i}(1-\frac{\theta_i}{n})\pi_i^*\mu_i+\widehat{\pi\mu}\frac{\theta_i}{\delta_i} + \frac{1}{2}(\nu_i^2+\sigma_i^2)(\pi_i^*)^2\left(\frac{1-\frac{\theta_i}{n}}{\delta_i}\right)^2 + \frac{1}{2}\left(\widehat{\pi\sigma}^2+\frac{1}{n}\widehat{\pi^2\nu^2}\right)\frac{\theta_i^2}{\delta_i^2} \\
      & -\pi_i^*\sigma_i\widehat{\pi\sigma}\frac{1}{\delta_i^2}(1-\frac{\theta_i}{n})\theta_i + \lambda_i\left(e^{-\frac{1}{\delta_i}(1-\frac{\theta_i}{n})\pi_i^*\alpha_i}-1+\pi_i^*\alpha_i\frac{1}{\delta}(1-\frac{\theta}{n})\right) \\ 
      & + \sum_{k\neq i}\lambda_k\left(e^{\frac{\theta_i}{\delta_i}\frac{\pi_k\alpha_k}{n}}-1-\frac{\pi_k\alpha_k}{n}\frac{\theta_i}{\delta_i}\right) \\ 
      & + \lambda_0\left(e^{-\frac{1}{\delta_i}(1-\frac{\theta_i}{n})\pi_i^*\beta_i}\cdot e^{\frac{\theta_i}{\delta_i}\widehat{\pi\beta}}-1+\pi^*_i\beta_i\frac{1}{\delta_i}(1-\frac{\theta_i}{n})-\widehat{\pi\beta}\frac{\theta_i}{\delta_i}\right).
    \end{aligned}
  \end{equation}
Thus $f(t)=e^{k(T-t)}$. Plugging it back to \eqref{eq.HJB_exp} and using the first-order condition yield
  \begin{equation} \label{eq.system_exp}
    \begin{aligned}
      \mu_i + \frac{1}{\delta_i}\sigma_i\theta_i\widehat{\pi\sigma} & - \frac{1}{\delta_i}(1-\frac{\theta_i}{n})(\nu_i^2+\sigma_i^2)\pi^*_i \\ 
      & + \lambda_0\beta_i \left( e^{-\frac{1}{\delta_i}\left((1-\frac{\theta_i}{n})\pi_i^*\beta_i-\theta_i\widehat{\pi\beta}\right)} -1 \right) + \lambda_i\alpha_i \left( e^{-\frac{1}{\delta_i}(1-\frac{\theta_i}{n})\pi_i^*\alpha_i} -1\right) = 0.
    \end{aligned}
  \end{equation}
The $\pi^\ast_i$ obtained is indeed the unique maximizer by checking the second-order condition in \eqref{eq.HJB_exp}:
  \begin{multline*}
    \frac{\dif^2}{\dif \pi_i^2} \left[\mathcal{L}^i v^i(t,x,y)\right] = \\ 
     \left((\nu_i^2+\sigma_i^2)+\lambda_0\beta_i^2 e^{\frac{1}{\delta_i}(1-\frac{\theta_i}{n})\pi_i\beta_i} e^{\frac{\theta_i}{\delta_i}\widehat{\pi\beta}} + \lambda_i\alpha_i^2 e^{\frac{1}{\delta_i}(1-\frac{\theta_i}{n})\pi_i\alpha_i}\right)\left(\frac{1-\frac{\theta_i}{n}}{\delta_i}\right)^2 v^i(t,x,y) < 0,
  \end{multline*}
  due to $v^i(t,x,y)<0$. 

  It is clear from \eqref{eq.system_exp} that $\pi_i^*$ is a constant. We have now found a constant solution to the 
  HJB equation~\eqref{eq.HJB_exp} for agent $i$. Additionally, according to the Verification Theorem (\cite[Section 3.5]{pham2009continuous}), this solution represents the $i$-th agent's value function and provides a constant optimal strategy. Given the strict concavity of the optimization problem, we can conclude that there is no other constant solution to the $i$-th agent's HJB equation. In other words, given the constant strategies $\pi_k$ of other agents, where $k \neq i$,  $\pi_i^\ast$ is unique and satisfies \eqref{eq.system_exp}.

  Symmetrically, all agents $j \in \mathcal{I}$ follow an analogous to \eqref{eq.system_exp} strategy, meaning that one can replace  $\{\pi_k\}_{k\neq i}$ with $\{\pi_k^*\}_{k\neq i}$ in \eqref{eq.system_exp}. This leads us to the condition \eqref{eq.thm_exp}.  
\end{proof}

\begin{remark}
When $\alpha_i=\beta_i=0$, the stock prices no longer exhibit jumps, and in this case, the constant Nash equilibrium has an explicit form; see \cite{lacker2019mean} for details. As discussed in \cite{lacker2019mean}, it remains unclear theoretically whether this multi-agent game has non-constant Nash equilibria, i.e., strategies depending on the time and/or state variables. However, in numerical experiments presented in Section~\ref{sec.numerical.games}, we search among strategies parameterized by neural networks taking inputs of the time $t$ and the state variables $(x_1,\cdots,x_n)$, and find that they ultimately converge to the constant Nash equilibrium given by Theorem \ref{thm.exp}.
\end{remark}

\begin{corollary}
If the parameters of the problem are such that, for any $(\pi_1, \cdots, \pi_n) \in \mathbb{R}^n$, the Jacobian matrix of the left-hand side of \eqref{eq.thm_exp} with respect to the strategies is invertible, then \eqref{eq.thm_exp} has a unique solution in $\mathbb{R}^n$, and the constant Nash equilibrium is unique in $\mathbb{R}^n$.
\end{corollary}
\begin{proof}
We note that finding the constant Nash equilibrium is equivalent to solving the nonlinear system of equations \eqref{eq.thm_exp}. The existence of a solution follows directly from the multivariate implicit function theorem.
\end{proof}

We next provide an explicit sufficient condition for the existence and uniqueness of the solution to system \eqref{eq.thm_exp} in a bounded region. The proof is provided in Appendix \ref{sec.appendix_proof_exp}.

\begin{corollary} \label{cor.unique_exp}
Assuming $\mu_i>0$, $\nu_i\geq0$, $\sigma_i\geq0$, $\nu_i^2+\sigma_i^2>0$, $\delta_i>0$, $\theta_i \in (0,1)$, and the following two conditions hold,
  \begin{gather}                \overline{\nu}^2+\overline{\sigma}^2+\lambda_0\beta_0^2 e^{\frac{2C\beta_0}{\underline{\delta}}} + \overline{\lambda}\alpha_0^2 e^{\frac{C\alpha_0}{\underline{\delta}}} < \underline{\delta}, \label{eq.unique_exp_cond1} \\ 
\underline{\nu}^2-\overline{\sigma}(\overline{\sigma}-\underline{\sigma})-\lambda_0\beta_0^2 e^{\frac{2C\beta_0}{\underline{\delta}}} > 0, \label{eq.unique_exp_cond2}
\end{gather}
where $\underline{\delta}=\min_{1\leq i\leq n} \delta_i$, $\overline{\lambda}=\max_{1\leq i\leq n} \lambda_i$, $\alpha_0=\max_{1\leq i\leq n}|\alpha_i|$, $\beta_0=\max_{1\leq i\leq n}|\beta_i|$, $\overline{\nu}=\max_{1\leq i\leq n}\nu_i$, $\underline{\nu}=\min_{1\leq i\leq n}\nu_i$, $\overline{\sigma}=\max_{1\leq i\leq n}\sigma_i$, $\underline{\sigma}=\min_{1\leq i\leq n}\sigma_i$,
  then the nonlinear system \eqref{eq.thm_exp} has a unique solution on $\{(\pi_1,\cdots,\pi_n): |\pi_i|\leq C, \forall i=1,2,\cdots,n\}$.
\end{corollary}

\subsection{CRRA: Power utility} \label{sec.power}
This section discusses a multi-agent game similar to the one in subsection \ref{sec.exp}, but with CRRA risk preference.  
Here, denote by $\pi_t^i$ the proportion of wealth invested in the $i$-th stock. Consequently, her wealth process satisfies
\begin{equation} \label{eq.dXt_power}
  \dif X_t^i = \frac{\pi_t^i X_{t-}^i }{S_{t-}^i}\dif S_t^i = \pi_t^i X_{t-}^i(\mu_i\dif t+\nu_i\dif W_t^i+\sigma_i\dif B_t+\alpha_i\dif M_t^i+\beta_i\dif M_t^0),\ X_0^i=x_0^i, \ i \in \mathcal{I}.
\end{equation}
In the power utility case, agent $i$'s utility function $U_i$ takes into account both her position and the geometric average of her peers: 
\begin{equation} \label{eq.utility_power}
  U_i(x_1,x_2,\cdots,x_n) = \frac{1}{p_i}\left(\frac{x_i}{\left(\prod_{k=1}^n x_k\right)^{\frac{\theta_i}{n}}}\right)^{p_i},
\end{equation}
where $p_i\in (0,1)$ is the relative risk tolerance and $\theta_i \in (0,1)$ models the agent’s interaction towards others. The definitions of $J_i$ and  Nash equilibrium stay the same as in equation \eqref{eq.payoff} and Definition \ref{def.Nash}. The following theorem is the main result of this subsection, providing a constant Nash equilibrium under the power utility.
\begin{theorem} \label{thm.power}
  Assume $\mu_i>0$, $\nu_i\geq0$, $\sigma_i\geq0$, $\nu_i^2+\sigma_i^2>0$, $p_i \in (0,1)$ and $\theta_i \in (0, 1)$, any constant Nash equilibrium $(\pi_1^*,\cdots,\pi_n^*)$ satisfies the coupled system:
  \begin{equation} \label{eq.thm_power}
    \begin{aligned}
      \mu_i & + (\nu_i^2+\sigma_i^2)\left(p_i(1-\frac{\theta_i}{n})-1\right)\pi_i^*-p_i\theta_i\sigma_i\widehat{\pi^*\sigma} - \lambda_i\alpha_i-\lambda_0\beta_i \\ 
      & + \lambda_i\alpha_i(1+\pi_i^*\alpha_i)^{p_i(1-\frac{\theta_i}{n})-1} + \lambda_0\beta_i\frac{(1+\pi_i^*\beta_i)^{p_i(1-\frac{\theta_i}{n})-1}}{\widetilde{1+\pi^*\beta}^{p_i\theta_i}} = 0, \quad \forall i \in \mathcal{I},
    \end{aligned}
  \end{equation} 
where $\widehat{\pi^*\sigma} := \frac{1}{n} \sum_{k\neq i} \pi^*_k \sigma_k, \quad \widetilde{1+\pi^*\beta} := \left(\prod_{k\neq i}(1+\pi_k^*\beta_k)\right)^{\frac{1}{n}}$.
If system \eqref{eq.thm_power} has a unique solution, the constant Nash equilibrium is unique.
\end{theorem}
\begin{proof}
As the utility depends on the geometric average, let us consider the wealth process of the $i$-th agent $X_t^i$ and the product of the remaining agents $Y_t^i:=(\prod_{k\neq i}X_t^k)^{\frac{1}{n}}$, assuming that all other agents ($k\neq i$) use constant strategies $\pi_k$. Applying It\^{o} formula to $\log X_t^i$ gives 
  \begin{equation*}
    \begin{aligned}
      \dif\left(\log X_t^i\right) = & \left(\pi_i\mu_i-\frac{1}{2}\pi_i^2(\nu_i^2+\sigma_i^2)+\lambda_i\left(\log(1+\pi_i\alpha_i)-\pi_i\alpha_i\right)+\lambda_0\left(\log(1+\pi_i\beta_i)-\pi_i\beta_i\right)\right)\dif t \\
      & + \pi_i\nu_i\dif W_t^i + \pi_i\sigma_i\dif B_t + \log(1+\pi_i\alpha_i)\dif M_t^i + \log(1+\pi_i\beta_i)\dif M_t^0.
    \end{aligned}
  \end{equation*}
Summing over $k \neq i$ brings
  \begin{equation*}
    \begin{aligned}
      \dif\left(\log Y_t^i\right) = \ & \frac{1}{n}\sum_{k\neq i}\dif\left(\log X_t^i\right) = \left(\widehat{\pi\mu}-\frac{1}{2}(\widehat{\pi^2\nu^2}+\widehat{\pi^2\sigma^2})+\frac{1}{n}\sum_{k\neq i}\lambda_k(\log(1+\pi_k\alpha_k)-\pi_k\alpha_k)\right. \\ 
      & +\left.\lambda_0\frac{1}{n}\sum_{k\neq i}\left(\log(1+\pi_k\beta_k)-\pi_k\beta_k\right)\right)\dif t + \frac{1}{n}\sum_{k\neq i}\pi_k\nu_k\dif W_t^k + \widehat{\pi\sigma}\dif B_t \\ 
      & + \frac{1}{n}\sum_{k\neq i}\log(1+\pi_k\alpha_k)\dif M_t^k + \log\left(\widetilde{1+\pi\beta}\right) \dif M_t^0,
    \end{aligned}
  \end{equation*}
  where $\widehat{\pi\mu} := \frac{1}{n}\sum_{k\neq i}\pi_k\mu_k$,  $\widehat{\pi\sigma} := \frac{1}{n}\sum_{k\neq i} \pi_k\sigma_k$, $\widehat{\pi^2\sigma^2} := \frac{1}{n}\sum_{k\neq i} \pi_k^2\sigma_k^2$, $\widehat{\pi^2\nu^2} := \frac{1}{n}\sum_{k\neq i} \pi_k^2\nu_k^2$,  $\widetilde{1+\pi\beta} := \left(\prod_{k\neq i}(1+\pi_k\beta_k)\right)^{\frac{1}{n}}$.
  Another application of It\^{o} formula produces
\begin{multline} \label{eq.dYt_power}
        \frac{\dif Y_t^i}{Y_{t-}^i} = \eta \dif t + \frac{1}{n}\sum_{k\neq i}\pi_k\nu_k\dif W_t^k + \widehat{\pi\sigma}\dif B_t  \\+ \sum_{k\neq i}\left((1+\pi_k\alpha_k)^{\frac{1}{n}}-1\right)\dif M_t^k + \left(\widetilde{1+\pi\beta}-1\right)\dif M_t^0,
\end{multline}
  where $\eta = \widehat{\pi\mu}-\frac{1}{2}(\widehat{\pi^2\nu^2}+\widehat{\pi^2\sigma^2})+\frac{1}{2}\left(\frac{1}{n}\widehat{\pi^2\nu^2}+\widehat{\pi\sigma}^2\right) +\sum_{k\neq i}\lambda_k\left((1+\pi_k\alpha_k)^{\frac{1}{n}}-1-\frac{\pi_k\alpha_k}{n}\right) + \lambda_0\left(\widetilde{1+\pi\beta}-1-\widehat{\pi\beta}\right)$.
  On the other hand, the utility function $U_i$ can be expressed in terms of $(X_T^i, Y_T^i)$
  \begin{equation}
    U_i(\overrightarrow{X_T}) = \frac{1}{p_i}\Bigg(\frac{X_T^i}{\big(\prod_{k=1}^n X_T^k\big)^{\frac{\theta_i}{n}}}\Bigg)^{p_i} = \frac{1}{p_i}\Bigg(\frac{(X_T^i)^{1-\frac{\theta_i}{n}}}{\big(\prod_{k\neq i} X_T^k\big)^{\frac{\theta_i}{n}}}\Bigg)^{p_i} = \widetilde{U}_i(X_T^i,Y_T^i),
  \end{equation} 
  where $\widetilde{U}_i(x,y) = \frac{1}{p_i}\left(\frac{x^{1-\frac{\theta_i}{n}}}{y^{\theta_i}}\right)^{p_i}.$
  
  In view of agent $i$’s portfolio optimization problem, her value function $v^i$ defined as
  \begin{equation} \label{eq.value_fun_power}
    v^i(t,x,y) = \sup_{\pi_i} \E^{t,x,y} \left[\widetilde{U}_i(X_T^i,Y_T^i)\right],
  \end{equation}
  satisfies the following HJB equation:
  \begin{equation} \label{eq.HJB_power}
    \left\{
    \begin{aligned}
        & \frac{\partial v^i}{\partial t}(t,x,y) + \sup_{\pi_i} \mathcal{L}^i v^i(t,x,y) = 0, \ t \in [0, T), \\
        & v^i(T, x, y) = \widetilde{U}_i(x,y).
    \end{aligned}
    \right.
  \end{equation}
 Here the generator $\mathcal{L}^i$ of process $(X_t^i,Y_t^i)$ given in \eqref{eq.dXt_power} and \eqref{eq.dYt_power} reads
  \begin{equation} \label{eq.generator_power}
{\small    \begin{aligned}
      & \mathcal{L}^i v^i(t,x,y) = \\ 
      & \left[\pi_i\mu_ixv_x^i + \eta yv_y^i + \frac{1}{2}(\nu_i^2+\sigma_i^2)\pi_i^2x^2v_{xx}^i + \frac{1}{2}\left(\widehat{\pi\sigma}^2+\frac{1}{n}\widehat{\pi^2\nu^2}\right)y^2v_{yy}^i + \pi_i\sigma_i\widehat{\pi\sigma}xyv_{xy}^i\right]_{(t,x,y)} \\ 
      & + \lambda_i\left(v^i(t,x+\pi_i\alpha_ix,y)-v^i(t,x,y)-\pi_i\alpha_i xv^i_x(t,x,y)\right) \\ 
      & + \sum_{k\neq i} \lambda_k\left[v^i(t,x,(1+\pi_k\alpha_k)^{\frac{1}{n}}y)-v^i(t,x,y)-\left((1+\pi_k\alpha_k)^{\frac{1}{n}}-1\right)yv^i_y(t,x,y)\right] \\ 
      & + \lambda_0\left[v^i(t,x+\pi_i\beta_ix,\widetilde{1+\pi\beta}y)-v^i(t,x,y)-\pi_i\beta_i xv^i_x(t,x,y)-\left(\widetilde{1+\pi\beta}-1\right)yv^i_y(t,x,y)\right].
    \end{aligned}}
  \end{equation}
From the fact that $\E^{t,x,y} \left[\widetilde{U}_i(X_T^i,Y_T^i)\right]$ is strictly concave with respect to the strategy process (absolute dollar amount being invested into $S^i$), we can deduce that the optimal strategy $\pi_i^*$ in equation \eqref{eq.value_fun_power} is unique. Making the same ansatz $v^i(t,x,y) = f(t)\cdot\widetilde{U}_i(x,y)$, we deduce 
  \begin{equation}
    f'(t)+kf(t)=0, \quad f(T)=1, \quad \text{with}
  \end{equation}
  \begin{equation}
    \begin{aligned}
      k = &\ p_i(1-\frac{\theta_i}{n})\pi_i^*\mu_i - p_i\theta_i\eta^* + \frac{1}{2}(\nu_i^2+\sigma_i^2)(\pi_i^*)^2p_i(1-\frac{\theta_i}{n})\left(p_i(1-\frac{\theta_i}{n})-1\right) \\  
      & + \frac{1}{2}\left(\widehat{\pi\sigma}^2+\frac{1}{n}\widehat{\pi^2\nu^2}\right)p_i\theta_i(p_i\theta_i+1)
      -\pi_i^*\sigma_i\widehat{\pi\sigma}p_i^2\theta_i(1-\frac{\theta_i}{n}) \\ 
      & + \lambda_i\left((1+\pi_i^*\alpha_i)^{p_i(1-\frac{\theta_i}{n})}-1-p_i(1-\frac{\theta_i}{n})\pi_i^*\alpha_i\right) \\
      & + \sum_{k\neq i}\lambda_k\left((1+\pi_k\alpha_k)^{-\frac{p_i\theta_i}{n}}-1+((1+\pi_k\alpha_k)^{\frac{1}{n}}-1)p_i\theta_i\right) \\ 
      & + \lambda_0\left(\frac{(1+\pi_i^*\beta_i)^{p_i(1-\frac{\theta_i}{n})}}{\widetilde{1+\pi\beta}^{p_i\theta_i}}-1-p_i(1-\frac{\theta_i}{n})\pi_i^*\beta_i+(\widetilde{1+\pi\beta}-1)p_i\theta_i\right).
    \end{aligned}
  \end{equation}
  Then $f$ is given by $f(t)=e^{k(T-t)}$, and the equilibrium strategy $\pi_t^\ast$ is obtained via the first-order condition in \eqref{eq.HJB_power}:
  \begin{equation} \label{eq.system_power}
    \begin{aligned}
      \mu_i & + (\nu_i^2+\sigma_i^2)\left(p_i(1-\frac{\theta_i}{n})-1\right)\pi_i^*-p_i\theta_i\sigma_i\widehat{\pi\sigma} \\ 
      & + \lambda_i\alpha_i\left( (1+\pi_i^*\alpha_i)^{p_i(1-\frac{\theta_i}{n})-1} -1\right) + \lambda_0\beta_i\left(\frac{(1+\pi_i^*\beta_i)^{p_i(1-\frac{\theta_i}{n})-1}}{\widetilde{1+\pi\beta}^{p_i\theta_i}} -1\right) = 0.
    \end{aligned}
  \end{equation}
  Its maximality is ensured by $v^i_{xx}(t,x,y)<0$ and the computation
  \begin{multline}\label{eq.second_derivate_power}
       \frac{\dif^2}{\dif \pi_i^2} \left[\mathcal{L}^i v^i(t,x,y)\right] = (\nu_i^2+\sigma_i^2)x^2v^i_{xx}(t,x,y) \\
       + \lambda_i\alpha_i^2x^2v^i_{xx}(t,(1+\pi_i\alpha_i)x,y) + \lambda_0\beta_i^2x^2v^i_{xx}(t,(1+\pi_i\beta_i)x,\widetilde{1+\pi\beta}y).
  \end{multline} 
Following the same reasoning as in the proof of Theorem~\ref{thm.exp}, we arrive at the condition~\eqref{eq.thm_power}.
Note that, compared to exponential utility, the additional conditions $\alpha_0<1$, $\beta_0<1$ and $0<C\leq1$ are introduced as the
expression \eqref{eq.thm_power} involves terms like $(1+\pi_i\alpha_i)^{p_i(1-\frac{\theta_i}{n})-1}$ and $(1+\pi_i\beta_i)^{p_i(1-\frac{\theta_i}{n})-1}$, and we aim to ensure that the base is always positive. 
\end{proof}

\begin{corollary}
If the parameters of the problem are such that, for any $(\pi_1, \dots, \pi_n) \in \mathbb{R}^n$, the Jacobian matrix of the left-hand side of \eqref{eq.thm_power} with respect to the strategies is invertible, then \eqref{eq.thm_power} admits a unique solution, and the constant Nash equilibrium is unique in $\mathbb{R}^n$.
\end{corollary}

Similarly, we provide a sufficient condition for the existence and uniqueness of the solution to the nonlinear system \eqref{eq.thm_power} on a bounded region:
\begin{corollary} \label{cor.unique_power}
    Assume $\mu_i>0$, $\nu_i\geq0$, $\sigma_i\geq0$, $\nu_i^2+\sigma_i^2>0$, $p_i \in (0,1)$ and $\theta_i \in (0, 1)$, and suppose the following three conditions hold, 
  \begin{gather}
    (1-0.5\underline{p})\left(\overline{\nu}^2+\overline{\sigma}^2+\overline{\lambda}\frac{\alpha_0^2}{(1-C\alpha_0)^{2-0.5\underline{p}}} + \lambda_0\frac{\beta_0^2}{(1-C\beta_0)^2}\right) < 1, \label{eq.unique_power_cond1} \\ 
    \overline{p}(1+\overline{\theta}) < 1, \label{eq.unique_power_cond2} \\ 
    \underline{\nu}^2-\overline{\sigma}(\overline{\sigma}-\underline{\sigma})-\lambda_0\beta_0^2\frac{1}{(1-C\beta_0)^2}\frac{1+C\beta_0}{1-C\beta_0} > 0, \label{eq.unique_power_cond3}
  \end{gather}
  where $\underline{p}=\min_{1\leq i\leq n} p_i$, $\overline{p}=\max_{1\leq i\leq n} p_i$, $\overline{\lambda}=\max_{1\leq i\leq n}\lambda_i$, $\overline{\nu}=\max_{1\leq i\leq n}\nu_i$, $\underline{\nu}=\min_{1\leq i\leq n}\nu_i$, $\overline{\sigma}=\max_{1\leq i\leq n}\sigma_i$, $\underline{\sigma}=\min_{1\leq i\leq n}\sigma_i$, $\alpha_0=\max_{1\leq i\leq n}|\alpha_i| < 1$, $\beta_0=\max_{1\leq i\leq n}|\beta_i| < 1$, $0<C\leq1$. Then the nonlinear system \eqref{eq.thm_power} has a unique solution in $\{(\pi_1,\cdots,\pi_n): |\pi_i|\leq C, \forall i\in\mathcal{I}\}$.
\end{corollary}

\subsection{CRRA: Logarithmic utility} \label{sec.log}

The case of logarithmic utility, although it is of the CRRA type, needs to be treated separately.
The model remains the same as in subsection \ref{sec.power}: the stock prices $S_t^i$ follow \eqref{eq.stock}; $\pi_t^i$ denotes the proportion of wealth $X_t^i$ to purchase the $i$-th stock; and $X_t^i$ follows \eqref{eq.dXt_power}. But the utility function has the following logarithmic form:
\begin{equation} \label{eq.utility_log}
  U_i(x_1,\cdots,x_n) = \log \frac{x_i}{\left(\prod_{k=1}^n x_k\right)^{\frac{\theta_i}{n}}},
\end{equation}
where $\theta_i\in(0,1)$ denotes the agent’s interaction towards others. 

The following result addresses the unique constant Nash equilibrium in the case of logarithmic utility. As seen below, unlike the first two cases, here the control $\pi_i^*$ for the $i$-th agent is independent of others. Consequently, the multi-agent game degenerates into several independent one-dimensional stochastic control problems. 

\begin{theorem} \label{thm.log}
  Assume $\mu_i>0$, $\nu_i\geq0$, $\sigma_i\geq0$, $\nu_i^2+\sigma_i^2>0$ and $\theta_i \in(0,1)$, then the constant Nash equilibrium $(\pi_1^*,\cdots,\pi_n^*)$ satisfies:
  \begin{equation} \label{eq.thm_log}
    \mu_i-(\nu_i^2+\sigma_i^2)\pi_i^*+\lambda_i\alpha_i\left(\frac{1}{1+\pi_i^*\alpha_i}-1\right)+\lambda_0\beta_i\left(\frac{1}{1+\pi_i^*\beta_i}-1\right) = 0, \quad \forall i \in \mathcal{I}.
  \end{equation}
Moreover, it has a unique solution $\pi_i^*$ satisfying $1+\pi_i^*\alpha_i>0$ and $1+\pi_i^*\beta_i>0$.
\end{theorem}
\begin{proof}
  Assuming that all other agents ($k\neq i$) use constant strategies $\pi_k$. The utility function $U_i$ can be written as 
  \begin{equation}
    U_i(X_T^1,\cdots,X_T^n) = \log \frac{X_T^i}{\left(\prod_{k=1}^n X_T^k\right)^{\frac{\theta_i}{n}}} = \log \frac{(X_T^i)^{1-\frac{\theta_i}{n}}}{(Y_T^i)^{\theta_i}} = \widetilde{U}_i(X_T^i,Y_T^i),
  \end{equation}
  where $Y_t^i:=(\prod_{k\neq i}X_t^k)^{\frac{1}{n}}$ is given in \eqref{eq.dYt_power}, and $\widetilde{U}_i(x,y) = \log(x^{1-\frac{\theta_i}{n}}y^{-\theta_i})$.
  Consequently, the value function $v^i$
  satisfies the HJB equation \eqref{eq.HJB_power}, with the terminal condition replaced by $v^i(T,x,y)=\log(x^{1-\frac{\theta_i}{n}}y^{-\theta_i})$. Similar to the reasoning in the proof of Theorem~\ref{thm.power}, given the strict concavity of the $\log$ function, this optimization problem has a unique maximizer.
  
  To solve the HJB equation, one makes the ansatz $v^i(t,x,y) = f(t) + \widetilde{U}_i(x,y)$ and it gives
  \begin{equation}
    f'(t)+k=0, \quad f(T)=0,\quad\text{with}
  \end{equation}
  \begin{equation} \label{eq.log_k}
    \begin{aligned}
      k = &\ (1-\frac{\theta_i}{n})\pi_i^*\mu_i - \frac{1}{2}(\nu_i^2+\sigma_i^2)(\pi_i^*)^2(1-\frac{\theta_i}{n})-\theta_i\eta^* + \frac{1}{2}\left(\frac{1}{n}\widehat{\pi^2\nu^2}+\widehat{\pi\sigma}^2\right)\theta_i \\ 
      & +\lambda_i\left(\log(1+\pi_i^*\alpha_i)-\pi_i^*\alpha_i\right)(1-\frac{\theta_i}{n}) \\
      & + \sum_{k\neq i} \lambda_k\left(-\frac{\theta_i}{n}\log(1+\pi_k^*\alpha_k)+\theta_i\left((1+\pi_k^*\alpha_k)^{\frac{1}{n}}-1\right)\right) \\ 
      & + \lambda_0\left(\log\frac{(1+\pi_i^*\beta_i)^{1-\frac{\theta_i}{n}}}{\widetilde{1+\pi\beta}^{\theta_i}}-\pi_i^*\beta_i(1-\frac{\theta_i}{n})+\theta_i(\widetilde{1+\pi\beta}-1)\right).
    \end{aligned}
  \end{equation}
  Thus one deduces  $f(t)=k(T-t)$ and agent $i$'s equilibrium strategy $\pi_i^\ast$ satisfies
  \begin{equation}
    \mu_i-(\nu_i^2+\sigma_i^2)\pi_i^*+\lambda_i\alpha_i\left(\frac{1}{1+\pi_i^*\alpha_i}-1\right)+\lambda_0\beta_i\left(\frac{1}{1+\pi_i^*\beta_i}-1\right) = 0.
  \end{equation}
  Now following the same reasoning as in the power utility case, we conclude~\eqref{eq.thm_log}. 
  The rest of the proof is deferred to Appendix \ref{sec.appendix_proof_log}. The uniqueness of system~\eqref{eq.thm_log} gives the uniqueness of constant Nash equilibrium.
  Note that, 
  The additional conditions $1+\pi_i^*\alpha_i>0$ and $1+\pi_i^*\beta_i>0$ are introduced because the expression \eqref{eq.log_k} involves $\log(1+\pi_i^*\alpha_i)$ and $\log(1+\pi_i^*\beta_i)^{1-\frac{\theta_i}{n}}$. 
\end{proof}

\section{Deep reinforcement learning framework} \label{sec.algorithm}

In this section, we will introduce a computational framework for solving stochastic control problems and differential games where the state process is driven by a general It\^o-L\'evy process. The framework is based on actor-critic-type algorithms and can handle cases where the agent controls the diffusion and jump terms, a scenario not addressed in the existing literature.

We start by solving the control problem under the general L\'evy model in  Section~\ref{sec.stochastic_control}. A special case where the jumps in the controlled state are modeled by compensated Poisson processes will be discussed in Section~\ref{sec.poisson}. We then present the algorithm for games in Section~\ref{sec.dlgame}.

\subsection{Stochastic control problems with jumps} \label{sec.stochastic_control}
Consider a generic stochastic control problem under the L\'evy model, where the state process $X_t \in \R^d$ follows a controlled It\^o-L\'evy process:
\begin{equation} \label{eq.dXt}
  \dif X_t = b(X_{t-}, u_{t})\dif t + \sigma(X_{t-}, u_{t})\dif W_t + \int_{\R^d} G(X_{t-},z,u_{t}) \widetilde{N}(\dif t,\dif z).
\end{equation}
Here, $\{W_t\}_{t=0}^T$ is a $d$-dimensional standard Brownian motion on the filtered probability space $(\Omega,\mathcal{F},(\mathcal{F}_t)_{t\geq0},\mathbb{P})$, $N$ is a Poisson random measure with the L\'evy measure $\nu$ satisfying
\[
\nu(\{0\}) = 0, \quad \int_{\R^d} 1 \wedge |z|^2 \nu(\dif z) < \infty, \quad \int_{|z| \geq 1} |z|^2 \nu(\dif z) < \infty,
\]
$\widetilde{N}(\dif t,\dif z) := N(\dif t,\dif z) - \nu(\dif z)\dif t$ is the compensated Poisson random measure, and $\{u_t\}_{t\geq0}\in U\subset\R^{d_c}$ is an $\mathcal{F}_t$ progressively measurable process. 

The agent aims to maximize her expected utility of running and terminal rewards:
\begin{equation} \label{eq.J}
  J^u(t,x) = \E^{t,x} \left[ \int_{t}^{T} f(s,X_s,u_s) \dif s + g(X_T) \right],
\end{equation}
where $\E^{t,x}$ denotes the expectation under the condition $X_t=x$. The coefficients $b(x,u)\in\R^d$,  $\sigma(x,u)\in\R^{d\times d}$, $G(x,z,u)\in\R^d$, the control space $U\subset\R^{d_c}$, and the reward functions $f(t,x,u)$, $g(x) \in \R$ are all known.

To find the optimal strategy $u^\ast$ for problem~\eqref{eq.dXt}--\eqref{eq.J}, one can define the optimal value function
\begin{equation} \label{eq.value_fun}
  v(t,x) = \sup_{u\in U} J^u(t,x).
\end{equation}
By dynamic programming principle, the value function $v(t,x)$ solves the
partial-integro differential equation (PIDE)
\cite{pham2009continuous,oksendal2019stochastic}:
\begin{equation} \label{eq.HJB}
  \left\{
  \begin{aligned}
      & \frac{\partial v}{\partial t}(t,x) + \sup_{u\in U} \left[\mathcal{L}^u v(t,x) + f(t,x,u)\right] = 0, \ t \in [0,T), \\
      & v(T, x) = g(x),
  \end{aligned}
  \right.
\end{equation}
with $\mathcal{L}^u$ being the generator of $X_t$
\begin{equation} \label{eq.generator}
  \begin{aligned}
    \mathcal{L}^u v(t,x) = &\ b(x,u)\cdot\nabla v(t,x) + \frac{1}{2}\mbox{Tr}\left[\sigma(x,u)\sigma^T(x,u)\mbox{H}(v(t,x))\right] \\ 
    & + \int_{\R^d}(v(t,x+G(x,z,u)) - v(t,x) - G(x,z,u)\cdot\nabla v(t,x))\nu(\dif z), \\
  \end{aligned}
\end{equation}
where $\nabla v$ and $\mbox{H}(v)$ denote the gradient and the Hessian matrix of $v(t,x)$ with respect to the spatial variable $x\in\R^d$, and $\mbox{Tr}$ refers to the trace of a matrix. 

The rest of this subsection will present a reinforcement learning framework to solve both the solution $v(t,x)$ to equation \eqref{eq.HJB} and the optimal control $u^*(t,x)$ simultaneously.

\subsubsection{The actor-critic framework}

Reinforcement learning is a machine learning approach that learns the optimal policy through the interaction of an agent with its environment. It combines aspects of both policy-based and value-based methods. The actor refers to the control function $u(t,x)$, corresponding to the current policy; and the critic refers to the value function $J^u(t,x)$, used to evaluate the goodness of the current policy $u$. The actor-critic method consists mainly of two steps: policy evaluation and policy improvement. By implementing them iteratively, one hopes the actor converges to the optimal control $u^*(t,x)$ and the critic to the optimal value function $v(t,x)=\sup_{u\in U} J^u(t,x)$. Note that in high-dimensional settings, $u$ and $J^u$ are usually approximated (parameterized) by deep neural networks, leading the algorithms known as deep reinforcement learning.  

Before delving into these two steps, let's first outline several key elements of the reinforcement learning model in this problem:
\begin{itemize}
  \item \texttt{State.} The state space $S=[0,T]\times\R^d$ encompasses all possible time-space pairs $(t,x)$, and one refers to $X_t$ as the state process. 
  \item \texttt{Transition.} The transition probability 
  \begin{equation}
    P_{s_ns_{n+1}} = \mathbb{P}(S_{n+1}=(t_{n+1},x_{n+1})|S_{n}=(t_n,x_n)) = \mathbb{P}(X_{t_{n+1}}=x_{n+1}|X_{t_n}=x_n)
  \end{equation}
  is given by the Euler scheme of~\eqref{eq.dXt}
  \begin{equation} \label{eq.DeltaXt}
    \begin{aligned}
      X_{t_{n+1}} = &\ X_{t_n} + b(X_{t_n},u(t_n,X_{t_n}))\Delta t + \sigma(X_{t_n},u(t_n,X_{t_n}))\Delta W_n \\
      & + \sum_{i=N_n+1}^{N_{n+1}} G(X_{t_n},z_i,u(t_n,X_{t_n})) - \Delta t\int_{\R^d} G(X_{t_n},z,u(t_n,X_{t_n})) \nu(\dif z), \\
    \end{aligned}
  \end{equation}
  representing the probability of the state being $x_{n+1}$ at time $t_{n+1}$ when it was $x_n$ at time $t_n$, and we equally divide the time $0=t_0<t_1<\cdots<t_L=T$ with intervals of length $\Delta t$. In equation \eqref{eq.DeltaXt}, $\Delta W_n:= W_{t_{n+1}} - W_{t_n}$ samples from the normal distribution $\mathcal{N}(0,\Delta t)$, $N_n$ denotes the number of jumps occurring in the interval $[0, t_n]$ under the Poisson random measure $N(t,\dif z)$, and $z_i$ represents the size of the $i$-th jump after sorting the jumps in $[0,T]$ in the chronological order.
  \item \texttt{Policy and value function.} The policy (control function) $u(t,x)$ is approximated by a neural network $\mathcal{N}_{\pi}$ which is a function of the state $(t,x)$, and we consider deterministic policy in this case. The value function $J^u(t,x)$ is approximated by a neural network $\mathcal{N}_v$ taking inputs $(t,x)$ and it measures the quality of the current policy $u$. 
  \item \texttt{Reward.} According to equation \eqref{eq.J}, the reward from $t_n$ to $t_{n+1}$ in an episode reads 
  \begin{equation}
    R_{t_nt_{n+1}} = \int_{t_n}^{t_{n+1}} f(t,X_t,u_t)\dif t.
  \end{equation}
  However, an alternative reward will be introduced in the subsequent in \eqref{eq.J_tilde}.
\end{itemize}

\medskip
\noindent\textbf{Policy evaluation.} The objective of policy evaluation is to compute the value function $J^u(t,x)$ associated with a given policy $u(t,x)$. We use temporal difference learning for policy evaluation, an incremental learning procedure driven by the error between temporally successive predictions. In other words, it allows for updates as the state transitions from $t_{n}$ to $t_{n+1}$, rather than waiting for the entire interval $[0,T]$. 
Recall that $\mathcal{N}_v$ denotes the neural network approximation of the value function $J^u$, motivated by the identity
\[
J^u(t_n, X_{t_n}) = \mathbb{E}[R_{t_nt_{n+1}} + J^u(t_{n+1}, X_{t_{n+1}}) \vert X_{t_n}],
\]
the classical one-step temporal difference learning updates $\mathcal{N}_v$ by (cf. \cite[Chapter~6]{sutton2018reinforcement}):
\begin{equation} \label{eq.TD_update}
  \mathcal{N}_v(t_n,X_{t_n}) \leftarrow \mathcal{N}_v(t_n,X_{t_n}) + \alpha\left(R_{t_nt_{n+1}} + \mathcal{N}_v(t_{n+1}, X_{t_{n+1}}) - \mathcal{N}_v(t_{n},X_{t_n})\right),
\end{equation}
where $R_{t_nt_{n+1}} + \mathcal{N}_v(t_{n+1}, X_{n+1}) - \mathcal{N}_v(t_{n},X_{n})$ is referred to as the \emph{TD error}. 

Our update rule for $\mathcal{N}_v$ improves \eqref{eq.TD_update}, and
the design of our loss function for the critic aligns with the approach proposed in \cite{lu2023temporal}.
For a given control 
$u(t,x)$, its associated value function $J^u(t,x)$ defined by \eqref{eq.J} satisfies \cite{applebaum2009levy, oksendal2013stochastic, pham2009continuous}:
\begin{equation}
  \frac{\partial J^u}{\partial t}(t,x) + \mathcal{L}^u J^u(t,x) + f(t,x,u) = 0, \quad J^u(T,x)=g(x),
\end{equation}
with $\mathcal{L}^u$ being defined in \eqref{eq.generator}. Applying the L\'{e}vy-type It\^{o} formula \cite{applebaum2009levy} to $J^u(t,X_t)$ gives:
\begin{equation*}
  \begin{aligned}
    \dif J^u(t,X_{t}) = & \left[\frac{\partial J^u}{\partial t}(t,X_{t-}) + \mathcal{L}^u J^u(t,X_{t-}) \right]\dif t + \nabla J^u(t,X_{t-})\cdot\sigma(X_{t-},u(t,X_{t-})) \dif W_t \\ 
    & + \int_{R^d} \left[J^u(t, X_{t-}+G(X_{t-},z,u(t,X_{t-}))) - J^u(t,X_{t-},u(t,X_{t-}))\right] \widetilde{N}(\dif t,\dif z) \\ 
    = & -f(t,X_{t-},u(t,X_{t-}))\dif t + \left(\sigma(X_{t-},u(t,X_{t-}))^T \nabla J^u(t,X_{t-})\right)^T\dif W_t \\ 
    & + \int_{R^d} \left[J^u(t, X_{t-}+G(X_{t-},z,u(t,X_{t-}))) - J^u(t,X_{t-},u(t,X_{t-}))\right] \widetilde{N}(\dif t,\dif z), \\ 
  \end{aligned}
\end{equation*}
with the terminal condition $J^u(T,X_T)=g(X_T)$. Integrating over $[t_n, t_{n+1}]$ and comparing with \eqref{eq.TD_update}, we define the alternative reward $\widetilde{R}_{t_nt_{n+1}}$ by:
\begin{equation} \label{eq.reward}
  \begin{aligned}
    \widetilde{R}_{t_nt_{n+1}} = &\ \int_{t_n}^{t_{n+1}}f(t,X_t,u(t,X_t))\dif t - \int_{t_n}^{t_{n+1}}\left(\sigma(X_t,u(t,X_t))^T \nabla \mathcal{N}_v(t,X_t)\right)^T\dif W_t \\ 
    & - \int_{t_n}^{t_{n+1}}\int_{R^d} \left[\mathcal{N}_v(t, X_t+G(X_t,z,u(t,X_t))) - \mathcal{N}_v(t,X_t,u(t,X_t))\right] \widetilde{N}(\dif t,\dif z), \\ 
  \end{aligned}
\end{equation}
and use this TD error $\widetilde{R}_{t_nt_{n+1}} + \mathcal{N}_v(t_{n+1}, X_{t_{n+1}}) - \mathcal{N}_v(t_{n},X_{t_n})$ as the update rule in \eqref{eq.TD_update}.
Note that the difference between $R_{t_nt_{n+1}}$ and $\widetilde{R}_{t_nt_{n+1}}$ lies in the  two additional martingale terms in $\widetilde{R}_{t_nt_{n+1}}$. 

For any control $u(t,x)$, if $\mathcal{N}_v$ is precisely the value function under the current control $u$, that is, $\mathcal{N}_v=J^u$, then we have 
\begin{align*}
 & \widetilde{R}_{t_nt_{n+1}} + \mathcal{N}_v(t_{n+1}, X_{t_{n+1}}) - \mathcal{N}_v(t_{n},X_{t_n}) = 0, \quad \mathbb{P}-a.s., \\ 
  &\E^{t,x} \left[R_{t_nt_{n+1}} + \mathcal{N}_v(t_{n+1}, X_{t_{n+1}}) - \mathcal{N}_v(t_{n},X_{t_n})\right] = 0, \quad t\leq t_n<t_{n+1},\\
&\E^{t,x} \big[\widetilde{R}_{t_nt_{n+1}} + \mathcal{N}_v(t_{n+1}, X_{t_{n+1}}) - \mathcal{N}_v(t_{n},X_{t_n})\big] \\
 &\qquad \qquad  = \E^{t,x} \left[R_{t_nt_{n+1}} + \mathcal{N}_v(t_{n+1}, X_{t_{n+1}}) - \mathcal{N}_v(t_{n},X_{t_n})\right].
\end{align*} 
In the case of infinite horizon stochastic control problems driven by Brownian motions,  \cite{zhou2021actor} discussed the advantage of $\tilde R_{t_nt_{n+1}}$ over $R_{t_nt_{n+1}}$ in theory and conducted numerical experiments on both choices. The main reason is that the $\mathbb{P}$-a.s. condition is stronger than the in-expectation one, thus exhibiting a smaller variance. We expect a similar reasoning in our setting, therefore, will use the reward $\tilde R_{t_nt_{n+1}}$ in designing the critic loss \eqref{eq.CriticLoss1} as well as in the numerical examples in Section~\ref{sec.numerical}. 

At the terminal time $T$, it is desired to have 
\begin{equation}\label{eq.terminal}
    \mathcal{N}_v(T,X_T) = g(X_T),
\end{equation}
which will be integrated into the loss function. Since $\mathcal{N}_v$ is updated after each transition from $X_{t_n}$ to $X_{t_{n+1}}$, it does not hurt to enforce \eqref{eq.terminal} simultaneously. However, $X_T$ is not observable when $t_{n+1} < T$. To remedy this issue, one can use observations from the previous iteration to substitute for $\mathcal{N}_v(T, \cdot)$ and $g(\cdot)$. 

\medskip
\noindent\textbf{Policy improvement.}
The objective of policy improvement is to find the control that maximizes its value function. As discussed earlier in this section, the agent aims to maximize 
over all controls $u \in U$ :
\begin{equation} \label{eq.supJ}
  v(t,x) \equiv \sup_{u\in U} J^u(t,x) = \sup_{u\in U} \E^{t,x} \left[ \int_{t}^{T} f(s,X_s,u_s) \dif s + g(X_T) \right].
\end{equation} 
Therefore, improving the actor $u$ by maximizing $J^u(t,x)$ is a natural approach. However, considering the alteration made to the reward $R_{t_n t_{n+1}}$ by adding martingale terms in \textbf{policy evaluation}, we make a similar adjustment to $J^u(t,x)$:
{\small
\begin{equation} \label{eq.J_tilde}
  \begin{aligned}
    &\widetilde{J}^u(t,x) = \ \E^{t,x} \left[ \int_{t}^{T} f(s,X_s,u_s) \dif s - \int_{t}^{T}\left(\sigma(X_s,u(s,X_s))^T \nabla \mathcal{N}_v(s,X_s)\right)^T\dif W_s \right. \\
    & - \left. \int_{t}^{T}\int_{R^d} \left[\mathcal{N}_v(s, X_s+G(X_s,z,u(s,X_s))) - \mathcal{N}_v(s,X_s,u(s,X_s))\right] \widetilde{N}(\dif s,\dif z) + g(X_T) \right].
  \end{aligned}
\end{equation}}This provides an alternative way for policy improvement, i.e., improving the actor $u$ by maximizing $\widetilde J^u(t,x)$.

In theory, it is clear that
\begin{equation}
  \widetilde{J}^u(t,x) = J^u(t,x),
\end{equation}
for any control $u$ and any approximation of the current value function $\mathcal{N}_v$. In practice, a numerical experiment is conducted in Section~\ref{sec.numerical.merton} to compare the effectiveness of these two objectives: $J$ \emph{vs.} $\widetilde J$. Table~\ref{tab.merton_main} shows that the accuracy is slightly better when using $\widetilde{J}$, but the difference is not significant.  
Additionally, we also attempted to replace the terminal function $g(x)$ with the approximated value function $\mathcal{N}_v(T,x)$, and the difference between the two approaches, i.e. $\E^{t,x}[ \int_{t}^{T} f(s,X_s,u_s) \dif s + g(X_T)]$ and $\E^{t,x}[ \int_{t}^{T} f(s,X_s,u_s) \dif s + \mathcal{N}_v(T,X_T)]$ is also not significant.  

\subsubsection{Implementation details} \label{sec.implementation_details}

This section dedicates to the implementation details of the proposed algorithm, including neural network selections, the computation of non-local terms and the loss function, and training stabilization.

\smallskip
\noindent\textbf{NN architecture.} 
In additional to parameterize the control $u(t,x)$ and the associated value function $J^u(t,x)$ by neural networks $\mathcal{N}_\pi$, $\mathcal{N}_v$, we introduce $\mathcal{N}_{non}$ for the non-local term $\int_{\R^d}(J^u(t,x+G(x,z,u))-J^u(t,x))\nu(\dif z)$ to aid in the computation of the reward. Specifically, we employ residual networks (ResNet) to enhance the generalization capabilities. Each neural network consists of an $n+1$-dimensional input $(t,x)$, a linear layer, several residual blocks, another linear layer, and an output layer. Each residual block consists of two linear layers with activation functions and a residual connection, as illustrated in Figure \ref{fig.ResNet}.
\begin{figure}[htbp]
  \centering
  \includegraphics[width=40em]{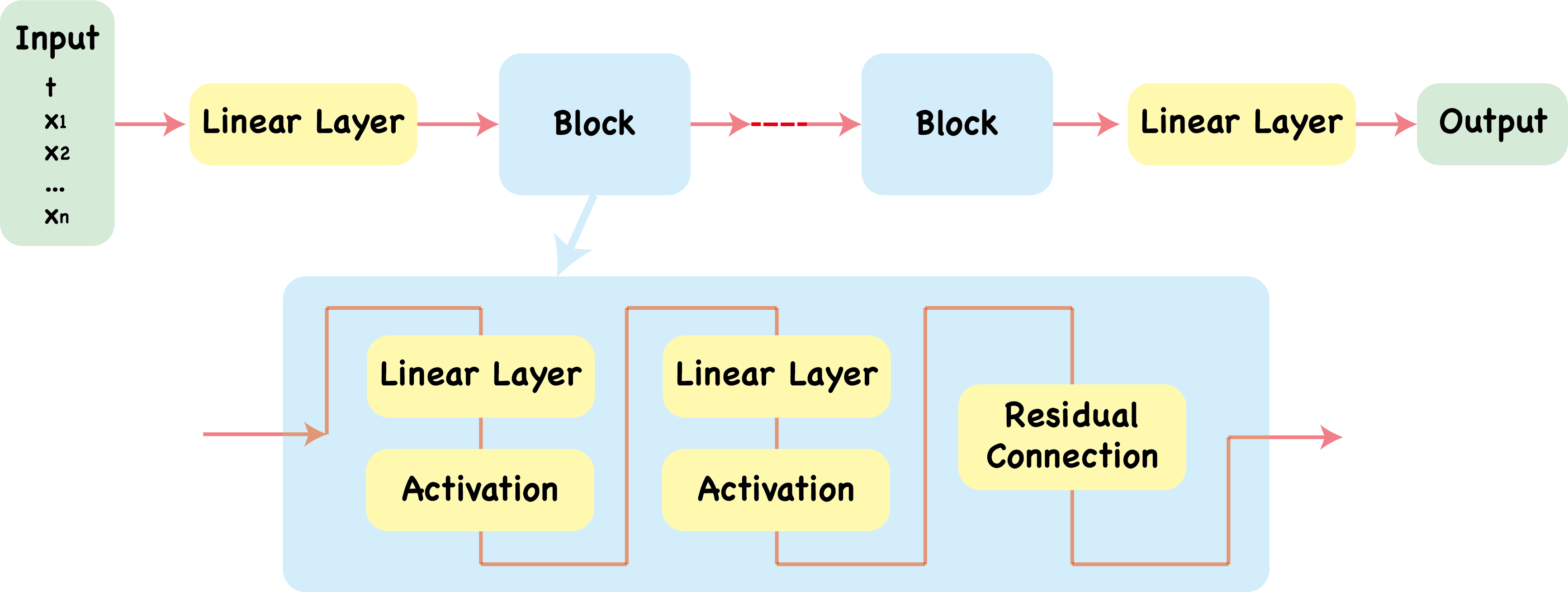}
  \caption{The structure for the actor $u(t,x)$, the critic $J^u(t,x)$ and the non-local term $\int_{\R^d}(J^u(t,x+G(x,z,u))-J^u(t,x))\nu(\dif z)$ employed in this work.} \label{fig.ResNet}
\end{figure}

\smallskip
\noindent\textbf{Loss function.}
Recall that we denote the neural network representing the control function $u(t,x)$ as $\mathcal{N}_\pi(t,x)$, the one representing the value function $J^u(t,x)$ as $\mathcal{N}_v(t,x)$, and the one representing the non-local term $\int_{\R^d}(J^u(t,x+G(x,z,u))-J^u(t,x))\nu(\dif z)$ as $\mathcal{N}_{non}(t,x)$. The update of the state process $X_t$ when time moves from $t_n$ to $t_{n+1} \equiv t_n+\Delta t$ is  given by \eqref{eq.DeltaXt}
\begin{equation} \label{eq.DeltaXt_nn}
  \begin{aligned}
    X^j_{t_{n+1}} = &\ X^j_{t_n} + b(X^j_{t_n},\mathcal{N}_\pi(t_n,X^j_{t_n}))\Delta t + \sigma(X^j_{t_n},\mathcal{N}_\pi(t_n,X^j_{t_n}))\Delta W_n  \\
  & + \sum_{i=N_n+1}^{N_{n+1}} G(X^j_{t_n},z_i,\mathcal{N}_\pi(t_n,X^j_{t_n})) - \Delta t\int_{\R^d} G(X^j_{t_n},z,\mathcal{N}_\pi(t_n,X^j_{t_n})) \nu(\dif z),
  \end{aligned}
\end{equation}
where the superscript $j$ indicates the $j$-th sample path. The associated reward $\widetilde R$ (cf. \eqref{eq.reward}) becomes: 
\begin{equation} \label{eq.reward_nn}
  \begin{aligned}
     \widetilde{R}^j_{t_nt_{n+1}} = &  \ f(t_n,X^j_{t_n},\mathcal{N}_\pi(t_n,X^j_{t_n}))\Delta t 
     -  \left(\sigma(X^j_{t_n},\mathcal{N}_\pi(t_n,X^j_{t_n}))^T \nabla \mathcal{N}_v(t_n,X^j_{t_n})\right)^T\Delta W_n  \\ 
    & -  \Bigg( \sum_{i=N_n+1}^{N_{n+1}} [\mathcal{N}_v(t_n, X^j_{t_n}+G(X^j_{t_n},z^j_i,\mathcal{N}_\pi(t_n,X^j_{t_n})))-\mathcal{N}_v(t_n,X^j_{t_n})] \\
    & \qquad  - \Delta t \mathcal{N}_{non}(t_n,X^j_{t_n}) \Bigg).
  \end{aligned}
\end{equation}
Here, $N_n$ represents the number of jumps occurring in the interval $[0,t_n]$, and $z_i$ represents the size of the $i$-th jump, sorted in the chronological order over the entire time interval $[0,T]$. In equations \eqref{eq.DeltaXt_nn} and \eqref{eq.reward_nn}, $\Delta W_n$ is a same sample drawn from  $\mathcal{N}(0,\Delta t)$.

We now are ready to define the \texttt{CriticLoss}, which has three parts. The first component is from the TD error:
\begin{equation} \label{eq.CriticLoss1}
  \mbox{\texttt{CriticLoss}}^1_{t_n} = \frac{1}{M} \sum_{j=1}^M \left( \widetilde{R}^j_{t_nt_{n+1}} + \mathcal{N}_v(t_{n+1},X^j_{t_{n+1}}) - \mathcal{N}_v(t_n, X^j_{t_n}) \right)^2,
\end{equation}
where we have adopted the idea from Least Squares Temporal Differences (LSTD) \cite{boyan1999least}, and $M$ denotes the sample number. The second component takes care of the terminal  condition:
\begin{equation} \label{eq.CriticLoss2}
  \mbox{\texttt{CriticLoss}}^2_{t_n} = \frac{1}{L}\frac{1}{M} \sum_{j=1}^M \left(\mathcal{N}_v(T,X^j_T) - g(X^j_T)\right)^2,
\end{equation}
and $L$ represents the number of sub-intervals into which $[0,T]$ is divided, motivated by \cite{zeng2022deep} which proposed to equally distribute the loss $\frac{1}{M} \sum_{j=1}^M \left(\mathcal{N}_v(T,X^j_T) - g(X^j_T)\right)^2$ into each interval $[t_n,t_{n+1}]$. Note that $X_T$ is not observed at $t_n <T$; observations of $X_T$ from the previous iteration will be used to evaluate \eqref{eq.CriticLoss2}.
The third component ensures the consistency between $\mathcal{N}_v$ and $\mathcal{N}_{non}$.  
For any control function $\mathcal{N}_\pi(t,x)$ and the value function $\mathcal{N}_v(t,x)$, 
\begin{equation*}
  \left\{ \int_{0}^t \int_{\R^d}\left[\mathcal{N}_v(t,X_t+G(X_t,z,\mathcal{N}_\pi(t,X_t))) - \mathcal{N}_v(t,X_t)\right]\widetilde{N}(\dif t,\dif z) \right\}_{t=0}^T
\end{equation*}
is a martingale, implying 
\begin{equation*}
  \E \left[ \int_{t_n}^{t_{n+1}} \int_{\R^d} \mathcal{N}_v(t,X_t+G(X_t,z,\mathcal{N}_\pi(t,X_t))) - \mathcal{N}_v(t,X_t) \widetilde{N}(\dif t,\dif z) \bigg| \mathcal{F}_{t_n} \right] = 0
\end{equation*}
for any $[t_n,t_{n+1}]\subset[0,T]$. Its  empirical version 
{\small\begin{equation} \label{eq.CriticLoss3}
  \begin{aligned}
    & \mbox{\texttt{CriticLoss}}^3_{t_n} = \\
    & \Bigg|
    \frac{1}{M} \sum_{j=1}^M\Big(\sum_{i=N_n+1}^{N_{n+1}} [\mathcal{N}_v(t_n,X^j_{t_n}+G(X_{t_n},z^j_i,\mathcal{N}_\pi(t_n,X^j_{t_n}))) - \mathcal{N}_v(t_n,X^j_{t_n})] - \Delta t \mathcal{N}_{non}(t_n, X^j_{t_n})\Big)
    \Bigg|.
  \end{aligned}
\end{equation}}serves as the loss function for the non-local term. As a result, the loss function for policy evaluation at time $t_n$ is
\begin{equation} \label{eq.CriticLoss}
  \mbox{\texttt{CriticLoss}}_{t_n} = \mbox{\texttt{CriticLoss}}^1_{t_n} + \mbox{\texttt{CriticLoss}}^2_{t_n} + \mbox{\texttt{CriticLoss}}^3_{t_n}.
\end{equation}

In the policy improvement, the loss function for the actor follows \eqref{eq.supJ}, thus given by:
\begin{equation} \label{eq.ActorLoss}
  \mbox{\texttt{ActorLoss}} = \frac{1}{M} \sum_{j=1}^M \left(\Delta t\sum_{n=0}^{L-1} f(t_n,X^j_{t_n},\mathcal{N}_\pi(t_n,X^j_{t_n})) + g(X^j_T)\right),
\end{equation}
where the state process $X_t$ evolves from $t=0$ to $t=T$ following the current control $\mathcal{N}_\pi(t,x)$.

\smallskip
\noindent\textbf{Explosion prevention.}
The evolution of the state process $X_t$ follows \eqref{eq.DeltaXt_nn} and depends on the current control function $\mathcal{N}_\pi(t,x)$. Due to factors such as neural network initialization, we may encounter the so-called explosion phenomenon during the initial training steps. Namely, the values of the control function may undergo significant changes, leading to situations where the state variables become extremely large. In numerical experiments, for some initialization, we observe that when $X_t$ becomes infinite (Inf), the next time step may result in \textit{Not a Number} (NaN) due to subtracting two infinities, causing subsequent loss function and gradient computations to fail. To stabilize the training,  we here adopt the following two approaches:

\begin{itemize}
  \item Apply the \texttt{Tanh} function to the output of the network $\mathcal{N}_\pi(t,x)$ representing the control, restricting the output to the range $(-1, 1)$, and then multiply it by a trainable parameter $b$, ensuring that the network's output is restricted to the interval $(-b, b)$:     
  \begin{equation}\label{eq.explosion}
    \widetilde{\mathcal{N}}_\pi(t,x) = b*\texttt{Tanh}(\mathcal{N}_\pi(t,x)).
  \end{equation}
  \item No restriction is imposed on the network representing the control. Meanwhile, during the several initial steps of training, restrict the state $X_t$ obtained from \eqref{eq.DeltaXt_nn} to a bounded region $[a, b]$, i.e.,
  \begin{equation}\label{eq.explosion-2}
    \widetilde{X}_t = X_t\cdot\mathds{1}_{X_t\in[a,b]} + a\cdot\mathds{1}_{X_t<a} + b\cdot\mathds{1}_{X_t>b}.
  \end{equation}
\end{itemize}

Both approaches have proven to be viable in our numerical experiments, with preference depending on the problem setting. For example, when the optimal control $u^*(t,x)$ is expected to be bounded (as in Section~\ref{sec.numerical.merton}), adjusting the parameter $b$ to cover the range of $u^\ast(t,x)$ is advantageous. Conversely, in cases where the optimal control is unbounded (as in Section~\ref{sec.numerical.lqr}), the second approach may be more suitable. It's worth noting that, in practice, the approach \eqref{eq.explosion-2} is only applied in the initial training steps to prevent potential early explosions and are subsequently removed. Thus, we do not modify the problem to be solved or restrict the solution space. Furthermore, experiments have shown that these approaches still yield correct solutions without requiring delicate choices of $a$ and $b$.

\smallskip
\noindent\textbf{Summary of the algorithm.}
The pseudo code for solving the jump-diffusion stochastic control problem~\eqref{eq.dXt}--\eqref{eq.J} using actor-critic type deep reinforcement learning is presented in Algorithm \ref{algo.AC_control}. In training the critic, we employ temporal difference learning, and update parameters at each time point. Meanwhile, in training the actor, our objective is to maximize the overall reward throughout the process. We would like to point out that, in line \ref{algo.actor} of the algorithm \ref{algo.AC_control}, we update the parameter multiple times based on one \texttt{ActorLoss}, rather than repeating the process of sampling, calculating the loss, and updating parameters. This choice is made because, as observed in experiments, both approaches yield similar results, but the former requires less computation time, thus being the preferred method for training the actor. Figure \ref{fig.diagram_AC} illustrates the entire algorithm.

\begin{algorithm}
  \renewcommand{\algorithmicensure}{\textbf{Input:}}
	\renewcommand{\algorithmicrequire}{\textbf{Output:}}
  \caption{An actor-critic algorithm for stochastic control problems with jumps.} \label{algo.AC_control}
  \begin{algorithmic}[1]
    \ENSURE{The problem \eqref{eq.dXt}--\eqref{eq.J}, 3 neural networks $(\mathcal{N}_v, \mathcal{N}_\pi, \mathcal{N}_{non})$, snapshots $0=t_0<t_1<t_2<\cdots<t_L=T$, sample size $M$, iteration \# $Iterations$, actor iteration \# $actor\_step$, initial condition $X_0=\xi$, learning rates.}
    \REQUIRE{optimal value function $\mathcal{N}_v(t,x)$ and optimal control $\mathcal{N}_\pi(t,x)$.}
    \STATE{Randomly initialize the previous terminal state $\mathbb{X}=\{X^j_T\}_{j=1}^M$.}
    \FOR{$iteration=0\to Iterations$}
    \STATE{\color{blue}\emph{\%-------- update the critic --------\%}}
    \STATE\label{algo.critic_start}{Sample $M$ Brownian motion increments $\{\Delta W^j_n\}_{j=1}^M$ on $t_n\in\{t_0,t_1,\cdots,t_L\}$.} 
    \STATE{Sample jumps $\{z^j_i\}_{j=1}^M$ from the L\'{e}vy random measure $N$.} 
    \FOR{time step $n=0\to L-1$}
    \STATE{Update the state process from $X^j_{t_n}$ to $X^j_{t_{n+1}}$ by \eqref{eq.DeltaXt_nn}, $j=1,2,\cdots,M$.}
    \STATE{Compute $\mbox{\texttt{CriticLoss}}^1_{t_n}$ by \eqref{eq.CriticLoss1} and $\mbox{\texttt{CriticLoss}}^3_{t_n}$ by \eqref{eq.CriticLoss3}.}
    \STATE{Update the previous terminal state $\mathbb{X}=\{X^j_T\}_{j=1}^M$ if $n=L-1$.}
    \STATE{Compute $\mbox{\texttt{CriticLoss}}^2_{t_n}$ by \eqref{eq.CriticLoss2} and the current $\mathbb{X}$.} 
    \STATE{Compute $\mbox{\texttt{CriticLoss}}_{t_n}$ by \eqref{eq.CriticLoss}. Update $\mathcal{N}_v, \mathcal{N}_{non}$ by minimizing $\mbox{\texttt{CriticLoss}}_{t_n}$.}
    \STATE{Decay the learning rate of the critic if necessary.}
    \ENDFOR\label{algo.critic_end}
    \STATE{\color{blue}\emph{\%-------- update the actor --------\%}}
    \STATE\label{algo.actor_start}{Sample $M$ Brownian motion increments $\{\Delta W^j_n\}_{j=1}^M$ on $t_n\in \{t_0, t_1, \cdots, t_L\}$.} 
    \STATE{Sample jumps $\{z^j_i\}_{j=1}^M$ from the L\'{e}vy random measure $N$.}
    \STATE{Simulate the state process from $X^j_0$ to $X^j_T$ by \eqref{eq.DeltaXt_nn}, $j=1,2,\cdots,M$.}
    \STATE{Compute \texttt{ActorLoss} by \eqref{eq.ActorLoss}.} 
    \STATE\label{algo.actor}{Update $\mathcal{N}_\pi$ $actor\_step$ times by maximizing \texttt{ActorLoss}.}
    \STATE\label{algo.actor_end}{Decay the learning rate of actor if necessary.}
    \STATE{\color{blue}\emph{\%-------- Model evaluation --------\%}}
    \STATE{Compute the error of $\mathcal{N}_v$ and $\mathcal{N}_\pi$.}
    \ENDFOR
  \end{algorithmic}
\end{algorithm}

\begin{figure}[htbp]
  \centering
  \includegraphics[width=40em]{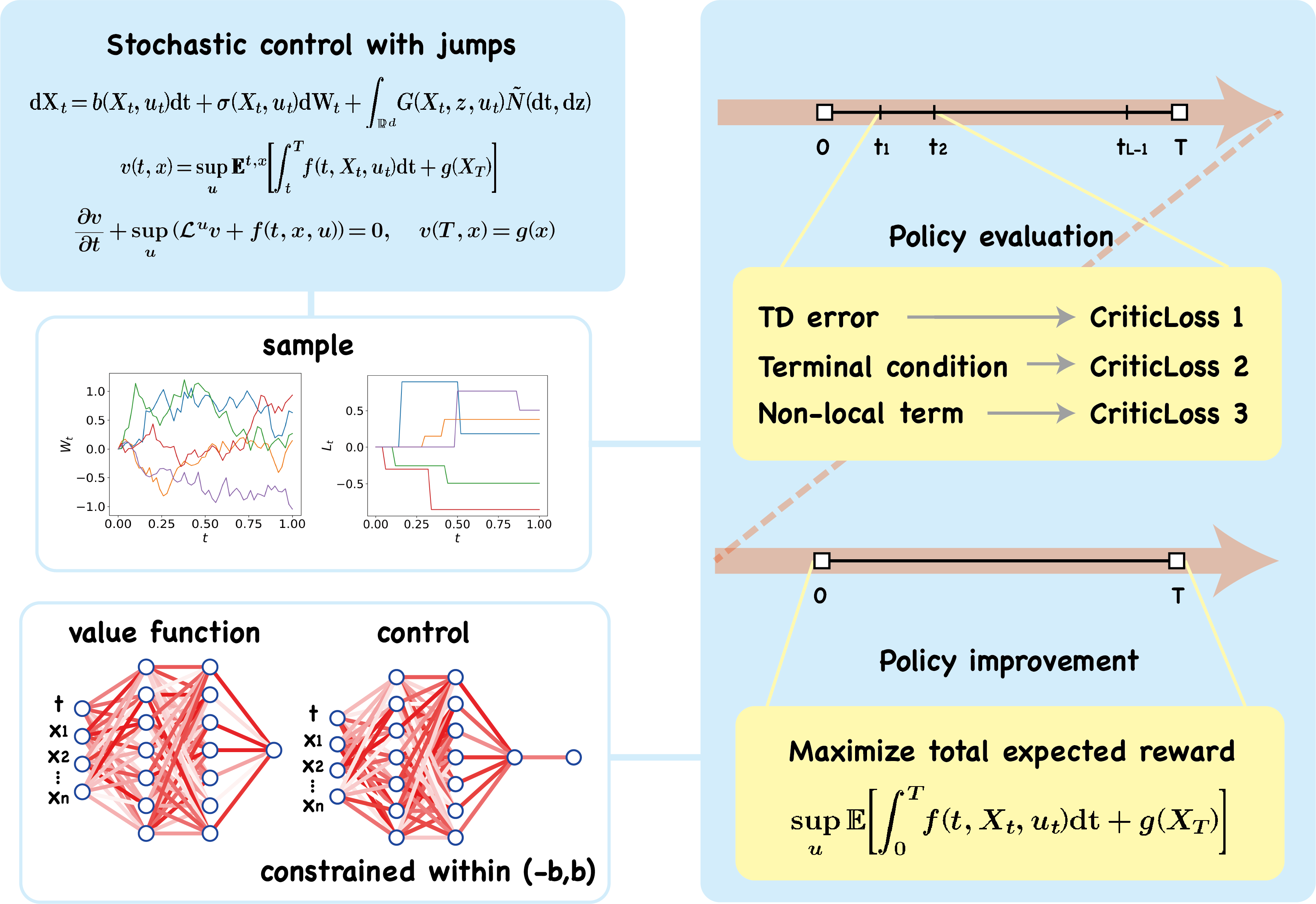}
  \caption{The diagram of solving stochastic control problems with jumps using deep reinforcement learning.} \label{fig.diagram_AC}
\end{figure}

\subsubsection{The simplified algorithm for problems subject to Poisson jumps}\label{sec.poisson}

This section considers the controlled state process $X_t$ subject to Poisson jumps  
\begin{equation} \label{eq.dXt_Poisson}
  \dif X_t = b(X_{t-}, u_t)\dif t + \sigma(X_{t-}, u_t)\dif W_t + \sum_{k=1}^m z_k(X_{t-},u_t)\dif M_t^k,
\end{equation}
where $z_k(x,u): \R^d\times\R^{d_c}\to\R^d$, $N_t^k$ is a Poisson process with intensity $\lambda_k$ and  $\{N_t^k\}_{k=1}^m$ are independent, and $M_t^k=N_t^k-\lambda_k t$ is the corresponding compensated Poisson process. It is a special case of \eqref{eq.dXt} where $\nu(\dif z)$ is a discrete measure. 
Such models are simpler and typically offer better interpretability and traceability, and have been widely used in modeling stock prices for study portfolio optimization and option pricing, for instance, in \cite{guo2004optimal,aase1984optimum}. In such cases, the proposed deep RL framework can be significantly simplified, and this section will highlight these difference. 

The definition of  $J^u(t,x)$ remains as in \eqref{eq.J}. The optimal value function 
$
  v(t,x) = \sup_{u\in U} J^u(t,x)
$
satisfies the same Hamilton-Jacobi-Bellman equation~\eqref{eq.HJB} albeit with a different generator: 
\begin{equation} \label{eq.generator_Poisson}
  \begin{aligned}
    \mathcal{L}^u v(t,x) = &\ b(x,u)\cdot\nabla v(t,x) + \frac{1}{2}\mbox{Tr}\left[\sigma(x,u)\sigma^T(x,u)\mbox{H}(v(t,x))\right] \\ 
    & + \sum_{k=1}^m \lambda_k\left(v(t,x+z_k(x,u))-v(t,x)-z_k(x,u)\cdot\nabla v(t,x)\right). \\
  \end{aligned}
\end{equation}
Applying the L\'{e}vy-type It\^{o} formula to $J^u(t,X_t)$ brings: 
\begin{equation}
  \begin{aligned}
    \dif J^u(t,X_t) = & -f(t,X_{t-},u(t,X_{t-}))\dif t + \left(\sigma(X_{t-},u(t,X_{t-}))^T \nabla J^u(t,X_{t-})\right)^T\dif W_t \\ 
    & + \sum_{k=1}^m \left[J^u(t, X_{t-}+G(X_{t-},z,u(t,X_{t-}))) - J^u(t,X_{t-},u(t,X_{t-}))\right] \dif M_t^k. \\ 
  \end{aligned}
\end{equation}
Moreover, the reward $\widetilde R$ is shifted to the following expression
\begin{equation}\label{eq.Rtilde.poisson}
  \begin{aligned}
    \widetilde{R}^j_{t_nt_{n+1}} = &\ f(t_n,X^j_{t_n},\mathcal{N}_\pi(t_n,X^j_{t_n}))\Delta t - \left(\sigma(X^j_{t_n},\mathcal{N}_\pi(t_n,X^j_{t_n}))^T \nabla \mathcal{N}_v(t_n,X^j_{t_n})\right)^T\Delta W_n  \\ 
    - & \sum_{k=1}^m\left( \mathcal{N}_v(t_n, X^j_{t_n}+G(X^j_{t_n},z^j_i,\mathcal{N}_\pi(t_n,X^j_{t_n})))-\mathcal{N}_v(t_n,X^j_{t_n}) \right) \Delta M_t^k.
  \end{aligned}
\end{equation}
Here $X_{t_n}$ is given by the Euler scheme of \eqref{eq.dXt_Poisson}, the sampling of $\Delta M_t^k$ is given by $\Delta M_t^k=k_i-\lambda_k\Delta t$ if the Poisson process 
$N_t^k$ experiences $k_i$ jumps ($k_i\geq1$) within the interval $[t_n, t_n+\Delta t]$; and  $\Delta M_t^k=-\lambda_k\Delta t$ otherwise.

Then, $\mbox{\texttt{CriticLoss}}^1_{t_n}$ (cf. \eqref{eq.CriticLoss1}) can be calculated based on~\eqref{eq.Rtilde.poisson}. The calculations of $\mbox{\texttt{CriticLoss}}^2_{t_n}$ and \texttt{ActorLoss} (cf. \eqref{eq.CriticLoss2} and \eqref{eq.ActorLoss}) are the same as in the general case. The term $\mbox{\texttt{CriticLoss}}^3_{t_n}$ becomes unnecessary. This is because, in the compound Poisson case, the non-local term in the HJB equation \eqref{eq.generator_Poisson} as well as in \eqref{eq.Rtilde.poisson} requires evaluating only a finite number of surrounding points, which already has an explicit form. Consequently, there is no need to introduce an additional neural network $\mathcal{N}_{non}$ for the purpose of efficiently evaluating the non-local term, and hence no need for $\mbox{\texttt{CriticLoss}}^3_{t_n}$ to  ensure the consistency between $\mathcal{N}_{non}$ and the neural network for the critic $\mathcal{N}_{v}$.


\subsection{Solving games using fictitious play and parallel computing}\label{sec.dlgame}

We extend the previously proposed framework to address Nash equilibrium in stochastic differential games within the It\^o-L\'evy model, leveraging the idea of fictitious play (FP) and parallel computing. Originally introduced in seminal works by Brown \cite{brown1949some,brown1951iterative}, FP is a learning process aimed at identifying Nash equilibrium. In FP, the game undergoes repeated resolution, wherein each agent iteratively optimizes their own payoff while assuming that other agents adhere to strategies from prior iterations, with the hope that the strategies obtained will converge to a Nash equilibrium. 
This approach transforms the problem into a series of $n$ decoupled control problems, amenable to the application of Algorithm~\ref{algo.AC_control}. Given the decoupled nature of the optimization problems in each iteration, parallel computing facilitates acceleration.

We first present the algorithm designed for stochastic differential games under general L\'{e}vy-It\^{o} processes. Let the state process $X_t=(X_t^1,X_t^2,\cdots,X_t^n)^T$ be accessible by all the agents and satisfies
\begin{equation}\label{eq.dXt.game}
  \dif X_t = b(X_{t-}, \pi_{t})\dif t + \sigma(X_{t-}, \pi_{t})\dif W_t + \int_{\R^d} G(X_{t-},z,\pi_{t}) \widetilde{N}(\dif t,\dif z).
\end{equation}
To simplify the notation, we consider the agent $i$'s state process $X_t^i\in\R$, the control vector of all agents $\pi_t=(\pi_t^1,\pi_t^2,\cdots,\pi_t^n)\in\R^n$, $b: \R^n\times\R^n\to\R^n$, $\sigma: \R^n\times\R^n\to\R^{n\times n}$, $W_t$ is a $n$-dimensional Brownian motion, $G: \R^n\times\R^n\times\R^n\to\R^n$, and $\widetilde{N}(\dif t,\dif z)$ is a $n$-dimensional compensated Poisson random measure, $i\in\mathcal{I}$. The $i$-th agent's goal is to maximize her expected utility of running and terminal reward:
\begin{equation}\label{eq.v.game}
    v^i(t,x) \equiv \sup_{\pi^i}J_i^{\pi}(t,x) \equiv \sup_{\pi^i}\E^{t,x}\left[\int_{t}^{T} f_i(s, X_s,\pi_s)\dif s+g_i(X_T)\right], \quad x\in\R^n.
\end{equation}
where $f_i: \R \times \R^n\times\R^n\to\R$ and $g_i: \R^n\to\R$.
The value functions $\{v^i(t,x)\}_{i \in \mathcal{I}}$ satisfies the coupled HJB system:
\begin{equation}\label{eq.HJB.game}
    \left\{
    \begin{aligned}
        & \frac{\partial v^i}{\partial t}(t,x) + \sup_{\pi^i} \bigg[ b(x,\pi)\cdot\nabla v^i(t,x) + \frac{1}{2}\mbox{Tr}\big[\sigma(x,\pi)\sigma^T(x,\pi)\mbox{H}(v^i(t,x))\big] \\ 
        & \quad + \int_{\R^n} \big(v^i(t,x+G(x,z,\pi))-v^i(t,x)-G(x,z,\pi)\cdot\nabla v^i(t,x)\big)\nu(\dif z) + f_i(t,x,\pi) \bigg] = 0, \\
        & v^i(T,x)=g_i(X_T).
    \end{aligned}
    \right.
\end{equation}

We propose to introduce $3n$ neural networks $\mathcal{N}_v^i(t,x_1,\cdots,x_n), \mathcal{N}_\pi^i(t,x_1,\cdots,x_n), \\
\mathcal{N}_{non}^i(t,x_1,\cdots,x_n)$, to approximate the controls $\pi_t^i$, value functions $v^i$ and non-local terms in \eqref{eq.HJB.game}, for all agents $i \in \mathcal{I}$, which depend on time $t$ and the current state $(x_1,\cdots,x_n)$ of all agents. Naturally, we add an additional loop outside Algorithm \ref{algo.AC_control}, which corresponds to the $n$ tasks for all agents. Specifically, when the loop iterates over the $i$-th agent, we fix the control functions of other agents $\{\mathcal{N}_\pi^k\}_{k\neq i}$, and train $\mathcal{N}_v^i$ and $\mathcal{N}_{non}^i$ using the ``update the critic'' step (line \ref{algo.critic_start}-\ref{algo.critic_end}) and train $\mathcal{N}_\pi^i$ using the ``update the actor'' step (line \ref{algo.actor_start}-\ref{algo.actor_end}) in Algorithm \ref{algo.AC_control}. The value functions of other agents $\{\mathcal{N}_v^k\}_{k\neq i}$ do not enter the computation graph for the \texttt{CriticLoss} and \texttt{ActorLoss} of agent $i$, thus can be disregarded. Finally, an outer loop is added, corresponding to FP, to ensure the convergence of the control $\mathcal{N}_\pi^i$ and value functions $\mathcal{N}_v^i$ and non-local terms $\mathcal{N}_{non}^i$ for all agents $i \in \mathcal{I}$ together.

The aforementioned framework effectively addresses the game with the desired accuracy. However, the computational cost has escalated substantially due to the incorporation of two additional loops. Inspired by Asynchronous Advantage Actor-Critic method (A3C) \cite{zhang2020deep} in deep reinforcement learning, we have adapted the agent loop in the algorithm that requires sequential computations for each agent to parallel computation, as depicted in Algorithm \ref{algo.game}. To this end, we
\begin{itemize}
  \item Firstly, define $3n$ global networks $\mathcal{N}_v^i,\mathcal{N}_\pi^i,\mathcal{N}_{non}^i, i \in \mathcal{I}$ to approximate the value function, the control and the non-local term, respectively. 
  \item Then, before updating the control and value functions for the $i$-th player, define $3n$ local networks $temp\_\mathcal{N}_v^k, temp\_\mathcal{N}_\pi^k, temp\_\mathcal{N}_{non}^k$, $k \in \mathcal{I}$ whose parameters are copied from the global networks. All subsequent computations of actor and critic losses for agent $i$ are executed through these local networks. 
  \item Finally, update the parameters of the local networks $temp\_\mathcal{N}_v^i, temp\_\mathcal{N}_\pi^i, temp\_\mathcal{N}_{non}^i$ and transfer them back to the global networks $\mathcal{N}_v^i,\mathcal{N}_\pi^i,\mathcal{N}_{non}^i$.
\end{itemize}
Now, the training of agents $i$ and $j (j\neq i)$ can proceed independently, without affecting each other. Consequently, parallel computation can be employed, leading to a significant reduction in computational costs. Figure \ref{fig.diagram_parallel} illustrates the parallel computing process.

\begin{algorithm}
  \renewcommand{\algorithmicensure}{\textbf{Input:}}
	\renewcommand{\algorithmicrequire}{\textbf{Output:}}
  \caption{A parallel computing actor-critic framework for solving games with jumps} \label{algo.game}
  \begin{algorithmic}[1]
    \ENSURE{Number of agents $n$, the problem \eqref{eq.dXt.game}--\eqref{eq.v.game}, $3n$ neural networks $(\mathcal{N}_v^i, \mathcal{N}_\pi^i,\mathcal{N}_{non}^i)$, $i \in \mathcal{I}$, snapshots $0=t_0<t_1<\cdots<t_L=T$, sample size $M$, outer iteration \# $Iter\_out$, iteration \# of each agent $Iter\_inner$, actor iteration \# $actor\_step$, initial condition $X_0^i=x_0^i$, learning rates. 
    }
    \REQUIRE{Nash equilibrium $\mathcal{N}_\pi^i(t,x_1,\cdots,x_n)$ and the optimal value function $\mathcal{N}_v^i(t,x_1,\cdots,x_n)$ for each agent $i \in \mathcal{I}$.}
    \STATE{Initialize the global networks $(\mathcal{N}_v^i, \mathcal{N}_\pi^i, \mathcal{N}_{non}^i)$, $i \in \mathcal{I}$.}
    \FOR{$iteration=0\to Iter\_out$}
    \STATE{\color{blue}\emph{\% --------- parallel computing among all agents --------- \%}}
    \FOR{agent $i=1,2,\cdots,n$}
    \STATE{Define $3n$ local networks $(temp\_\mathcal{N}_v^k, temp\_\mathcal{N}_\pi^k, temp\_\mathcal{N}_{non}^k)$, $k \in \mathcal{I}$.}
    \STATE{Load the parameters from the global networks $(\mathcal{N}_v^k, \mathcal{N}_\pi^k,\mathcal{N}_{non}^k)$ into the local networks $(temp\_\mathcal{N}_v^k, temp\_\mathcal{N}_\pi^k, temp\_\mathcal{N}_{non}^k)$, $\forall k \in \mathcal{I}$.}
    \STATE{Fix $(temp\_\mathcal{N}_v^k$, $temp\_\mathcal{N}_\pi^k$, $temp\_\mathcal{N}_{non}^k)$, $k\neq i$, and update $(temp\_\mathcal{N}_v^i$, $temp\_\mathcal{N}_\pi^i$, $temp\_\mathcal{N}_{non}^i)$ based on Algorithm \ref{algo.AC_control} $Iter\_inner$ times.}
    \STATE{Load the parameters from the local networks $(temp\_\mathcal{N}_v^i, temp\_\mathcal{N}_\pi^i, temp\_\mathcal{N}_{non}^i)$ into the global networks $(\mathcal{N}_v^i, \mathcal{N}_\pi^i, \mathcal{N}_{non}^i)$.}
    \ENDFOR
    \STATE{Wait for all agents to finish their computations.}
    \STATE{Evaluate the accuracy of $\mathcal{N}_v^i$ and $\mathcal{N}_\pi^i$ for all $i \in \mathcal{I}$.}
    \ENDFOR
  \end{algorithmic}
\end{algorithm}

\begin{figure}[htbp]
  \centering
  \includegraphics[width=25em]{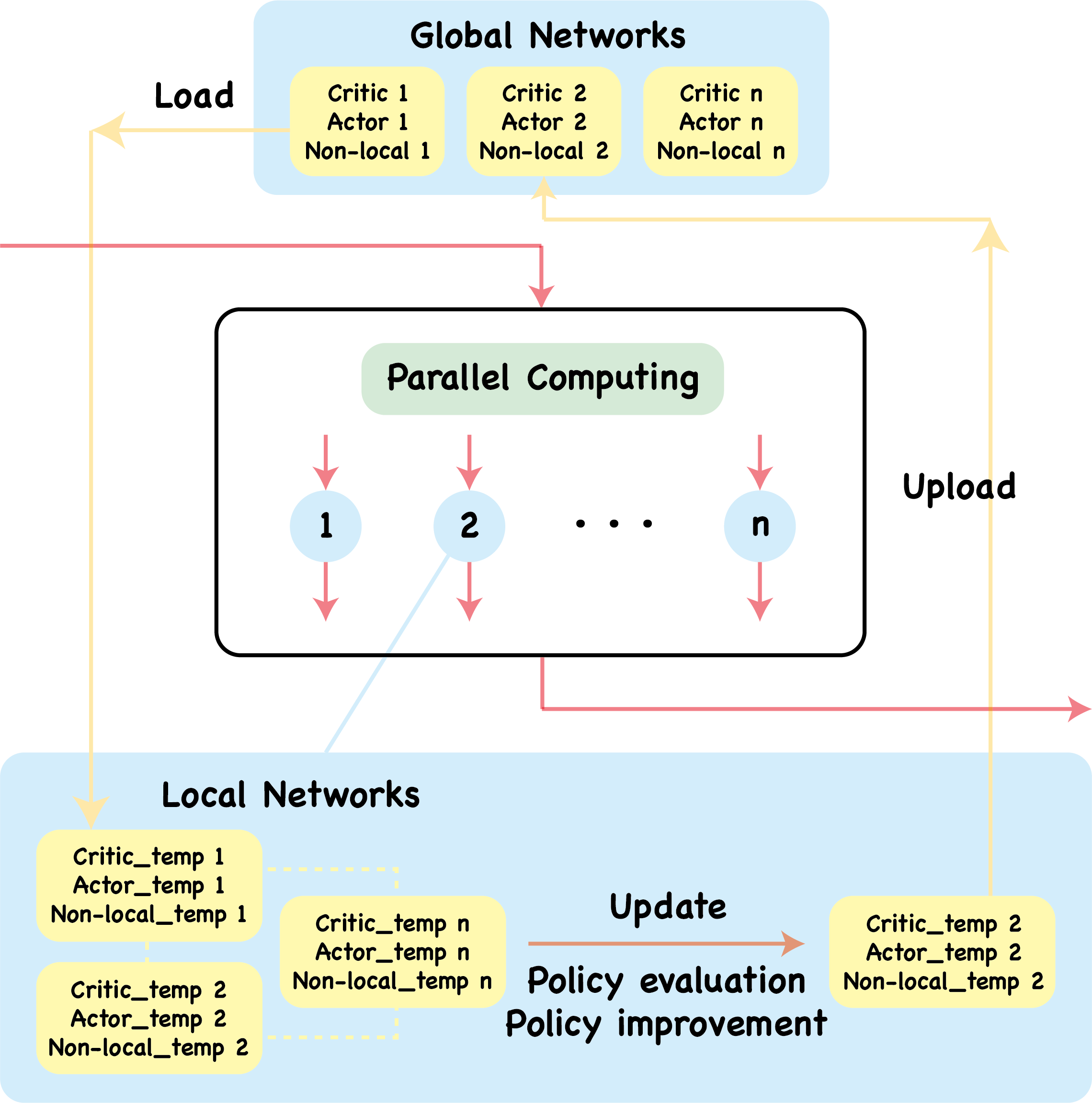}
  \caption{The illustration for solving multi-agent games with jumps using parallel computing. In each iteration of the outer loop, $n$ agents simultaneously perform computations. Each agent instantiates local networks for the critic, actor, and non-local terms for subsequent computation and then uploads the network parameters back to the global networks.} \label{fig.diagram_parallel}
\end{figure}

\medskip
\noindent\textbf{The simplified algorithm for games analyzed in Section~\ref{sec.game}.}
For the portfolio game discussed in Section \ref{sec.game}, the parallel computing framework can be simplified, much like the simplification performed in Section~\ref{sec.poisson} for control problems. We parameterize the control and the value function as functions of the states of all agents $(t,x_1,\cdots,x_n)$, despite the semi-explicit solution derived in Section~\ref{sec.game} that depends solely on $(t, x_i, y_i)$ (where $y_i$ represents the sum (or product) of other agents' states).  

As for the \texttt{CriticLoss}, the reward $\widetilde R_{t_nt_{n+1}}$ in $\mbox{\texttt{CriticLoss}}_{t_n}^1$ is modified to the following 
\begin{equation*}
  \begin{aligned}
    & \widetilde{R}_{t_nt_{n+1}} = -\sum_{k=1}^n \frac{\partial v^i}{\partial x_k} \pi_t^k(\nu_k\Delta W_n^k+\sigma_k\Delta B_n) \\ 
    & + \sum_{k=1}^n \left(v^i({t_n},X_{t_n}^1,\cdots,X_{t_n}^{k-1},X_{t_n}^k+\pi_{t_n}^k\alpha_k,X_{t_n}^{k+1},\cdots,X_{t_n}^n)-v^i({t_n},X_{t_n}^1,\cdots,X_{t_n}^n)\right)\Delta M_n^k \\ 
    & + \left(v^i({t_n},X_{t_n}^1+\pi_{t_n}^1\beta_1,\cdots,X_{t_n}^n+\pi_{t_n}^n\beta_n)-v^i({t_n},X_{t_n}^1,\cdots,X_{t_n}^n)\right)\Delta M_n^0,
  \end{aligned}
\end{equation*}
obtained by applying It\^o formula to $v^i(t,X_t^1,\cdots,X_t^n)$, and state dynamics \eqref{eq.dXt_exp}.
The neural networks representing the non-local term $\mathcal{N}_{non}^i$ are no longer necessary, nor is the $\mbox{\texttt{CriticLoss}}_{t_n}^3$.
There is no change in the $\mbox{\texttt{CriticLoss}}_{t_n}^2$ and the computation of the \texttt{ActorLoss}.


\section{Numerical experiments} \label{sec.numerical}

This section presents several numerical examples of stochastic control problems and games to demonstrate the effectiveness of the actor-critic framework proposed earlier, including Merton's problem, linear quadratic regulator, and the portfolio game analyzed in Section~\ref{sec.game}. The algorithm is implemented using the machine learning library \texttt{Pytorch}, and the code is available on GitHub repository\footnote{https://github.com/LiWeiLu2016/Multi-agent-game-with-jumps}.
Therefore, the results presented here can be reproduced and further developed.

\smallskip
\noindent\textbf{Hyperparameter choices and the computational environment. }
The time interval $0=t_0<t_1<\cdots<t_L=T$ is discretized equally with $T=1$ and $L=50$. The batch size, referred to as the number of trajectories for each update, is set to $M=500$. The neural network adopts the ResNet structure illustrated in Figure \ref{fig.ResNet} with three residual blocks, each having a width of $d+10$, where $d$ is the dimension of the spatial variable $x$. The \texttt{Tanh} activation function is employed, and the Adam optimizer is used for updating network parameters. For the stochastic control problem using Algorithm~\ref{algo.AC_control}, the number of iterations is set to $Iterations=1000$, with $actor\_step=10$. For multi-agent games using Algorithm~\ref{algo.game}, the outer (FP) loop $Iter\_out=100$, the inner loop $Iter\_inner=100$, and $actor\_step=10$. The learning rate is set to $1e-3$ for the first 60\% of the total iterations, $1e-4$ for the next 20\%, and $1e-5$ for the final 20\%. 

These choices of neural network architectures do not necessitate prior knowledge of any explicit solution structure, and no pre-training is performed. All experiments use the same random seed 2023. The experiments were conducted on a MacBook Pro with 32GB of memory and a 10-core M1 Pro chip.

\smallskip
\noindent\textbf{Metrics for evaluation.}
For stochastic control problems, the error in the value function $v$ and the control $u^\ast$ are defined as follows 
\begin{gather}\label{def.error.value}
  \mbox{\texttt{Error\_value}} = \sum_{n=0}^{L-1} \Delta t \sqrt{\frac{\sum_{j=1}^M \left(\hat{v}(t_n, X_{t_n}^j)-v(t_n, X_{t_n}^j)\right)^2}{\sum_{j=1}^M v^2(t_n, X_{t_n}^j)}}, \\ 
  \mbox{\texttt{Error\_control}} = \sum_{n=0}^{L-1} \Delta t \sqrt{\frac{\sum_{j=1}^M \left(\hat{u}(t_n, X_{t_n}^j)-u^\ast(t_n, X_{t_n}^j)\right)^2}{\sum_{j=1}^M (u^\ast)^2(t_n, X_{t_n}^j)}}, \label{def.error.control}
\end{gather}
where $\hat{v}$ and $\hat u$ represent their respective approximations, and $X^j_{t_n}$ denotes the value of the $j$-th trajectory at time $t_n$. In the context of multi-agent games, the error is defined as the average of the errors in value functions/controls across all agents
\begin{gather}\label{def.error.value.game}
  \mbox{\texttt{Error\_value\_game}} = \frac{1}{n} \sum_{i=1}^n \mbox{\texttt{Error\_value\_agent}}_i, \\ 
  \mbox{\texttt{Error\_control\_game}} = \frac{1}{n} \sum_{i=1}^n \mbox{\texttt{Error\_control\_agent}}_i.\label{def.error.control.game}
\end{gather}

\subsection{Merton's problem under a jump-diffusion model}\label{sec.numerical.merton}
This is the single-agent version of the game analyzed in Section~\ref{sec.game}. For completeness, we will repeat the problem setup below. 

Consider a financial market with two assets: a risk-free asset $S^0_t$ earning interest at a rate $r$, and a stock $S^1_t$ subject to Brownian and Poisson noises \cite{oksendal2019stochastic}: 
\begin{align*}
  &\dif S_t^0 = r S_t^0 \dif t, \\ 
  &\dif S_t^1 = S_{t-}^1(\mu \dif t+\sigma\dif B_t+z\dif M_t),
\end{align*}
where $\mu>r$, $z>-1$, $M_t=N_t-\lambda t$ and $N_t$ is a Poisson process with intensity $\lambda$, and $B_t$ is a standard Brownian motion. A person with an initial wealth $x_0$ forms an investment portfolio. At time $t$, she invests a proportion $u_t$ of her wealth in the stock, while the remaining wealth is invested in the risk-free asset. Assuming $u_t$ is self-financing, her wealth process $X_t$ satisfies:
\begin{equation*}
  \dif X_t = \frac{u_tX_{t-}}{S_{t-}^1}\dif S_t^1 + \frac{(1-u_t)X_{t-}}{S_t^0} \dif S_t^0 = \left(r+u_t(\mu-r)\right)X_{t-}\dif t+\sigma u_t X_{t-}\dif B_t + z u_t X_{t-}\dif M_t.
\end{equation*}
The goal is to determine the process $u_t$ to maximize the expected utility at the terminal time:
\begin{equation*}
  v(t,x) = \sup_{u} \E^{t,x} [g(X_T)],
\end{equation*}
where the utility function $g(x)=x^p/p$ with $p \in (0,1)$, or $g(x)=\log x$. Via dynamic programming, the optimal value function $v(t,x)$ satisfies the following HJB equation
\begin{equation*}
  \left\{
  \begin{aligned}
    & \partial_t v + \sup_{u} \Big[((r+u(\mu-r))x \partial_x v + \frac{1}{2}\sigma^2 u^2 x^2 \partial_{xx} v \\
    & \qquad \qquad \qquad + \lambda\left(v(t,x(1+zu))-v(t,x)-zux \partial_x v(t,x) \right) \Big] = 0, \\ 
    & v(T,x) = g(x).
  \end{aligned}
  \right.
\end{equation*}

In the case of $g(x)=\log x$, the problem has an analytical solution of the form $$v(t,x) = k(T-t) + \log x,$$
with $k = r+u^*(\mu-r) - \frac{1}{2}\sigma^2(u^*)^2 + \lambda(\log(1+zu^*)-zu^*)$ and the optimal control is given by: 
$$
u^* = \frac{1}{2\sigma^2z}\left(-(\lambda z^2-(\mu-r)z+\sigma^2)+\sqrt{(\lambda z^2-(\mu-r)z+\sigma^2)^2+4\sigma^2z(\mu-r)}\right).
$$

When $g(x) = \frac{1}{p}x^p$, the solution is semi-explicit and takes the form:
\begin{equation*}
    v(t,x) = e^{k(T-t)}\cdot \frac{1}{p}x^p,
\end{equation*}
with $k = p(r+u^*(\mu-r)) - \frac{1}{2}\sigma^2(u^*)^2p(1-p) + \lambda((1+zu^*)^{p}-1-pzu^*)$. Note that $k$ depends on the optimal control $u^\ast$, which solves the algebraic equation
\begin{equation}\label{eq.merton_optimal_control}
    \mu-r-(1-p)\sigma^2u^*+\lambda z\left((1+zu^*)^{p-1}-1\right)=0, \quad g(x)=\frac{1}{p}x^p.
\end{equation}

In this example, we set the model parameters as $\mu=0.05$, $r=0.03$, $\sigma=0.4$, $p=0.5$, $\lambda=0.3$, $z=0.2$ and $x_0=10$. The output of the neural network representing the control $\mathcal{N}_\pi$ is processed by the \texttt{Tanh} function and then multiplied by a learnable parameter $b$, as described in subsection \ref{sec.implementation_details}. The initial value of $b$ is set to 1. The benchmark under logarithmic utility is fully explicit, while, to obtain a reference solution in the case of power utility, the Python function \texttt{scipy.optimize.fsolve}  is used to solve equation~\eqref{eq.merton_optimal_control}.

Table \ref{tab.merton_main} displays the errors in the value functions and control obtained using two \texttt{ActorLoss} formulations,  $J$ and  $\widetilde{J}$ (cf. \eqref{eq.supJ}--\eqref{eq.J_tilde}), under power and logarithmic utility. It can be observed that the accuracy is slightly better when using \texttt{ActorLoss} with the added martingale term $\widetilde{J}$, but the difference is not significant. 
Therefore, we choose to use the original \texttt{ActorLoss} $J$ defined in \eqref{eq.supJ} in subsequent experiments for a easier implementation.
In Figure \ref{fig.merton_traj}, panels~(a)--(c) visualize several trajectories of the value function $v$ and the control $u$ and their approximated counterparts along the simulated state process $X_t$, and state process $X_t$ under power utility and \texttt{ActorLoss} $J$. It can be seen that they closely match the exact trajectories. Panels~(d) and(e) depicts how the losses and errors, defined in \eqref{def.error.value}--\eqref{def.error.control}, change with training iterations. Panel~(f) illustrates the relative error at each point between the learned control and the exact solution, indicating that the learned control function remains consistently small over time.

\begin{table}[!htb]
  \centering
  \caption{The errors of the value function $v$ and the control $u^\ast$ of Merton's problem with jumps under different utilities, trained with different actor losses $J$ or $\widetilde J$.} \label{tab.merton_main}
  \begin{tabular}{l|cc|cc}
    \toprule
    Utility & \multicolumn{2}{c}{$g(x)=x^p/p$} & \multicolumn{2}{|c}{$g(x)=\log x$} \\ 
    \midrule
    \texttt{ActorLoss} & $J$ & $\widetilde{J}$ & $J$ & $\widetilde{J}$ \\
    \midrule
    \texttt{Error\_value} & 0.013\% & 0.012\% & 0.010\% & 0.017\% \\ 
    \texttt{Error\_control} & 1.227\% & 0.555\% & 4.718\% & 3.386\% \\
    \midrule
    Time (min) & 9.6 & 10.3 & 10.8 & 10.3 \\
    \bottomrule
  \end{tabular}
\end{table}

\begin{figure}[!htb]
  \centering 
  \subfloat[The value function]{\includegraphics[width=13.5em,height=11em]{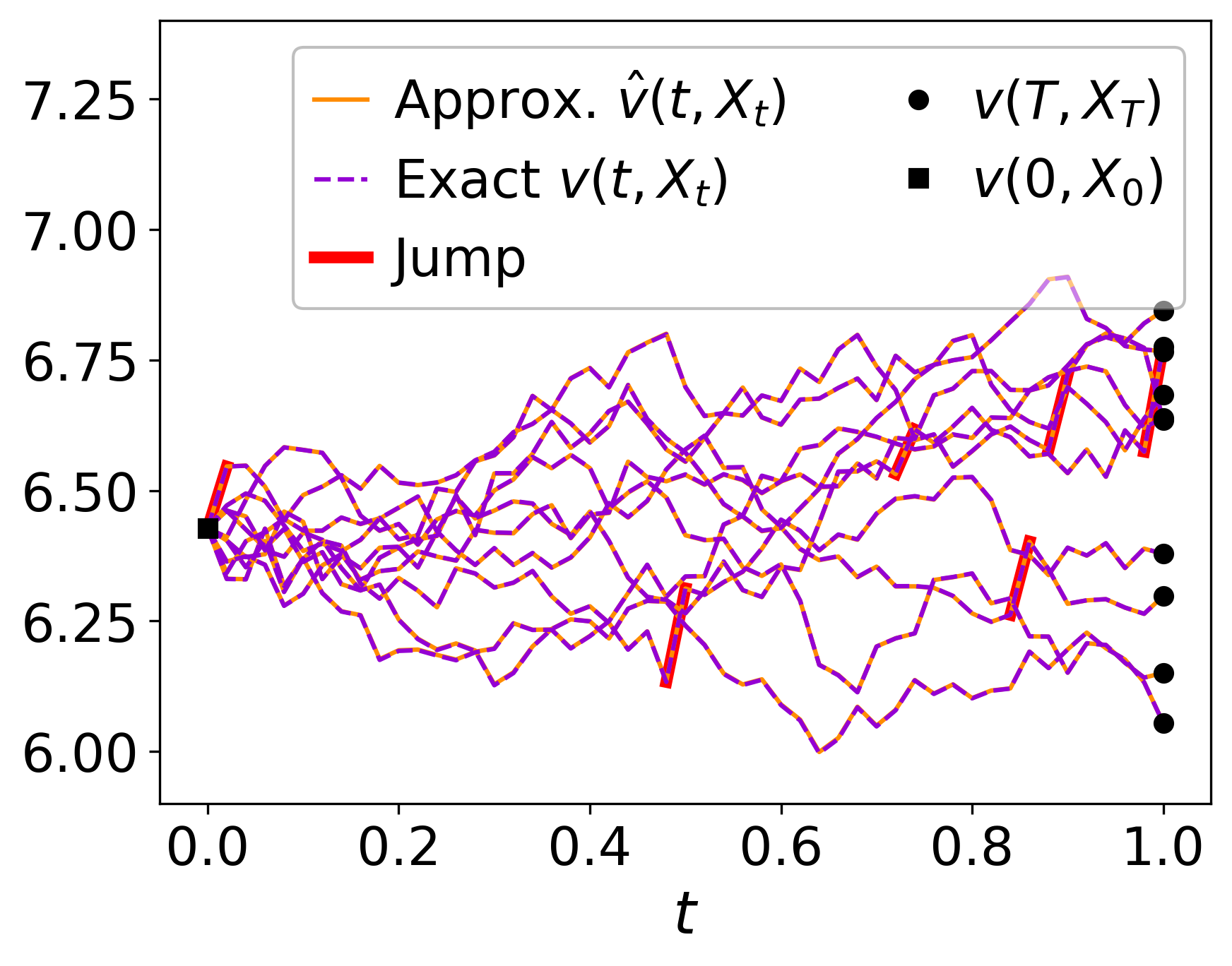}} 
  \subfloat[The control process]{\includegraphics[width=13.5em,height=11em]{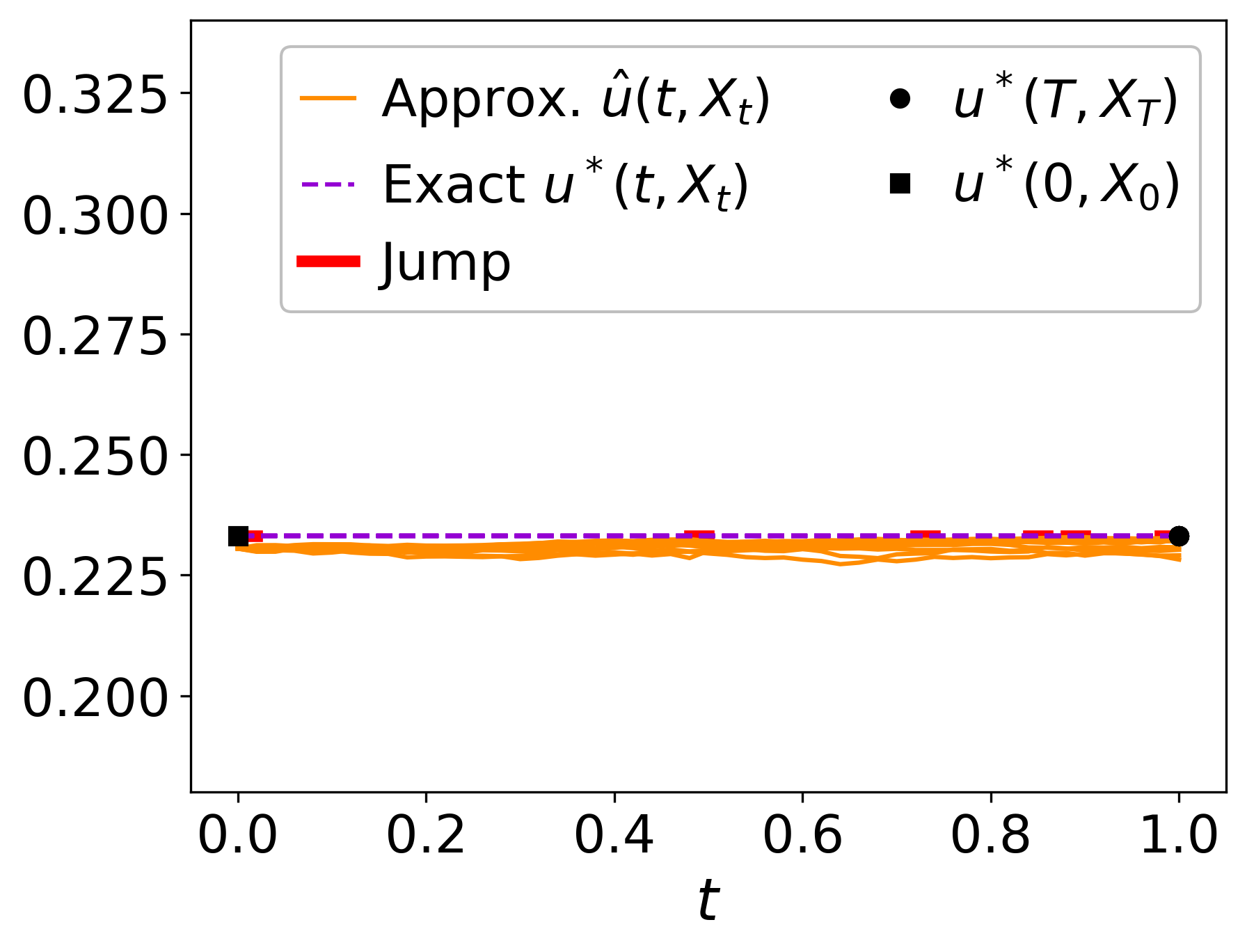}} 
  \subfloat[The state process $X_t$]{\includegraphics[width=13.5em,height=11.25em]{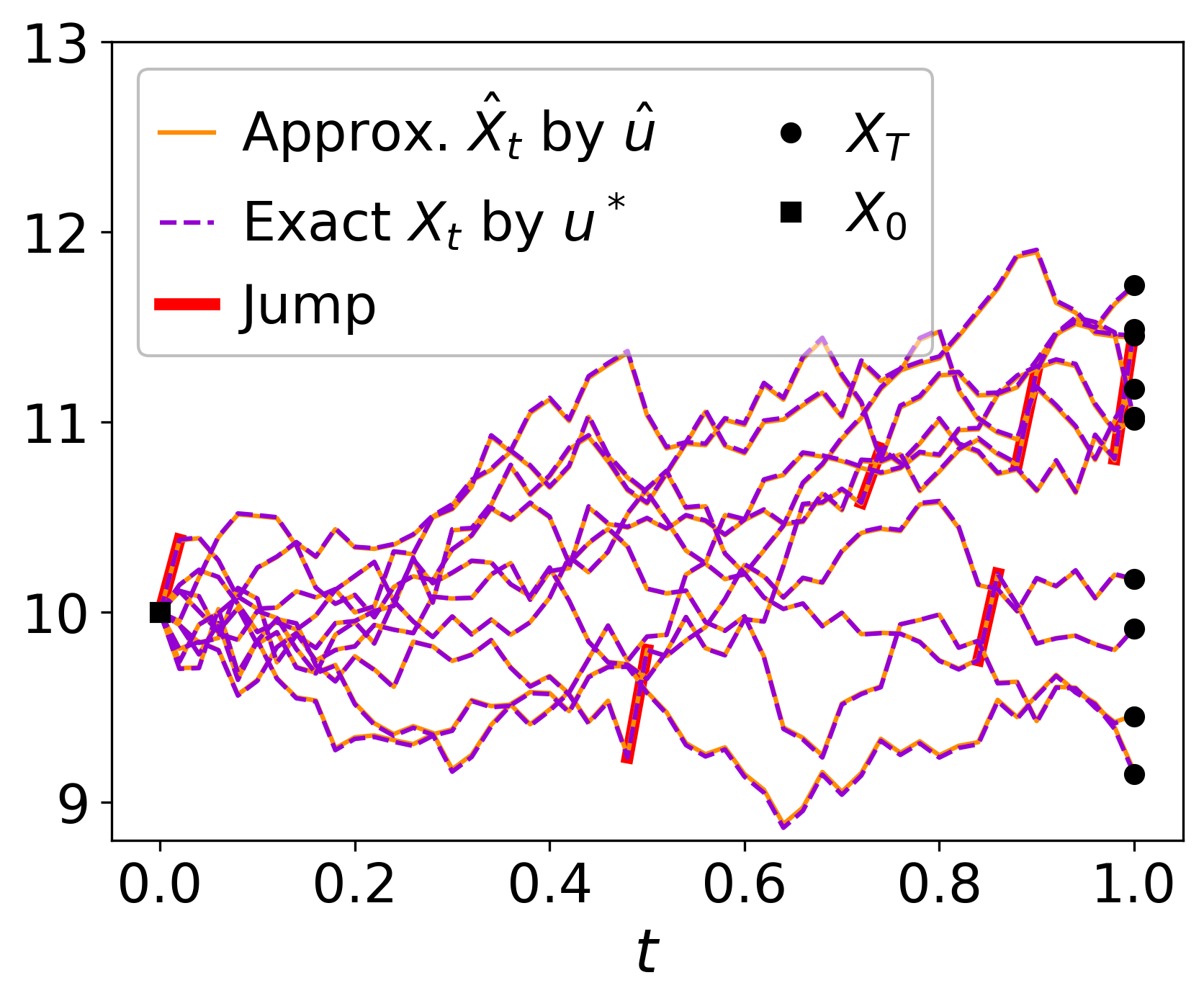}} 

  \subfloat[Losses during the training]{\includegraphics[width=13.5em,height=10.2em]{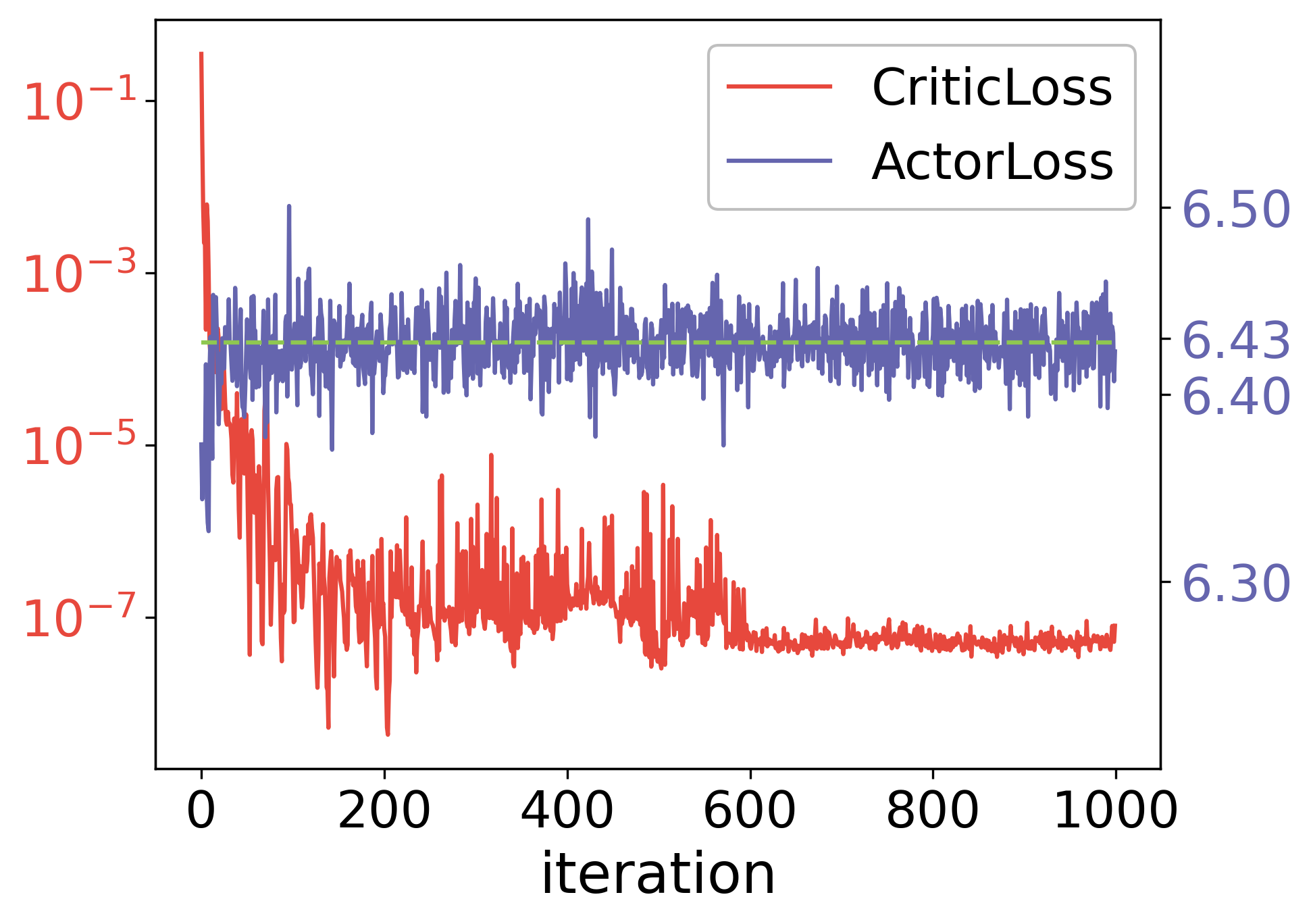}}
  \subfloat[Errors during the training]{\includegraphics[width=13.5em]{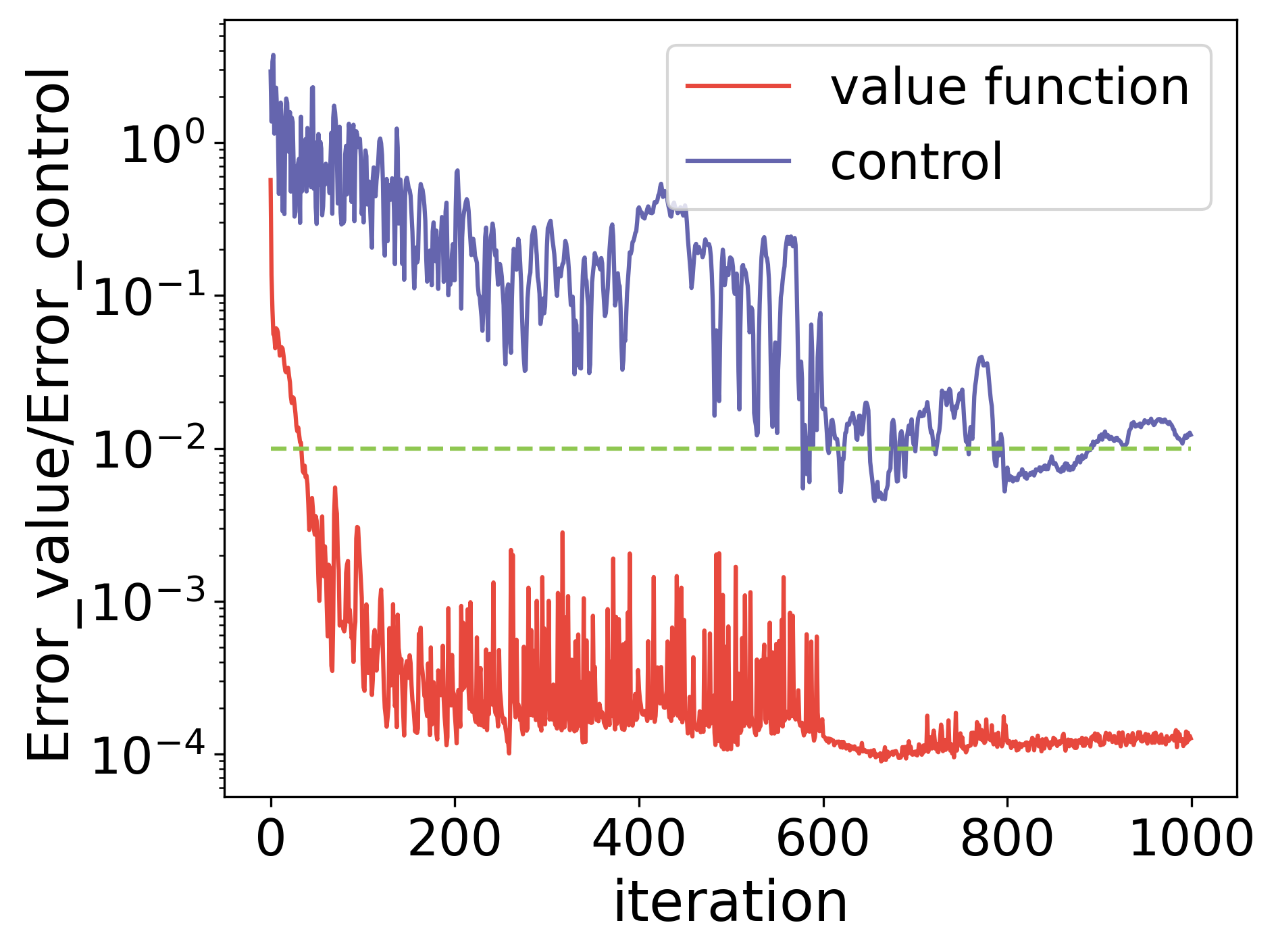}} 
  \subfloat[$L^1$ relative error of the control]{\includegraphics[width=13.5em]{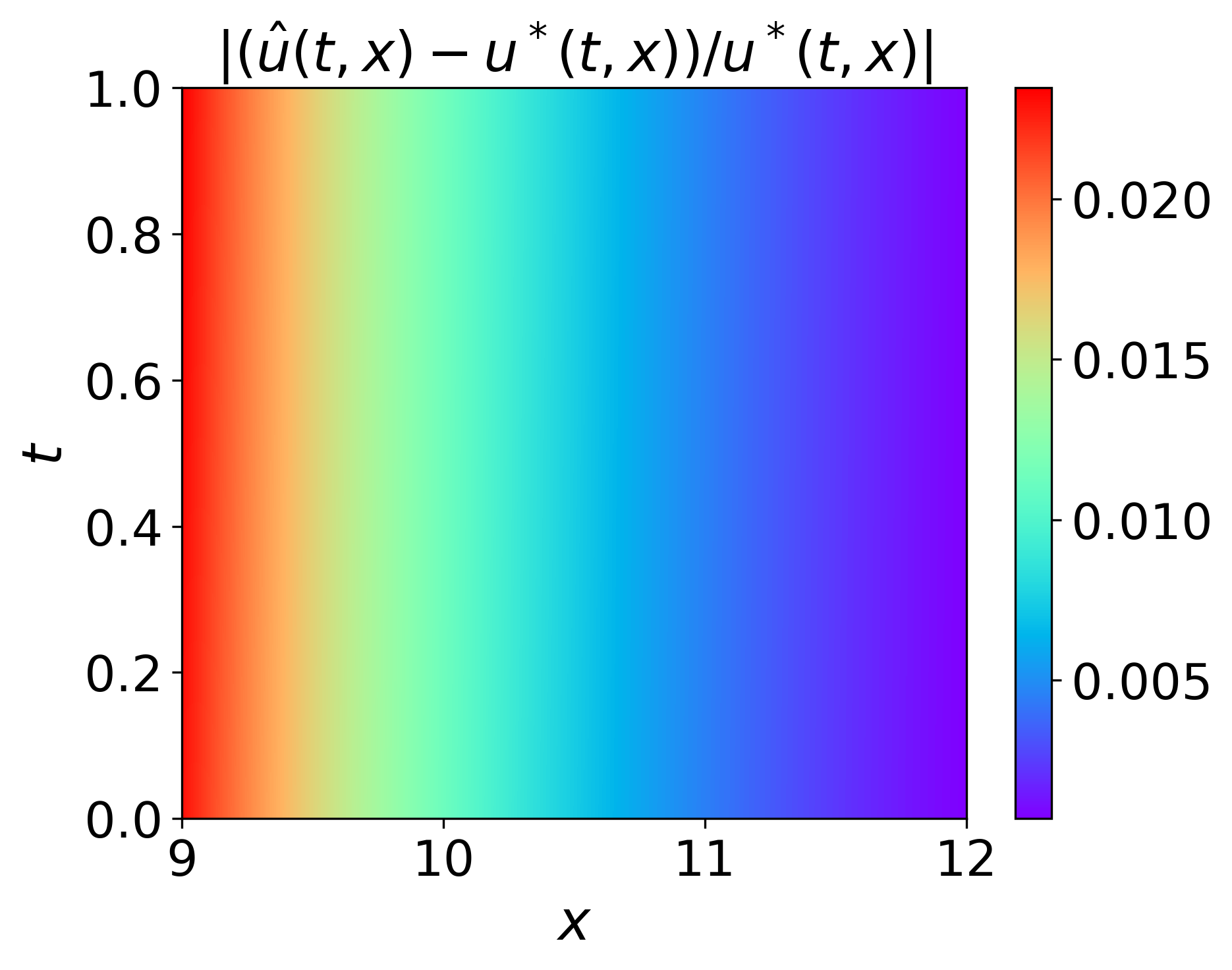}} 
  \caption{Merton's problem with jumps under power utility. (a)--(c) Visualization of sample trajectories of the value function along the optimal state process $v(t, X_t)$, the control process $u^\ast(t, X_t)$, and the state process $X_t$, as well as  their approximated counterparts; (d) Plot of the \texttt{CriticLoss} \eqref{eq.CriticLoss} and the \texttt{ActorLoss} \eqref{eq.ActorLoss} with respect to the training iterations; (e) Plot of the \texttt{Error\_value} \eqref{def.error.value} and \texttt{Error\_control} \eqref{def.error.control} with respect to the training iterations; and (f) Heatmap of the $L^1$ relative error of the approximated control $\hat u(t,x)$ as a bivariate function of $t$ and $x$.}
  \label{fig.merton_traj}
\end{figure}

The performance is further analyzed in Figure~\ref{fig.merton_error}, where Panels (a)--(b) display the $L^2$ relative errors of the value function $e_t^v$ and the control $e_t^u$ with respect to the spatial variable $x$ at each time $t$:
\begin{equation}\label{eq.l2relative}
    e_t^v=\sqrt{\frac{\sum_{j=1}^M \left(\hat{v}(t, X_{t}^j)-v(t, X_{t}^j)\right)^2}{\sum_{j=1}^M v^2(t, X_{t}^j)}}, \,  e_t^u=\sqrt{\frac{\sum_{j=1}^M \left(\hat{u}(t, X_{t}^j)-u^\ast(t, X_{t}^j)\right)^2}{\sum_{j=1}^M (u^\ast)^2(t, X_{t}^j)}}.
\end{equation}
Panel (c) shows the robustness of our algorithm with respect to poor initializations. Recall that we set a bound $b$ for 
the neural network representing the control $\mathcal{N}_\pi$ to prevent possible early explosions in the entire training process. The original initialization choice of $b$ produces a range $(-1, 1)$ covering the optimal control $u^*=0.2331$, numerically solved from \eqref{eq.merton_optimal_control}. In this test, we run the algorithm with two other initializations $(-0.05, 0.05)$ and $ (-0.15, 0.15)$. Panel (c) shows that, despite neither of them covering $u^*=0.2331$, the trainable parameter $b$ will gradually converge to $u^\ast$. This observation indicates that, the artificial bound is purely to increase the stability of the algorithm, without excluding any potential solution space.

\begin{figure}[!htb]
  \centering 
  \subfloat[$L^2$ relative errors $e_t^v$/$e_t^u$]{\includegraphics[width=15.75em,height=11.25em]{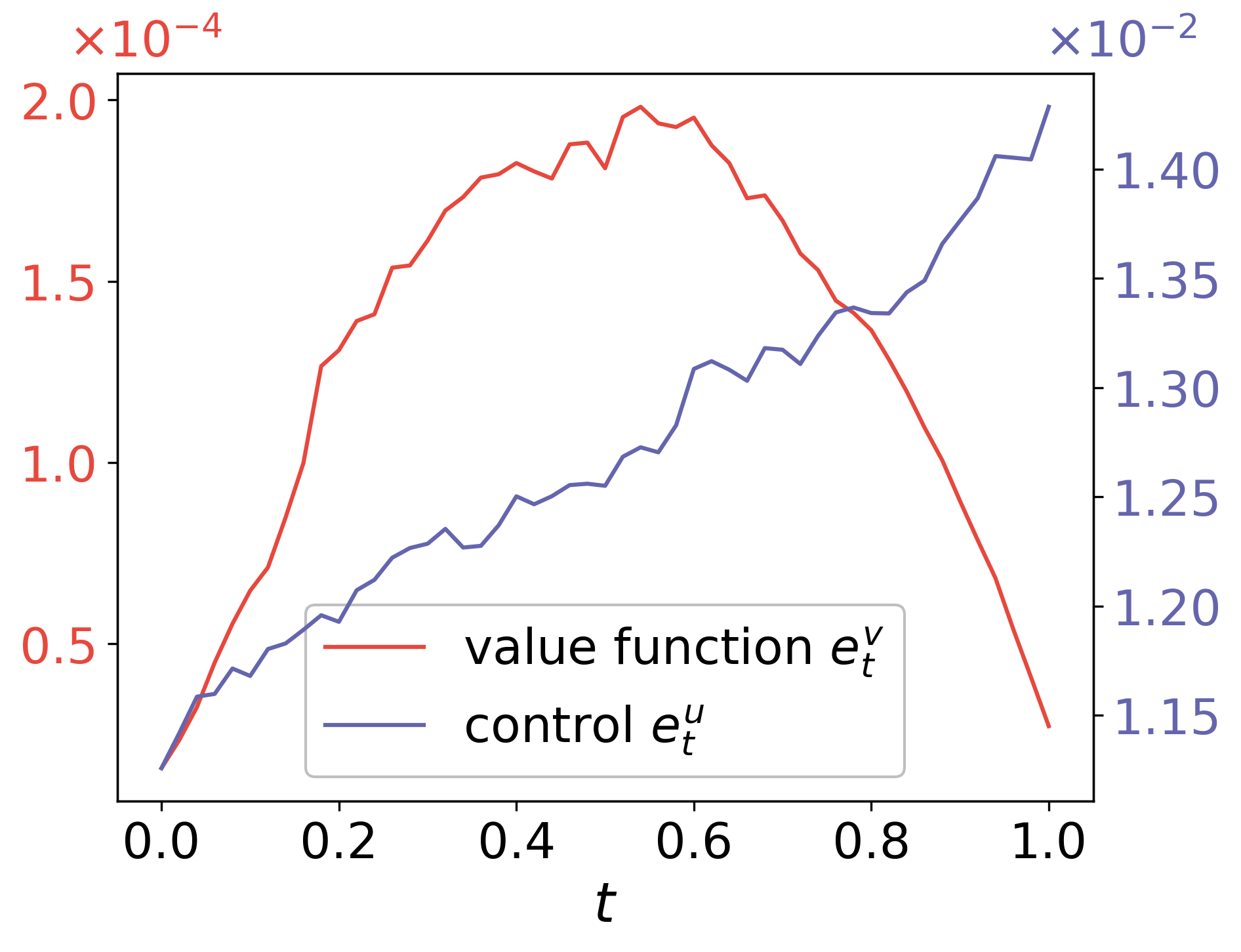}} 
  \subfloat[$b$]{\includegraphics[width=16em,height=11em]{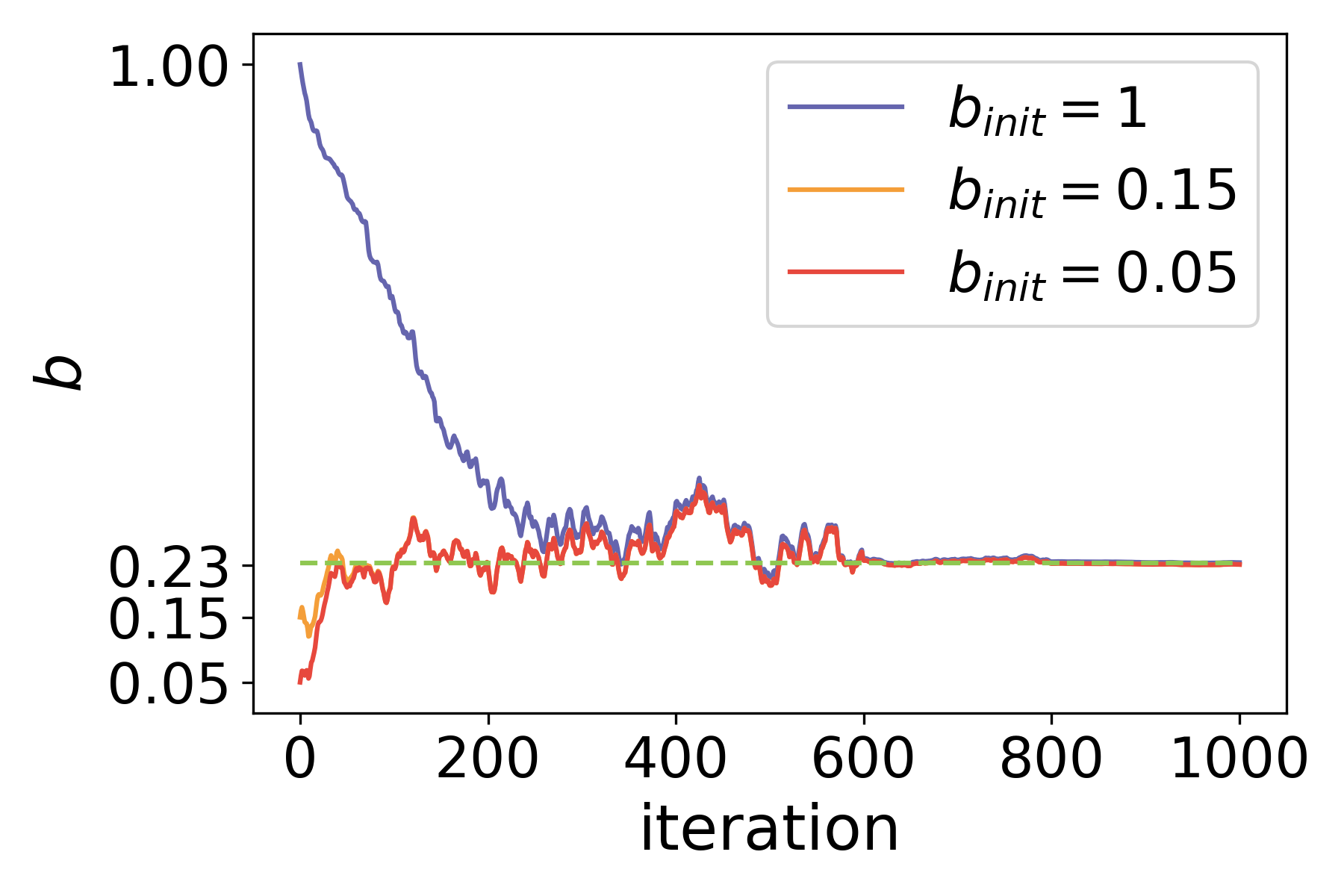}} 
  \caption{Merton's problem with jumps under power utility. (a) $L^2$ relative errors of the value function $e_t^v$ and the control $e_t^u$ at different times $t$; (b) Plot showing the evolution of the control function's bound $b$ (a trainable parameter in $\mathcal{N}_\pi$) with respect to training iterations under different initializations.}
  \label{fig.merton_error}
\end{figure}

Finally, we examine the impact of the choice of $M$ (number of trajectories) on the algorithm's accuracy, as demonstrated in Table \ref{tab.merton_traj}. The tests were conducted with power utility and the same parameter settings as before, except for the choice of $M$. It can be observed that increasing the number of trajectories does not significantly improve the accuracy. Balancing between accuracy and computational cost, we opt to use $M = 500$ trajectories in subsequent experiments.

\begin{table}[h]
  \centering
  \caption{The errors in the value function $v$ and the control $u^\ast$ of Merton's problem with jumps under power utility, trained with different numbers $M$ of trajectories. 
  } \label{tab.merton_traj}
  \begin{tabular}{lcccccc}
    \toprule
    Trajectories $M$ & 125 & 250 & 500 & 2000 & 4000 & 6000 \\
    \midrule
    \texttt{Error\_value} & 0.005\% & 0.011\% & 0.013\%  & 0.012\% & 0.011\% & 0.010\% \\ 
    \texttt{Error\_control} & 15.757\% & 1.803\% & 1.227\%  & 2.878\% & 1.293\% & 2.428\% \\
    Time (min) & 5.8 & 7.6 & 9.6 & 29.6 & 49.4 & 56.9 \\
    \bottomrule
  \end{tabular}
\end{table}

\subsection{Stochastic linear quadratic regulator problem}\label{sec.numerical.lqr}
In the second example, we consider a high-dimensional
linear-quadratic regulator control problem for jump-diffusion models. The state process  $X_t$ satisfies
\begin{equation*}
  \dif X_t = u_t \dif t+\sigma\dif B_t+z\dif M_t,
\end{equation*}
where the state $X_t\in\R^d$, the control $u_t\in\R^d$, $B_t$ is a $d$-dimensional standard Brownian motion, $M_t=(M_t^1,\cdots,M_t^d)$, $M_t^i=N_t^i-\lambda_i t$ and $N_t^i$ denotes a Poisson process with intensity $\lambda_i$. The $d \times d$ matrices $\sigma$ and $z$ are specified by 
\begin{equation*}
    \sigma = \sigma_0
  \begin{pmatrix}
    1 & 0 & 0 & 0 & \cdots & 0 \\
    1 & 1 & 0 & 0 & \cdots & 0 \\
    0 & 1 & 1 & 0 & \cdots & 0 \\
    0 & 0 & 1 & 1 & \cdots & 0 \\
    \vdots & \vdots & \vdots & \vdots & \ddots & 0 \\ 
    0 & 0 & 0 & \cdots & 1 & 1 
  \end{pmatrix}, \quad
  z = 
  \begin{pmatrix}
      z_1 & 0 & 0 & \cdots & 0 \\
      0 & z_2 & 0 & \cdots & 0 \\
      0 & 0 & z_3 & \cdots & 0 \\
      \vdots & \vdots & \vdots & \ddots & 0 \\
      0 & 0 & 0 & \cdots & z_d \\
  \end{pmatrix}.
\end{equation*}
The running cost and terminal cost in \eqref{eq.J} are of the form:
$$
    f(t,x,u) = q\|u\|^2 + b\|x\|^2, \quad g(x) = a\|x\|^2,
$$
with $q>0$, $b>0$, $a>0$. 

The optimal value function, defined by,
\begin{equation*}
  v(t,x) = \inf_{u} \E^{t,x} \left[\int_{t}^{T} (q\|u_s\|^2+b\|X_s\|^2)\dif s + a\|X_T\|^2\right],
\end{equation*}
satisfies the PIDE
\begin{equation*}
\begin{aligned}
    \partial_t v & +\inf_u \bigg[u\cdot \nabla v +\frac{1}{2}\mbox{Tr}(\sigma\sigma^T\mbox{H}(v)) \\
    & +\sum_{i=1}^d\lambda_i\Big(v(t,x_1,\cdots,x_i+z_i,\cdots,x_d)-v(t,x)-z_i \partial_{x_i} v(t,x) \Big) +q\|u\|^2+b\|x\|^2 \bigg] = 0, \\
\end{aligned}
\end{equation*}
with the terminal condition $v(T,x)=a\|x\|^2$. The solution is quadratic in $x$
\begin{equation*}
  \begin{aligned}
    v(t,x) = &\ \sqrt{bq}\left(-1+\frac{2(\sqrt{bq}+a)}{(\sqrt{bq}-a)e^{-2\sqrt{\frac{b}{q}}(T-t)}+\sqrt{bq}+a}\right)\|x\|^2 \\ 
    & + \left(-\sqrt{bq}(T-t)+q\log\frac{\sqrt{bq}-a+(\sqrt{bq}+a)e^{2\sqrt{\frac{b}{q}}(T-t)}}{2\sqrt{bq}}\right) \Big((2d-2)\sigma_0^2+\sum_{i=1}^d\lambda_i z_i^2\Big),
  \end{aligned}
\end{equation*}
and the optimal control is linear in the state $x$: 
\begin{equation*}
  u^*(t,x)=-\frac{\nabla v}{2q}=-\sqrt{\frac{b}{q}}\left(-1+\frac{2(\sqrt{bq}+a)}{(\sqrt{bq}-a)e^{-2\sqrt{\frac{b}{q}}(T-t)}+\sqrt{bq}+a}\right)x.
\end{equation*}

In this problem, we set the model parameters as follows: $\sigma_0=0.4$, $z_i=0.3-0.1(i-1)/(n-1)$, $\lambda_i=0.2+0.1(i-1)/(n-1)$, $b=0.1$, $q=5$, $a=1$ and the initial value $x_0=(1,1,\cdots,1)$. The \texttt{Error\_control} is computed as the average of errors in each of its dimensions. In the 5-dimensional experiment, using the original \texttt{ActorLoss} $J$, the \texttt{Error\_value} is 2.991\%, and the \texttt{Error\_control} is 2.061\%. 
We then tested our algorithm in various dimensions with the same set of parameters and recorded the errors in Table~\ref{tab.LQR}. It is evident that our approach demonstrates promising performance even in high-dimensional LQR problems under jump-diffusion models. The numerical performance of the 25-dimensional problem is presented in Figure \ref{fig.LQR}. Panels (a)--(b) plots, respectively, multiple trajectories of the value function $v$ and the 25-th control process $u_{25}$ along the optimal trajectories, and state process $X_t$ \emph{vs.} their approximated counterparts, all of which closely match the exact trajectories. Panels~(c)--(d) demonstrates the decreasing trend of the losses and the errors during training, and panel~(e) shows the $L^2$ relative error of the control $e_t^u$ and the value function $e_t^v$ at each time point $t$ with respect to the spatial variable $x$, where $e_t^v$ follows the definition~\eqref{eq.l2relative} and $e_t^u$ is defined similarly but further averaged over all 25 dimensions. 

\begin{table}[h]
    \centering
    \caption{The errors in the value function $v$ and the control $u^\ast$ of LQR problems in different dimensions $d$. }
    \label{tab.LQR}
    \begin{tabular}{lccccc}
        \toprule
        Dimension $d$ & 5 & 10 & 15 & 20 & 25 \\
        \midrule
        \texttt{Error\_value} & 2.991\% & 2.080\% & 1.530\% & 1.457\% & 1.410\% \\ 
        \texttt{Error\_control} & 2.061\% & 2.288\% & 2.471\% & 3.125\% & 3.696\% \\
        Time (min) & 11.1 & 22.4 & 33.6 & 40.5 & 63.2 \\
        \bottomrule
    \end{tabular}
\end{table}

\begin{figure}[!htb]
  \centering 
  \subfloat[The value function]{\includegraphics[width=13.5em]{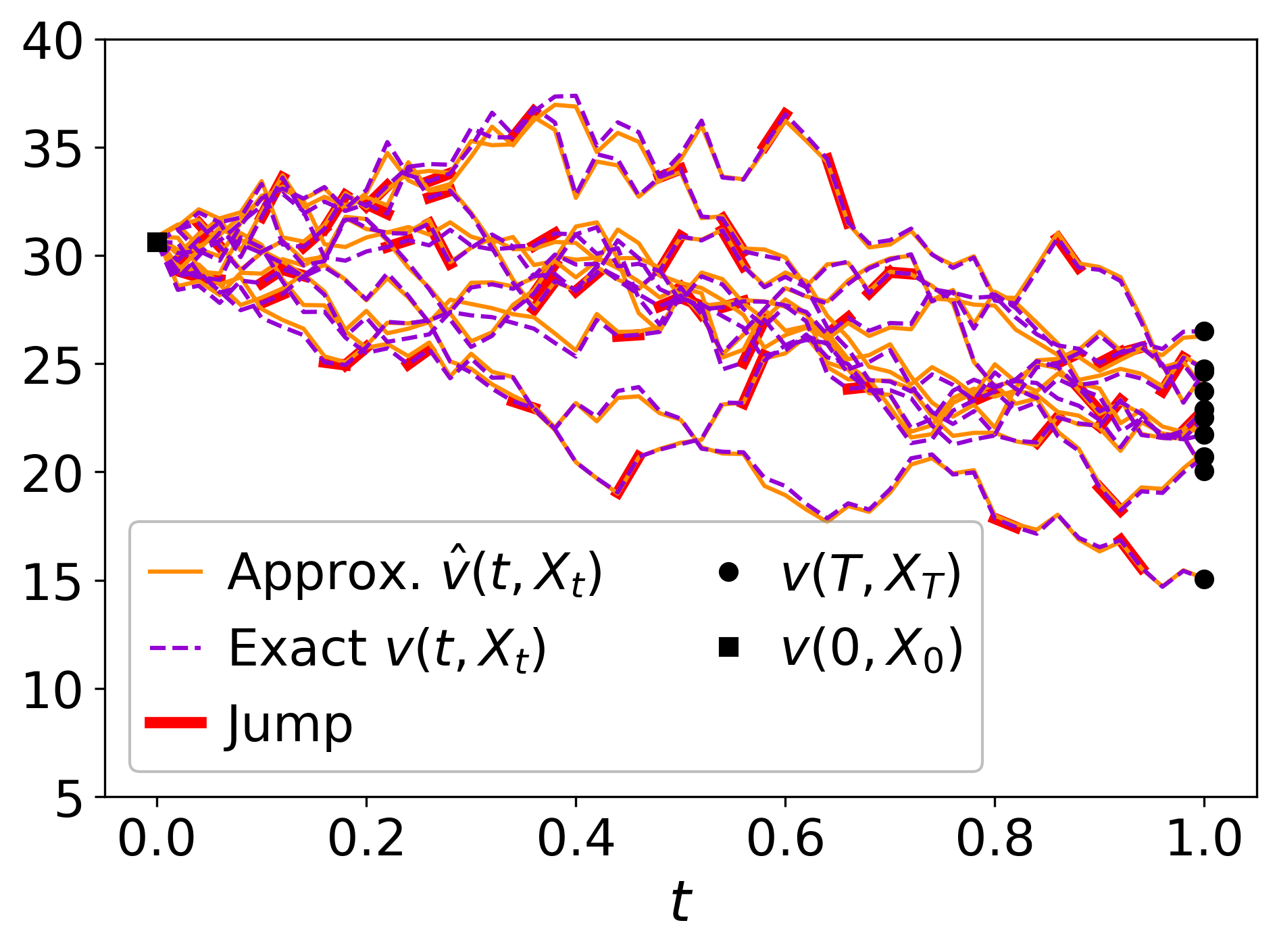}} 
  \subfloat[The $25^{th}$ entry of the control process]{\includegraphics[width=13.5em]{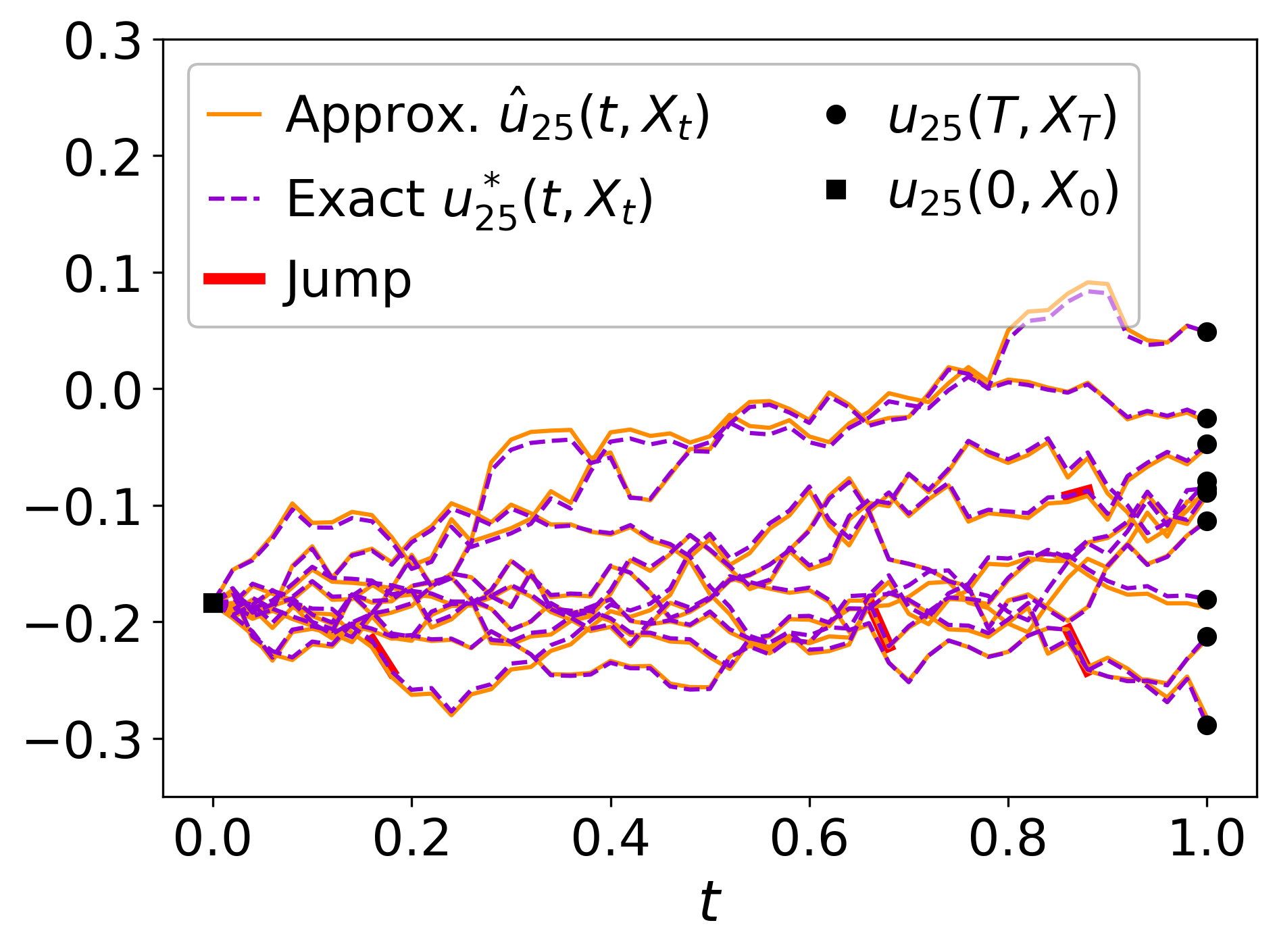}}  
 
  \subfloat[Losses during the training]{\includegraphics[width=13.5em,height=10.25em]{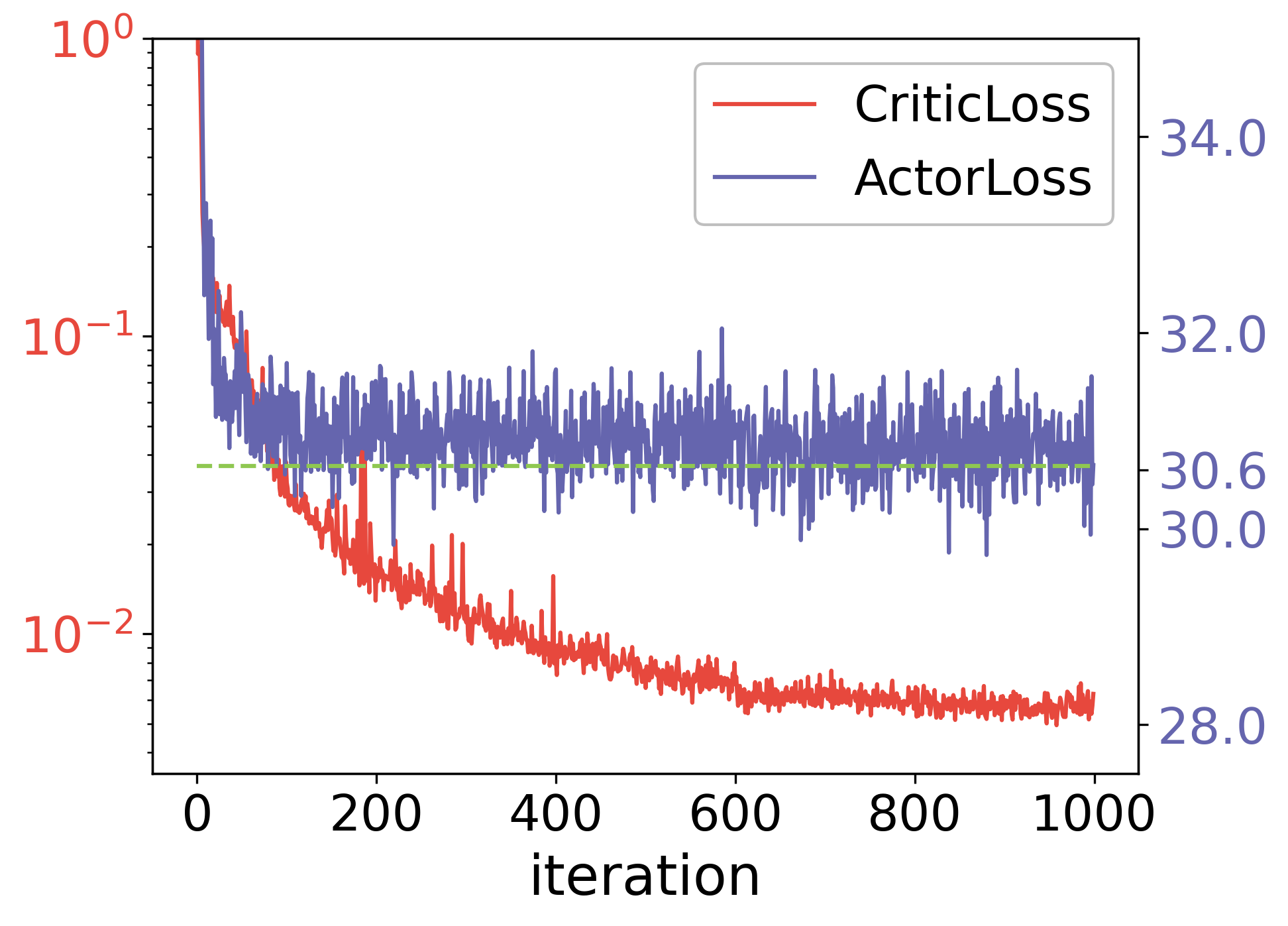}}
  \subfloat[Errors during the training]{\includegraphics[width=13.5em]{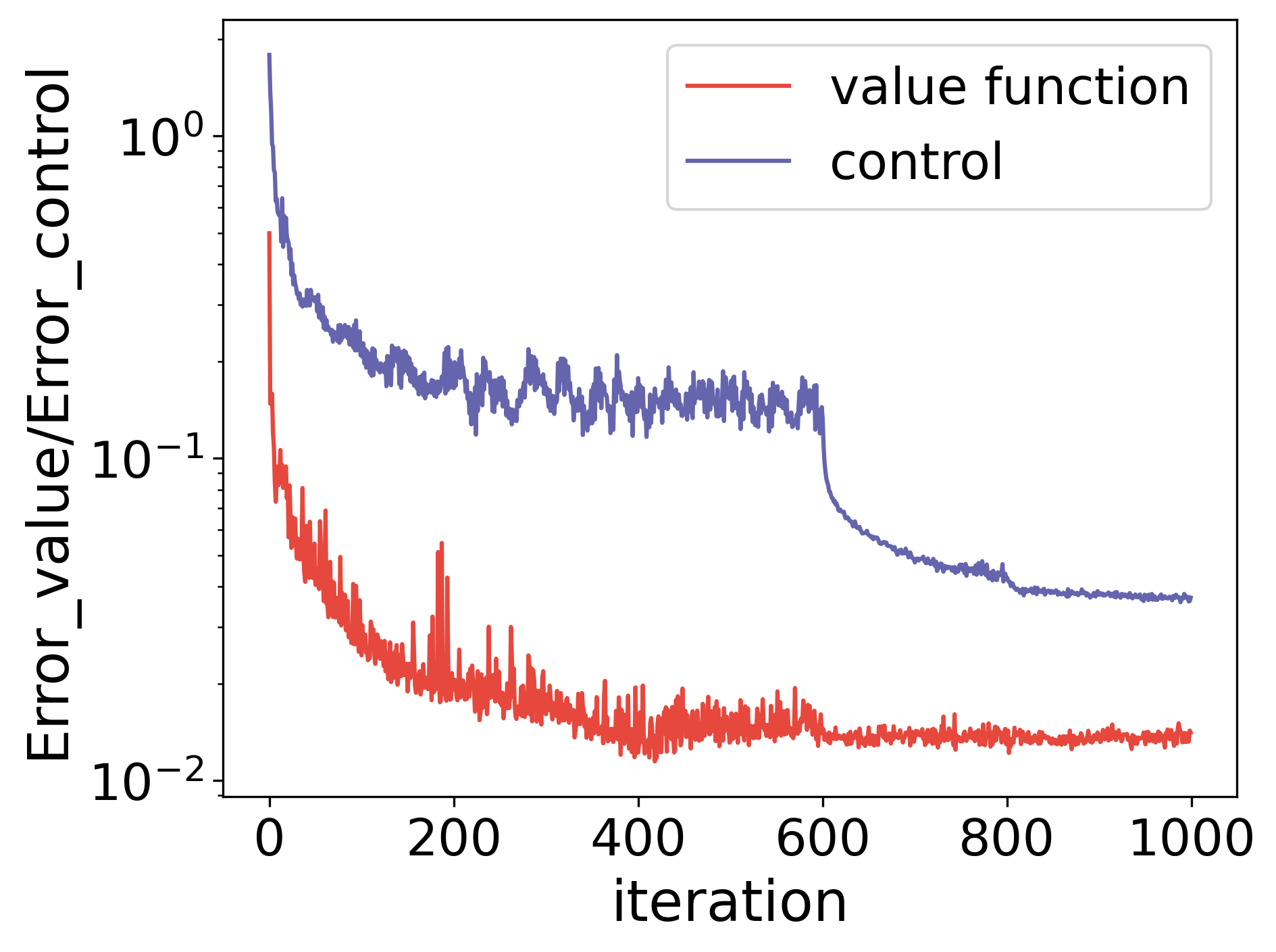}}
  \subfloat[$L^2$ relative errors $e_t^u$/$e_t^v$]{\includegraphics[width=13.5em]{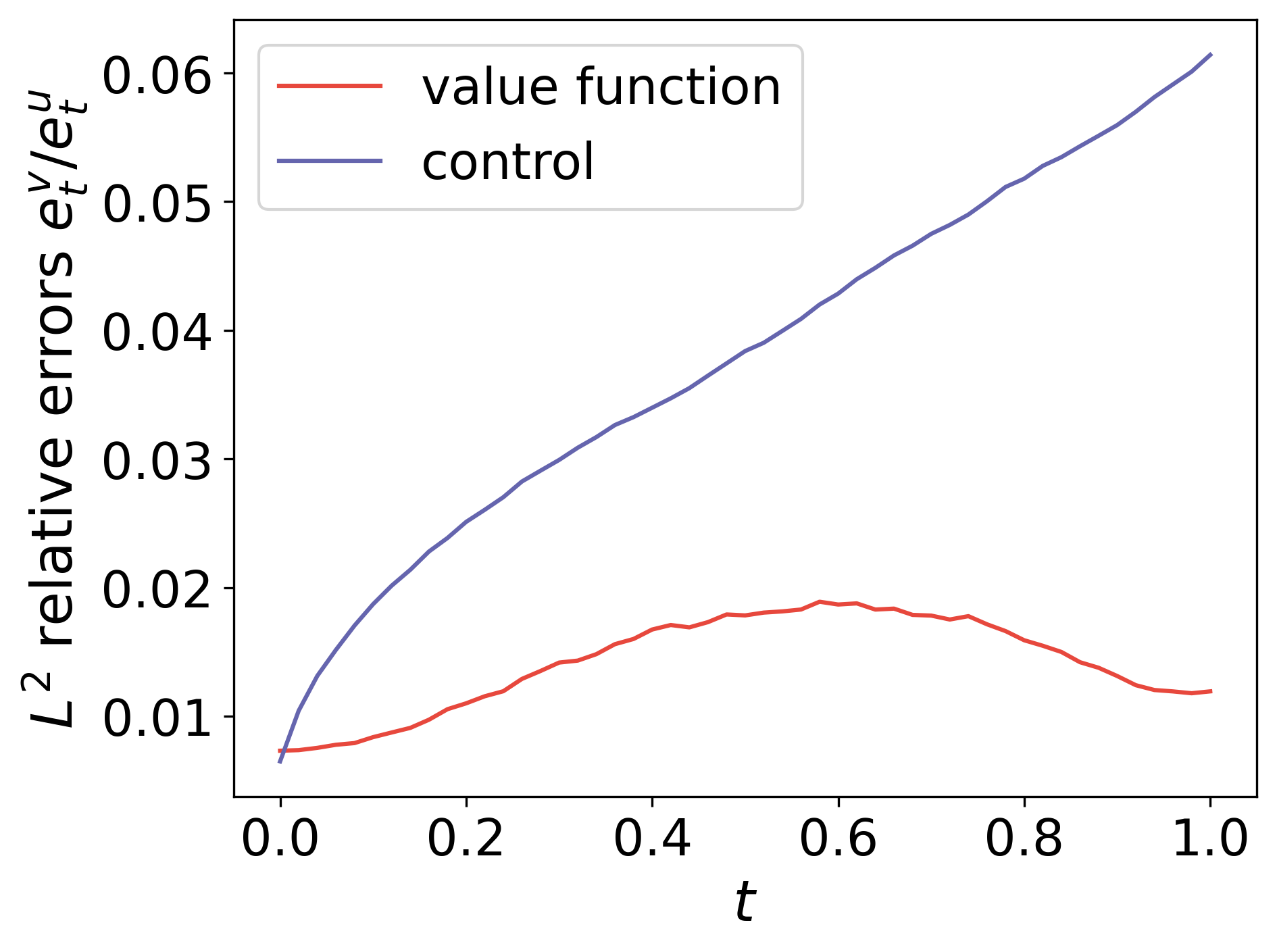}}  
  \caption{Linear quadratic regulator with jumps in 25 dimensions. (a)--(b) Visualization of sample trajectories of the value function along the optimal state process $v(t, X_t)$ and the $25^{th}$ entry of the optimal control problem $u^\ast(t, X_t)$ as well as their approximated counterparts; (c) Plot of the \texttt{CriticLoss} and the \texttt{ActorLoss} with respect to the training iterations; (d) Plot of the \texttt{Error\_value} and \texttt{Error\_control} with respect to the training iterations; and (e) $L^2$ relative errors of the value function $e_t^v$ and the control $e_t^u$ at each time $t$. The errors of the control function in (d)--(e) are computed as the average of errors in each of its dimensions.}
  \label{fig.LQR}
\end{figure}

In addition, we tested the method \eqref{eq.explosion-2} proposed in Section~\ref{sec.implementation_details}, which aims to prevent the explosion of the process $X_t$ by not restricting the output of the control but limiting the values of $X_t$ in the first few steps of training. With all other parameters unchanged in the 25-dimensional example, we limited the value of $X_t$ to $X_t=\min\{|X_t|,5\}$ in the first 100 steps of training. The final \texttt{Error\_value} is 1.332\%, and the \texttt{Error\_control} is 3.355\%, which are comparable with 1.410\% and 3.696\%, respectively. This indicates that restricting the state process at the early stage of training does not necessarily preclude the possibility of finding the optimal control.  
To conclude, we find both methods proposed in \eqref{eq.explosion}--\eqref{eq.explosion-2} to be viable solutions. 

\subsection{The multi-agent portfolio game}\label{sec.numerical.games}
In this final example, we assess the performance of Algorithm~\ref{algo.game} on the optimal investment game under relative performance in a jump-diffusion market for $n$ agents. Semi-explicit Nash equilibrium solutions, along with the corresponding optimal value functions, have been derived under exponential utility, power utility, and logarithmic utility in Sections~\ref{sec.exp}--\ref{sec.log} respectively. We will use those results as references to evaluate the accuracy of the proposed deep RL-based framework.

For the stock process $S_t^i$ traded by agent $i$ \eqref{eq.stock}, we set the parameters as follows:
\begin{gather*}
  \mu_i=\begin{cases} 0.05, & i=1 \\ 0.04, & 2\leq i\leq n \end{cases}, \quad 
  \nu_i=\begin{cases} 0.4, & i=1 \\ 0.3, & 2\leq i\leq n \end{cases}, \quad 
  \sigma_i=\begin{cases} 0.35, & i=1 \\ 0.25, & 2\leq i\leq n \end{cases}, \\ 
  \alpha_i=\begin{cases} 0.3, & i=1 \\ 0.2, & 2\leq i\leq n \end{cases}, \quad 
  \beta_i=\begin{cases} 0.25, & i=1 \\ 0.15, & 2\leq i\leq n \end{cases}, \quad 
  \lambda_i=\begin{cases} 0.3, & i=1 \\ 0.2, & 2\leq i\leq n \end{cases},
\end{gather*}
and $\lambda_0=0.25$. As for the utility functions \eqref{eq.utility_exp}, \eqref{eq.utility_power} and \eqref{eq.utility_log}, we choose
\begin{equation}
  \theta_i=\begin{cases} 0.8, & i=1 \\ 0.7, & 2\leq i\leq n \end{cases}, \quad
  \delta_i=\begin{cases} 2, & i=1 \\ 1, & 2\leq i\leq n \end{cases}, \quad
  p_i=\begin{cases} 0.5, & i=1 \\ 0.4, & 2\leq i\leq n \end{cases}.
\end{equation}
This way, the situation faced by the first agent differs from that of the other agents. To reduce computational cost, we decrease the number of residual blocks in the network architecture from 3 to 2 (cf. Figure~\ref{fig.ResNet}). By Theorems~\ref{thm.exp}, \ref{thm.power} and \ref{thm.log}, we can conclude that there exists a unique Nash equilibrium in $\{(\pi_1,\cdots,\pi_n): |\pi_i|\leq 1, \forall i=1,2,\cdots,n\}$ for exponential and power utilities, and in $\{(\pi_1,\cdots,\pi_n):1+\pi_i\alpha_i>0$ and $1+\pi_i\beta_i>0\}$ for logarithmic utility, under the above parameter choices. The Python function \texttt{scipy.optimize.fsolve} is used to solve equations~\eqref{eq.thm_exp}, \eqref{eq.thm_power} and \eqref{eq.thm_log}, to obtain reference solutions. 

Figure \ref{fig.game_main} illustrates the results for the 20-agent game under the exponential utility. The \texttt{Error\_value\_game} is 0.477\%, and the \texttt{Error\_control\_game} is 0.862\%. Panel~(a) depicts the change in errors with training iterations. Panels~(b)--(c) display the $L^2$ relative errors of the value functions and the control for all 20 agents at each time point $t$ with respect to the spatial variable $x$. 
Additionally, we conduct experiments for different numbers of agents $n$ and different types of utility functions $U$. The results are summarized in Table \ref{tab.game}. It can be observed that our algorithm performs well in various scenarios, with slightly better accuracy under exponential utility compared to logarithmic utility. Table \ref{tab.game_time} compares the computational time of parallel computing framework and non-parallel computing framework under exponential utility and different numbers of agents $n$. It can be observed that the design of the parallel computing framework greatly improves the efficiency of the algorithm.

\begin{figure}[!htb]
  \centering 
  \subfloat[Errors during the training]{\includegraphics[width=13.5em]{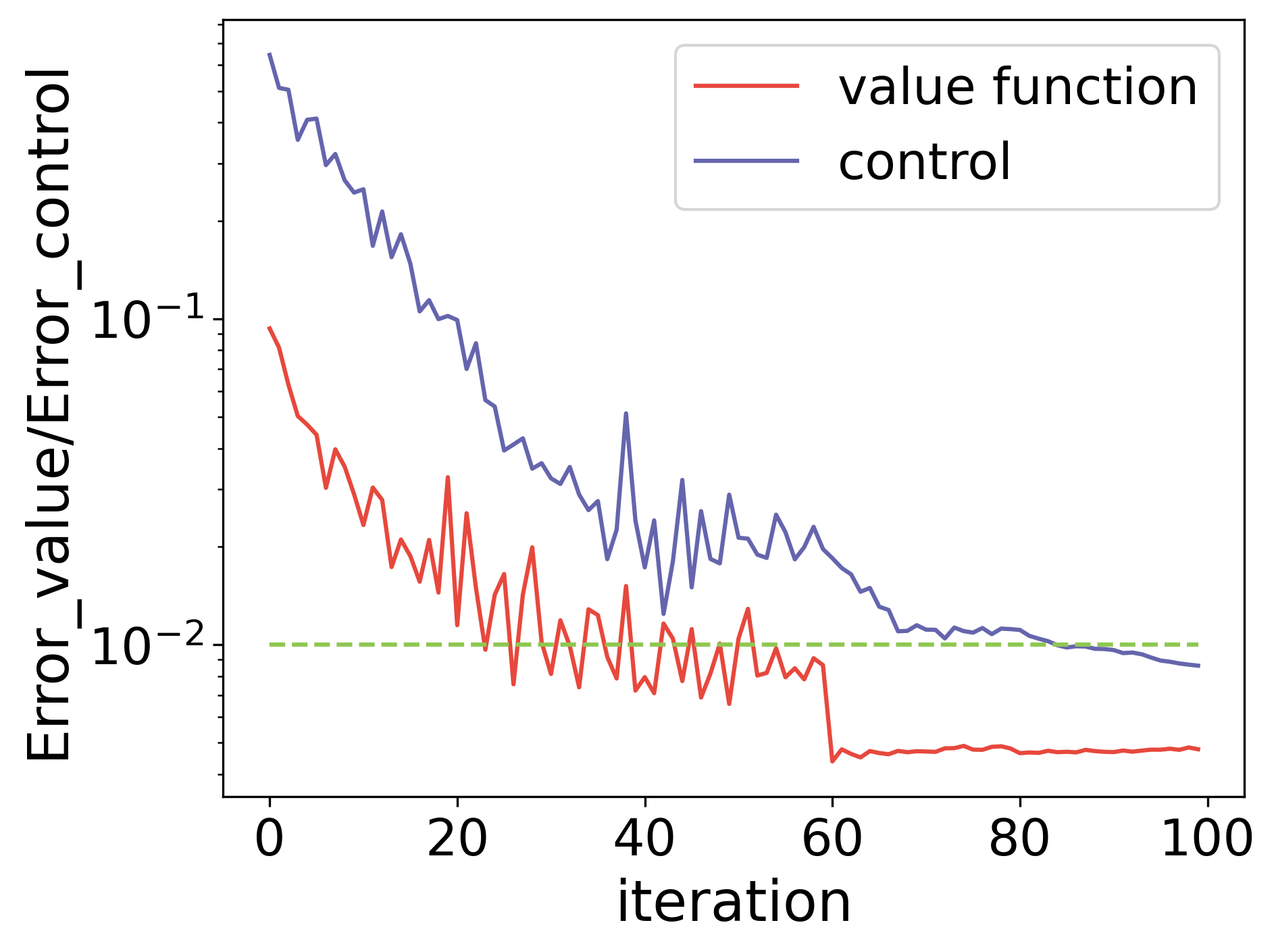}} 
  \subfloat[$L^2$ relative error $e_t^v$]{\includegraphics[width=13.5em]{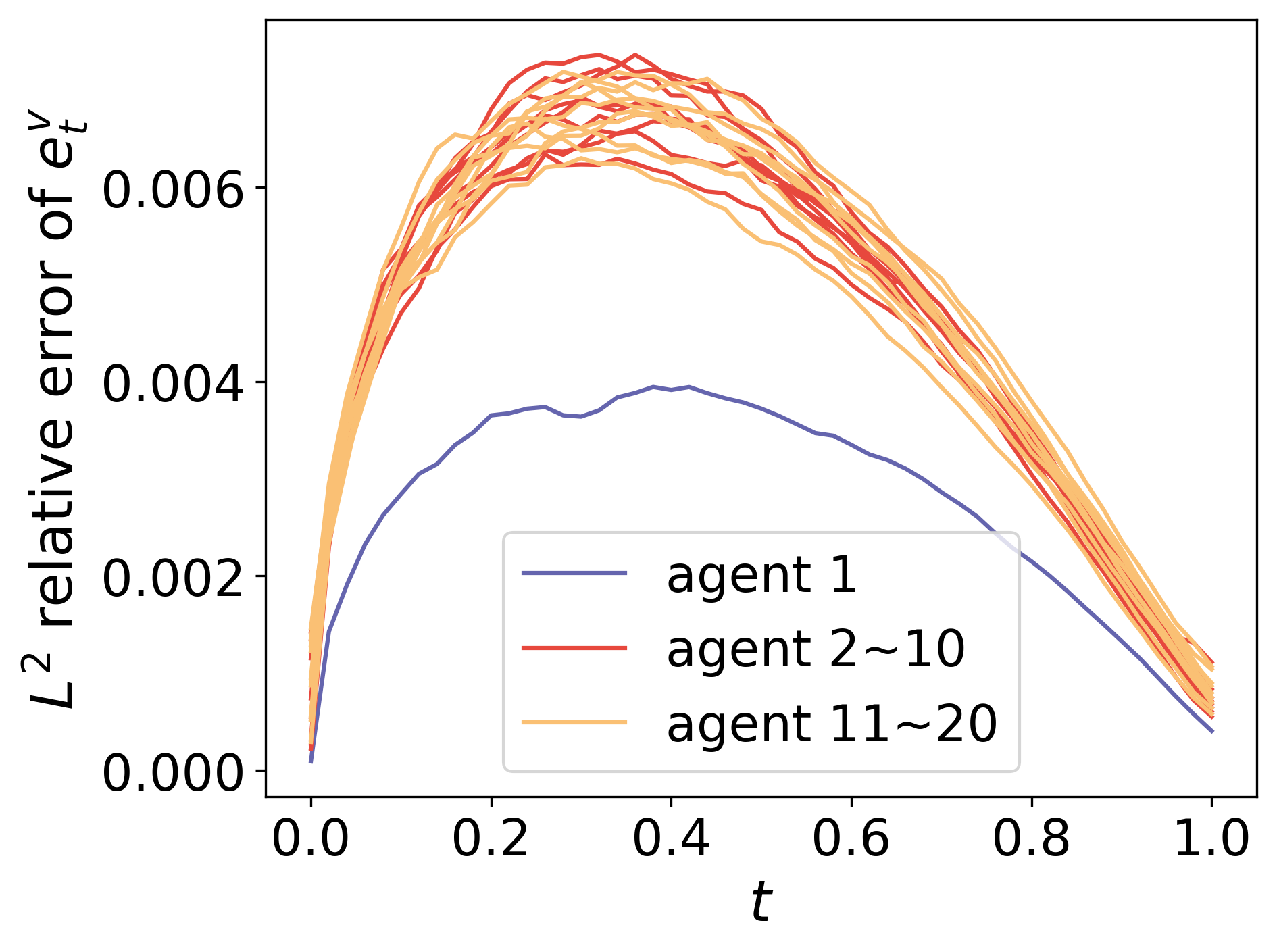}} 
  \subfloat[$L^2$ relative error $e_t^u$]{\includegraphics[width=13.5em]{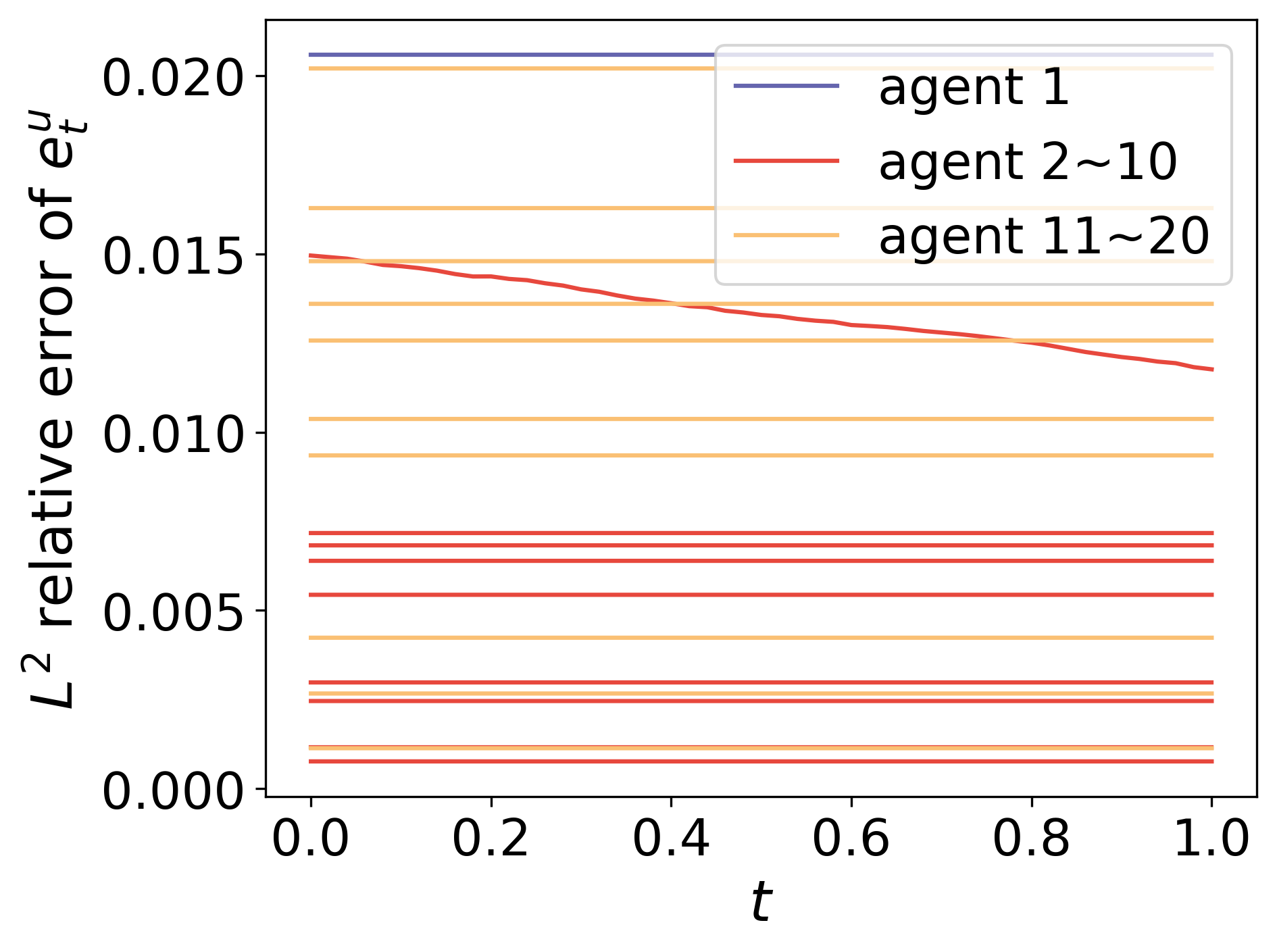}} 
  \caption{Optimal investment game with 20 agents in a jump-diffusion market under exponential utility. (a) Plot of the \texttt{Error\_value} and the \texttt{Error\_control} with respect to the training iterations; (b)--(c) $L^2$ relative error of the value function $e_t^v$ and the control $e_t^u$ with respect to the spatial variable $x$ at different time $t$ for each agent.}
  \label{fig.game_main}
\end{figure}

\begin{table}[!htb]
  \centering
  \caption{The errors in the value functions and the equilibrium strategies of the multi-agent portfolio game under different utilities and with different numbers of agents.} \label{tab.game}
  \begin{tabular}{l|cc|cc|cc}
    \toprule
    & \multicolumn{2}{c}{Exponential} & \multicolumn{2}{|c|}{Power} & \multicolumn{2}{c}{Logarithmic} \\ 
    \midrule
    & \texttt{Err\_value}  & \texttt{Err\_control} & \texttt{Err\_value}  & \texttt{Err\_control} & \texttt{Err\_value} & \texttt{Err\_control} \\
    \midrule
    $d=5$ & 0.537\% & 0.873\% & 0.237\% & 1.567\% & 1.114\% & 1.316\% \\
    $d=10$ & 0.504\% & 0.933\% & 0.266\% & 1.027\% & 1.187\% & 1.162\% \\ 
    $d=15$ & 0.511\% & 0.461\% & 0.281\% & 1.415\% & 1.194\% & 1.085\% \\ 
    $d=20$ & 0.477\% & 0.862\% & 0.280\% & 1.360\% & 1.204\% & 0.936\% \\
    \bottomrule
  \end{tabular}
\end{table}

\begin{table}[!htb]
  \centering
  \caption{The computational time (in minutes) of the multi-agent optimal investment game under exponential utility.} \label{tab.game_time}
  \begin{tabular}{lccccc}
    \toprule
    \# of agents $n$ & 2 & 4 & 6 & 8 & 10 \\
    \midrule
    Parallel computing & 57 & 97 & 143 & 309 & 497 \\ 
    Non-parallel computing & 83 & 250 & 479 & 1286 & 2086 \\
    Ratio & 0.687 & 0.388 & 0.299 & 0.240 & 0.238 \\
    \toprule
    \# of agents $n$ & 12 & 14 & 16 & 18 & 20 \\ 
    \midrule
    Parallel computing & 774 & 1076 & 1734 & 1959 & 2236 \\
    Non-parallel computing & 3164 & 4508 & 6368 & 6516 & 8180 \\
    Ratio & 0.245 & 0.239 & 0.272 & 0.301 & 0.273 \\
    \bottomrule
  \end{tabular}
\end{table}

Finally, we test a scenario in which agents are all heterogeneous, and where the parameters are chosen on purpose such that the conditions for the existence and uniqueness of the Nash equilibrium given are \emph{not} satisfied. This serves to illustrate that the given conditions are sufficient technical assumptions but not necessary. We computed this case with the exponential utility, $n = 10$ agents, and parameters for the stock price  \eqref{eq.stock} and utility \eqref{eq.utility_exp} are set as 
\begin{align*}
  &\mu_i=0.04+0.01\frac{i-1}{n-1}, \quad &&\nu_i=0.1+0.3\frac{i-1}{n-1}, \quad &&\sigma_i=0.2+0.2\frac{i-1}{n-1}, \\ 
  &\alpha_i=0.2+0.1\frac{i-1}{n-1}, \quad &&\beta_i=0.2+0.1\frac{i-1}{n-1}, \quad &&\delta_i=1+\frac{i-1}{n-1}, \\ 
  &\theta_i=0.7+0.1\frac{i-1}{n-1}, \quad &&\lambda_i=0.2+0.1\frac{i-1}{n-1}, \quad &&\lambda_0=0.25,  \qquad \forall i \in \mathcal{I}.
\end{align*}
In this case, the equilibrium strategy, obtained by running \textit{scipy.optimize.fsolve} from multiple initializations, is unique and bounded by $C = 1.2$, while the condition \eqref{eq.unique_exp_cond2} in Theorem \ref{thm.exp} is not satisfied for this bound.   
Our deep RL algorithm yields an \texttt{Error\_value\_game} of 0.152\% and an \texttt{Error\_control\_game} of 0.965\%, with the performance visualized in Figure \ref{fig.game_dissatisfy}. 
Panel~(c) is computed based on the constant Nash equilibrium derived in Theorem \ref{thm.exp},where the small errors indicate convergence to this equilibrium, despite using neural networks with inputs of both time $t$ and space $x$ to parameterize the control and value functions. Without theoretical verification, we conjecture that the Nash equilibrium in Section \ref{sec.game} can only be constant.

\begin{figure}[tbhp]
  \centering 
  \subfloat[Errors during the training]{\includegraphics[width=13.5em]{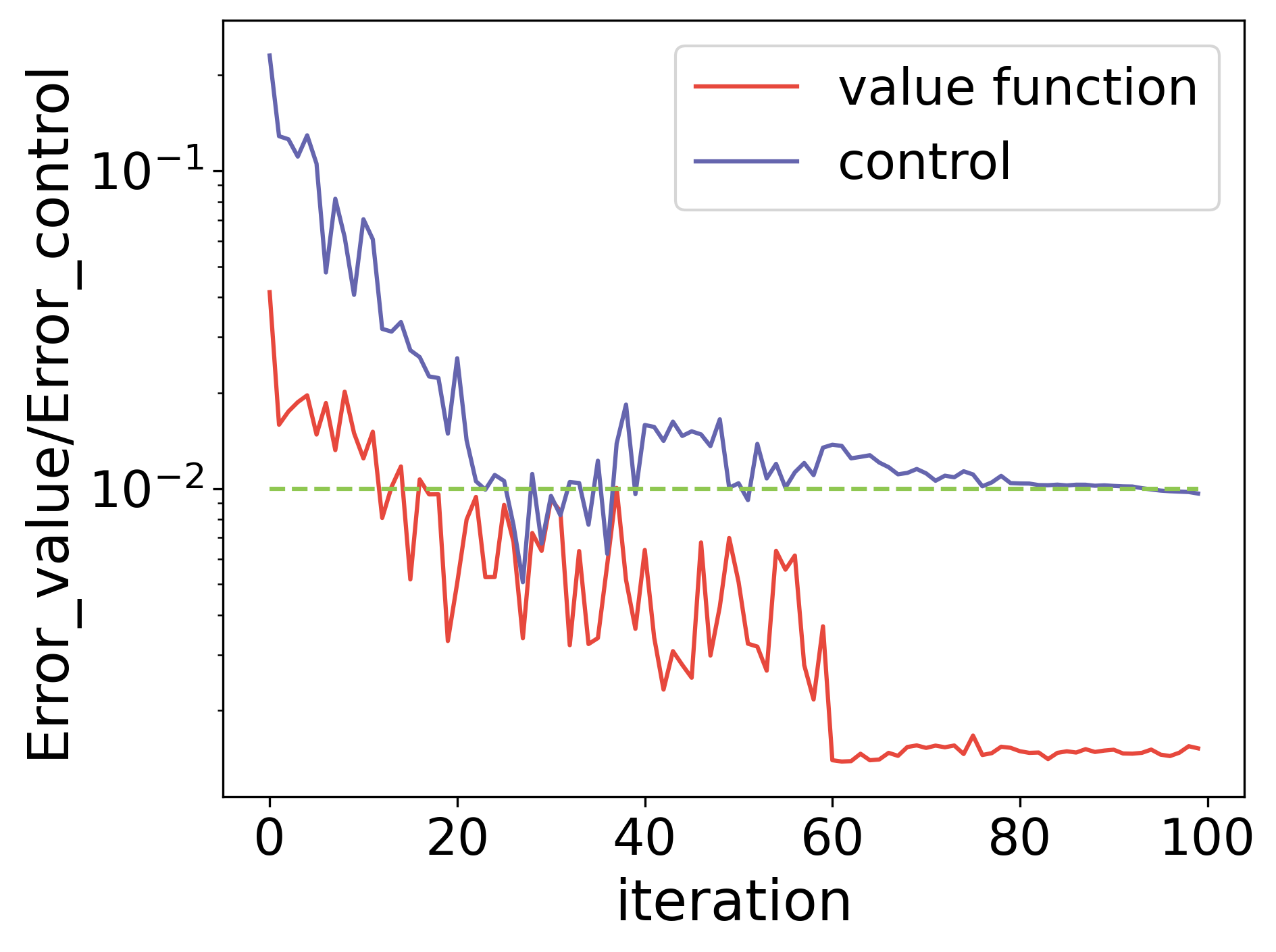}} 
  \subfloat[$L^2$ relative error $e_t^v$]{\includegraphics[width=13.5em]{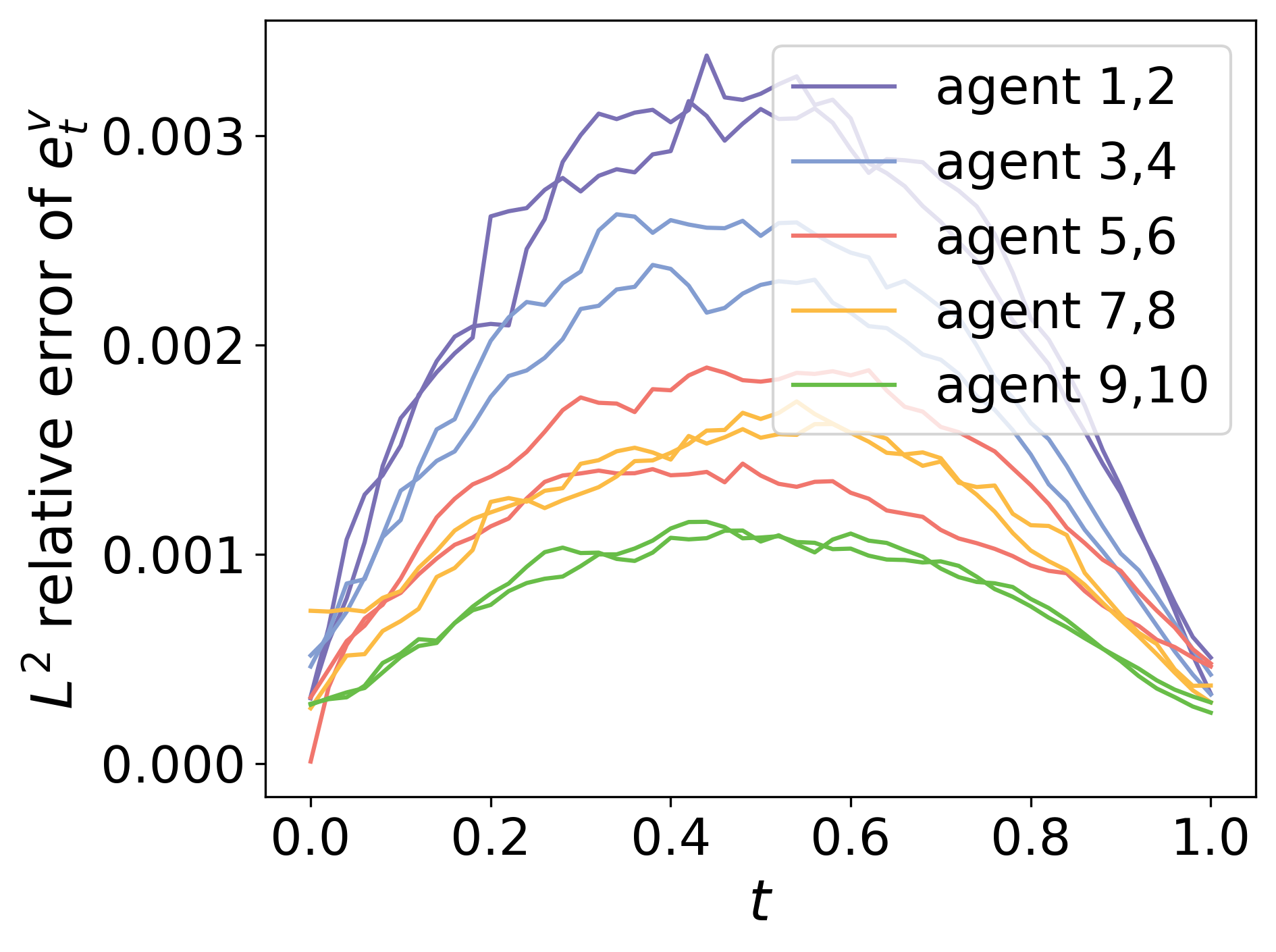}} 
  \subfloat[$L^2$ relative error $e_t^u$]{\includegraphics[width=13.5em]{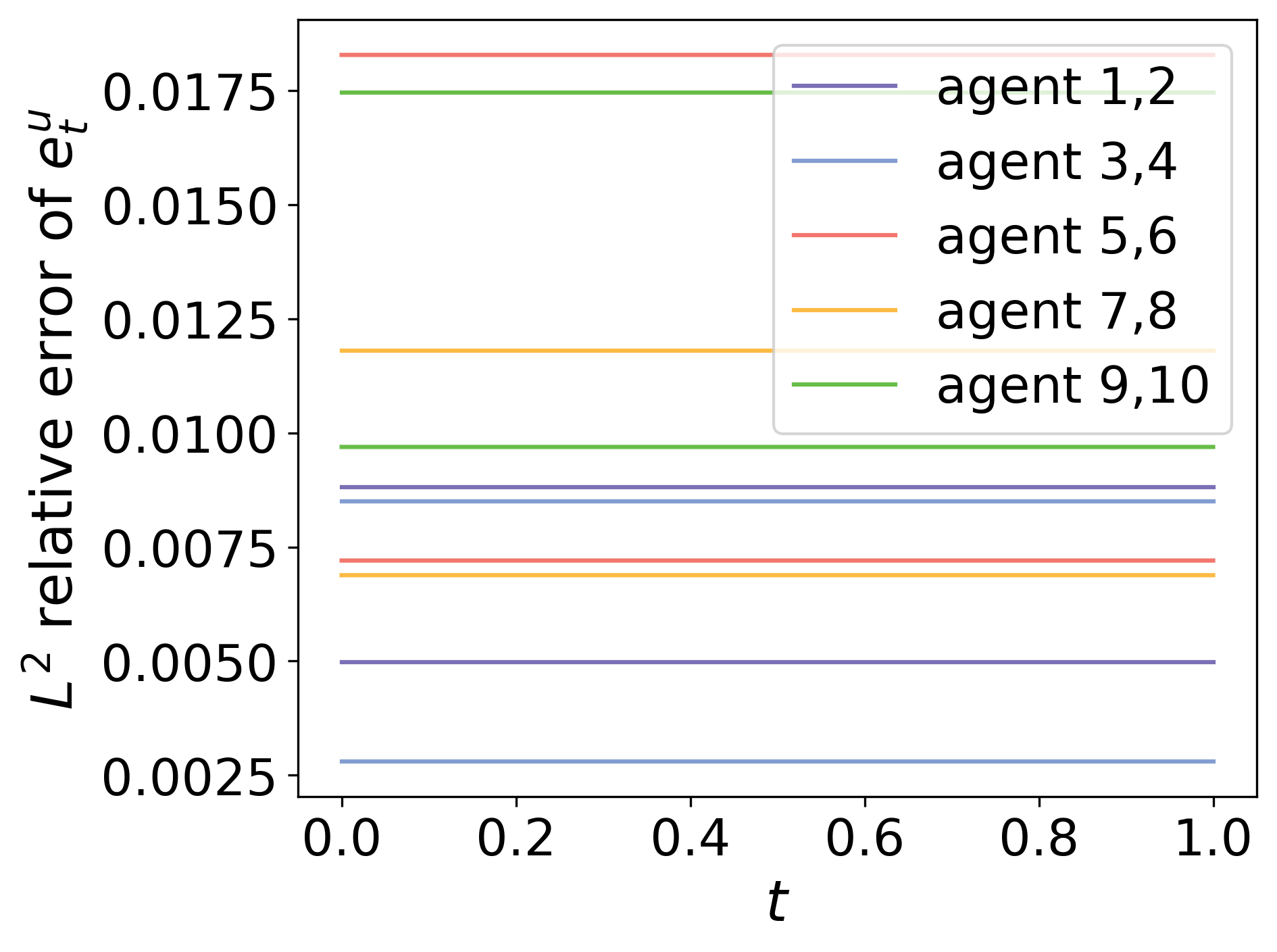}} 
  \caption{Optimal investment game with 10 agents under exponential utility which do not satisfy the condition of Theorem \ref{thm.exp}. (a) Plot of the \texttt{Error\_value} and the \texttt{Error\_control} with respect to the training iterations; and (b)--(c) $L^2$ relative error of the value function $e_t^v$ and the control $e_t^u$ with respect to the spatial variable $x$ at different time $t$ for each agent.}
  \label{fig.game_dissatisfy}
\end{figure}


\section{Conclusion} \label{sec.conclusion}
In this paper, we analyzed the multi-agent investment games in a jump diffusion market. On the theoretical side, we derive constant Nash equilibria and provide their existence and uniqueness conditions under exponential, power, and logarithmic utilities. On the numerical side, we propose a framework based on the actor-critic method in deep reinforcement learning to solve general stochastic control problems and multi-agent games with jumps, where both the diffusion and jump terms are controlled by the agents. The methods demonstrate accurate and robust numerical results in all examples including the Merton problem, the linear quadratic regulator problem, and multi-agent games of optimal investment under relative performance. 

We identify several avenues for future work. Firstly, regarding multi-agent games, we aim to enhance the framework by further leveraging parallel computing, thus reducing computational costs and ensuring good performance, even in exceptionally high-dimensional problems such as 100. Secondly, since our proposed algorithms are model-free and based on reinforcement learning, they allow for integration with real-world data without the need for calibrating the stock modeling. Finally, we are interested in conducting convergence analysis, similar to works such as \cite{gonon2021deep,zhou2024solving}, of the proposed algorithm. Although meaningful, such tasks present challenges.

\section*{Acknowledgements}
Part of this work was done during L. L.'s visit to the Department of Mathematics at the University of California, Santa Barbara. L. L. would like to express gratitude to the department for their hospitality.  

\bibliographystyle{siamplain}
\bibliography{references}

\appendix

\section{Proof of Corollary \ref{cor.unique_exp}} \label{sec.appendix_proof_exp}
  We denote
  \begin{equation}
    \begin{aligned}
      f_i(\pi_1,\pi_2,\cdots,\pi_n) := &\ \mu_i + \frac{1}{\delta_i}\sigma_i\theta_i\widehat{\pi\sigma} - \frac{1}{\delta_i}(1-\frac{\theta_i}{n})(\nu_i^2+\sigma_i^2)\pi_i - \lambda_0\beta_i - \lambda_i\alpha_i \\ 
      & + \lambda_0\beta_i e^{-\frac{1}{\delta_i}\left((1-\frac{\theta_i}{n})\pi_i\beta_i-\theta_i\widehat{\pi\beta}\right)} + \lambda_i\alpha_i e^{-\frac{1}{\delta_i}(1-\frac{\theta_i}{n})\pi_i\alpha_i} + \pi_i,
    \end{aligned}
  \end{equation}
  for $i\in\mathcal{I}$, then the nonlinear system \eqref{eq.thm_exp} can be written as $f_i(\pi_1,\pi_2,\cdots,\pi_n)=\pi_i$. Now we can compute its partial derivative 
  \begin{equation}
    \frac{\partial f_i}{\partial \pi_j} = \left\{ 
    \begin{aligned}
      & 1 - \frac{1}{\delta_i}(1-\frac{\theta_i}{n})\left((\nu_i^2+\sigma_i^2)+\lambda_0\beta_i^2 e^{-\frac{1}{\delta_i}\left((1-\frac{\theta_i}{n})\pi_i\beta_i-\theta_i\widehat{\pi\beta}\right)} + \lambda_i\alpha_i^2 e^{-\frac{1}{\delta_i}(1-\frac{\theta_i}{n})\pi_i\alpha_i}\right),\ j=i, \\ 
      & \frac{\theta_i\sigma_i\sigma_j}{n\delta_i} + \frac{1}{n\delta_i}\lambda_0\theta_i\beta_i\beta_j e^{-\frac{1}{\delta_i}\left((1-\frac{\theta_i}{n})\pi_i\beta_i-\theta_i\widehat{\pi\beta}\right)}, \ j\neq i.
    \end{aligned}
    \right.
  \end{equation}
  Due to $|\pi_i|\leq C$,  $|\beta_i|\leq\beta_0$, $|\alpha_i|\leq\alpha_0$, $0<\theta_i<1$, we have
  \begin{gather}
    \left| -\frac{1}{\delta_i}\left((1-\frac{\theta_i}{n})\pi_i\beta_i-\theta_i\widehat{\pi\beta}\right) \right|
    \leq \frac{1}{\delta_i}\left(1-\frac{\theta_i}{n}+\frac{n-1}{n}\theta_i\right)C\beta_0
    < \frac{2C\beta_0}{\delta_i}, \\ 
    \left| -\frac{1}{\delta_i}(1-\frac{\theta_i}{n})\pi_i\alpha_i \right| < \frac{C\alpha_0}{\delta_i},
  \end{gather}
  and furthermore, using condition \eqref{eq.unique_exp_cond1}, we have
  \begin{equation} \label{eq.appendix_exp_dfdi}
    \begin{aligned}
      & \left| \frac{1}{\delta_i}(1-\frac{\theta_i}{n})\left((\nu_i^2+\sigma_i^2)+\lambda_0\beta_i^2 e^{-\frac{1}{\delta_i}\left((1-\frac{\theta_i}{n})\pi_i\beta_i-\theta_i\widehat{\pi\beta}\right)} + \lambda_i\alpha_i^2 e^{-\frac{1}{\delta_i}(1-\frac{\theta_i}{n})\pi_i\alpha_i}\right) \right| \\ 
      & \leq \frac{1}{\underline{\delta}}\left(\overline{\nu}^2+\overline{\sigma}^2+\lambda_0\beta_0^2 e^{\frac{2C\beta_0}{\underline{\delta}}} + \overline{\lambda}\alpha_0^2 e^{\frac{C\alpha_0}{\underline{\delta}}}\right) < 1,
    \end{aligned}
  \end{equation}
  and
  \begin{equation} \label{eq.appendix_exp_dfdj}
      \left| \frac{\theta_i\sigma_i\sigma_j}{n\delta_i} + \frac{1}{n\delta_i}\lambda_0\theta_i\beta_i\beta_j e^{-\frac{1}{\delta_i}\left((1-\frac{\theta_i}{n})\pi_i\beta_i-\theta_i\widehat{\pi\beta}\right)} \right| 
      \leq \frac{\theta_i}{n\delta_i}\left(\sigma_i\overline{\sigma} + \lambda_0\beta_0^2e^{\frac{2C\beta_0}{\delta_i}}\right).
  \end{equation}
  Combining equations \eqref{eq.appendix_exp_dfdi} and \eqref{eq.appendix_exp_dfdj} gives the estimate of the Jacobian matrix
  \begin{equation}
    \begin{aligned}
      \left \| \left(\frac{\partial f_i}{\partial \pi_j}\right)_{ij} \right \|_{\infty} = & \max_{1\leq i\leq n} \left[\sum_{j=1}^n \left|\frac{\partial f_i}{\partial \pi_j}\right|\right] \\ 
      = & \max_{1\leq i\leq n} \left[ 1 - \frac{1}{\delta_i}(1-\frac{\theta_i}{n})\left((\nu_i^2+\sigma_i^2)+\lambda_0\beta_i^2 e^{-\frac{1}{\delta_i}\left((1-\frac{\theta_i}{n})\pi_i\beta_i-\theta_i\widehat{\pi\beta}\right)} \right.\right. \\
      & + \left.\left. \lambda_i\alpha_i^2 e^{-\frac{1}{\delta_i}(1-\frac{\theta_i}{n})\pi_i\alpha_i}\right) 
      + \sum_{j\neq i} \left| \frac{\theta_i\sigma_i\sigma_j}{n\delta_i} + \frac{1}{n\delta_i}\lambda_0\theta_i\beta_i\beta_j e^{-\frac{1}{\delta_i}\left((1-\frac{\theta_i}{n})\pi_i\beta_i-\theta_i\widehat{\pi\beta}\right)} \right| \right] \\ 
      \leq & \max_{1\leq i\leq n} \left[  1 - \frac{1}{\delta_i}(1-\frac{\theta_i}{n})(\nu_i^2+\sigma_i^2) + \frac{n-1}{n}\frac{\theta_i}{\delta_i}\left(\sigma_i\overline{\sigma} + \lambda_0\beta_0^2e^{\frac{2C\beta_0}{\delta_i}}\right) \right] \\ 
      \leq & \max_{1\leq i\leq n} \left[ 1 - \frac{1}{\delta_i}(1-\frac{\theta_i}{n})\left(\nu_i^2+\sigma_i^2-\sigma_i\overline{\sigma}-\lambda_0\beta_0^2e^{\frac{2C\beta_0}{\delta_i}}\right) \right] \\ 
      \leq & \max_{1\leq i\leq n} \left[ 1 - \frac{1}{2\overline{\delta}}\left(\underline{\nu}^2-\overline{\sigma}(\overline{\sigma}-\underline{\sigma})-\lambda_0\beta_0^2 e^{\frac{2C\beta_0}{\underline{\delta}}}\right) \right] \leq 1-\frac{K}{2\overline{\delta}}.
    \end{aligned}
  \end{equation}
  Here, because of condition \eqref{eq.unique_exp_cond2},
  \begin{equation}
    K:=\underline{\nu}^2-\overline{\sigma}(\overline{\sigma}-\underline{\sigma})-\lambda_0\beta_0^2 e^{\frac{2C\beta_0}{\underline{\delta}}} > 0.
  \end{equation}
For any $\overrightarrow{\pi}^0\in\{(\pi_1,\cdots,\pi_n): |\pi_i|\leq C, \forall i\in\mathcal{I}\}$, construct a sequence $\overrightarrow{\pi}^{k+1}=f(\overrightarrow{\pi}^{k})$, then 
  \begin{gather}
    \left \| \overrightarrow{\pi}^{k+1}-\overrightarrow{\pi}^{k} \right \|_{\infty} = \left \| f(\overrightarrow{\pi}^{k})-f(\overrightarrow{\pi}^{k-1}) \right \|_{\infty} \leq L\left \| \overrightarrow{\pi}^{k}-\overrightarrow{\pi}^{k-1} \right \|_{\infty} , \ L=1-\frac{K}{2\overline{\delta}}\in(0,1).
  \end{gather}
  Therefore, $\{\overrightarrow{\pi}^{k}\}_{k=0}^\infty$ is a Cauchy sequence, with a limit $\overrightarrow{\pi}^*$, and $\overrightarrow{\pi}^*$ satisfies \eqref{eq.thm_exp}. Thus, the nonlinear system \eqref{eq.thm_exp} has a solution in $\{(\pi_1,\cdots,\pi_n): |\pi_i|\leq C, \forall i\in\mathcal{I}\}$.

  On the other hand, if the system \eqref{eq.thm_exp} has two different solutions $\overrightarrow{\pi}^*, \overrightarrow{\pi}^{**}$ in $\{(\pi_1,\cdots,\pi_n): |\pi_i|\leq C, \forall i\in\mathcal{I}\}$, then 
  \begin{equation}
    \left \| \overrightarrow{\pi}^*-\overrightarrow{\pi}^{**} \right \|_{\infty} = \left \| f(\overrightarrow{\pi}^*)-f(\overrightarrow{\pi}^{**}) \right \|_{\infty} \leq L\left \| \overrightarrow{\pi}^*-\overrightarrow{\pi}^{**} \right \|_{\infty} < \left \| \overrightarrow{\pi}^*-\overrightarrow{\pi}^{**} \right \|_{\infty}, 
  \end{equation}
  which is a contradiction.

\section{Proof of Corollary \ref{cor.unique_power}} \label{sec.appendix_proof_power}
  We denote
  \begin{equation}
    \begin{aligned}
      f_i(\pi_1,\pi_2,\cdots,\pi_n) := &\ \mu_i + (\nu_i^2+\sigma_i^2)\left(p_i(1-\frac{\theta_i}{n})-1\right)\pi_i-p_i\theta_i\sigma_i\widehat{\pi\sigma} - \lambda_i\alpha_i-\lambda_0\beta_i \\ 
      & + \lambda_i\alpha_i(1+\pi_i\alpha_i)^{p_i(1-\frac{\theta_i}{n})-1} + \lambda_0\beta_i\frac{(1+\pi_i\beta_i)^{p_i(1-\frac{\theta_i}{n})-1}}{\widetilde{1+\pi\beta}^{p_i\theta_i}} + \pi_i,
    \end{aligned}
  \end{equation}
  for $i\in\mathcal{I}$, then the nonlinear system \eqref{eq.thm_power} can be written as $f_i(\pi_1,\pi_2,\cdots,\pi_n)=\pi_i$. Now we can compute its partial derivative 
  \begin{small}
  \begin{equation}
    \frac{\partial f_i}{\partial \pi_j} = \left\{ 
      \begin{aligned}
        & 1 - \left(1-p_i(1-\frac{\theta_i}{n})\right)\left((\nu_i^2+\sigma_i^2) + \lambda_i\alpha_i^2(1+\pi_i\alpha_i)^{p_i(1-\frac{\theta_i}{n})-2} + \lambda_0\beta_i^2\frac{(1+\pi_i\beta_i)^{p_i(1-\frac{\theta_i}{n})-2}}{\widetilde{1+\pi\beta}^{p_i\theta_i}}\right),\ j=i, \\ 
        & -\frac{p_i\theta_i\sigma_i\sigma_j}{n} - \frac{\lambda_0p_i\theta_i}{n}\frac{\beta_i\beta_j}{1+\pi_j\beta_j}\frac{(1+\pi_i\beta_i)^{p(1-\frac{\theta_i}{n})-1}}{\widetilde{1+\pi\beta}^{p_i\theta_i}}, \ j\neq i.
      \end{aligned}
    \right.
  \end{equation}
  \end{small}
  Due to $|\pi_i|\leq C\leq1$, $|\alpha_i|\leq\alpha_0<1$, $|\beta_i|\leq\beta_0<1$, $0<p_i<1$, $0<\theta_i<1$, we have 
  \begin{gather}
    \left|1+\pi_i\alpha_i\right|^{p_i(1-\frac{\theta_i}{n})-2} \leq \left(\frac{1}{1-C\alpha_0}\right)^{2-p_i(1-\frac{\theta_i}{n})} \leq \left(\frac{1}{1-C\alpha_0}\right)^{2-0.5\underline{p}}, \\ 
    \left|\frac{(1+\pi_i\beta_i)^{p_i(1-\frac{\theta_i}{n})-2}}{\widetilde{1+\pi\beta}^{p_i\theta_i}}\right| \leq \left(\frac{1}{1-C\beta_0}\right)^{2-p_i+p_i\theta_i} \leq \left(\frac{1}{1-C\beta_0}\right)^2,
  \end{gather}
  and furthermore, using condition \eqref{eq.unique_power_cond1}, we have
  \begin{equation} \label{eq.appendix_power_dfdi}
    \begin{aligned}
      & \left(1-p_i(1-\frac{\theta_i}{n})\right)\left((\nu_i^2+\sigma_i^2) + \lambda_i\alpha_i^2(1+\pi_i\alpha_i)^{p_i(1-\frac{\theta_i}{n})-2} + \lambda_0\beta_i^2\frac{(1+\pi_i\beta_i)^{p_i(1-\frac{\theta_i}{n})-2}}{\widetilde{1+\pi\beta}^{p_i\theta_i}}\right) \\
      & \leq (1-0.5\underline{p})\left(\overline{\nu}^2+\overline{\sigma}^2+\overline{\lambda}\frac{\alpha_0^2}{(1-C\alpha_0)^{2-0.5\underline{p}}} + \lambda_0\frac{\beta_0^2}{(1-C\beta_0)^2}\right) < 1,
    \end{aligned}
  \end{equation}
  and meanwhile
  \begin{gather}
    \left|\frac{\beta_j}{1+\pi_j\beta_j}\right| = \frac{1}{\left|\frac{1}{\beta_j}+\pi_j\right|} \leq \frac{1}{\frac{1}{|\beta_j|}-|\pi_j|} \leq \frac{1}{\frac{1}{\beta_0}-C} = \frac{\beta_0}{1-C\beta_0}, \\ 
    \sum_{j\neq i} \left|\frac{\partial f_i}{\partial \pi_j}\right| \leq \frac{n-1}{n}p_i\theta_i\left(\sigma_i\overline{\sigma} + \lambda_0|\beta_i|\frac{\beta_0}{1-C\beta_0}\frac{(1+\pi_i\beta_i)^{p_i(1-\frac{\theta_i}{n})-1}}{\widetilde{1+\pi\beta}^{p_i\theta_i}}\right). \label{eq.appendix_power_dfdj}
  \end{gather}
  Combining equations \eqref{eq.appendix_power_dfdi} and \eqref{eq.appendix_power_dfdj} brings the estimate of the Jacobian matrix
  \begin{equation}
    \begin{aligned}
      \left \| \left(\frac{\partial f_i}{\partial \pi_j}\right)_{ij} \right \|_{\infty} = &\ \max_{1\leq i\leq n} \left[\sum_{j=1}^n \left|\frac{\partial f_i}{\partial \pi_j}\right|\right] \\ 
      = &\ \max_{1\leq i\leq n} 1 - \left[\left(1-p_i(1-\frac{\theta_i}{n})\right)\left((\nu_i^2+\sigma_i^2) + \lambda_i\alpha_i^2(1+\pi_i\alpha_i)^{p_i(1-\frac{\theta_i}{n})-2} \right.\right.\\
      & + \left. \lambda_0\beta_i^2\frac{(1+\pi_i\beta_i)^{p_i(1-\frac{\theta_i}{n})-2}}{\widetilde{1+\pi\beta}^{p_i\theta_i}}\right) \\
      & - \left. \sum_{j\neq i}\left|-\frac{p_i\theta_i\sigma_i\sigma_j}{n} - \frac{\lambda_0p_i\theta_i}{n}\frac{\beta_i\beta_j}{1+\pi_j\beta_j}\frac{(1+\pi_i\beta_i)^{p_i(1-\frac{\theta_i}{n})-1}}{\widetilde{1+\pi\beta}^{p_i\theta_i}}\right|\right].
    \end{aligned}
  \end{equation}
  Using $1-p_i(1-\frac{\theta_i}{n})\geq\frac{n-1}{n}p_i\theta_i$ and because of the condition \eqref{eq.unique_power_cond2}:
  \begin{equation}
    \begin{aligned}
      \left \| \left(\frac{\partial f_i}{\partial \pi_j}\right)_{ij} \right \|_{\infty} \leq &\ \max_{1\leq i\leq n} 1-(1-p_i(1-\frac{\theta_i}{n}))\left[\nu_i^2+\sigma_i(\sigma_i-\overline{\sigma}) \right. \\ 
      & + \left.\lambda_0|\beta_i|\frac{(1+\pi_i\beta_i)^{p_i(1-\frac{\theta_i}{n})-2}}{\widetilde{1+\pi\beta}^{p\theta_i}}\left(|\beta_i|-(1+\pi_i\beta_i)\frac{\beta_0}{1-C\beta_0}\right) \right] \\
      \leq &\ 1 - (1-\overline{p})\left(\underline{\nu}^2-\overline{\sigma}(\overline{\sigma}-\underline{\sigma})-\lambda_0\beta_0^2\frac{1}{(1-C\beta_0)^2}\frac{1+C\beta_0}{1-C\beta_0}\right) \\ 
      = & \ 1 - (1-\overline{p})K < 1.
    \end{aligned}
  \end{equation}
  Here, because of condition \eqref{eq.unique_power_cond3},
  \begin{equation}
    K := \underline{\nu}^2-\overline{\sigma}(\overline{\sigma}-\underline{\sigma})-\lambda_0\beta_0^2\frac{1}{(1-C\beta_0)^2}\frac{1+C\beta_0}{1-C\beta_0} > 0.
  \end{equation}Similar to the proof in Appendix \ref{sec.appendix_proof_exp}, $f$ is a contraction mapping on $\{(\pi_1,\cdots,\pi_n): |\pi_i|\leq C, \forall i\in\mathcal{I}\}$, therefore, the system \eqref{eq.thm_power} has a unique solution in $\{(\pi_1,\cdots,\pi_n): |\pi_i|\leq C, \forall i\in\mathcal{I}\}$.

\section{Proof of Theorem \ref{thm.log}} \label{sec.appendix_proof_log}
  We denote the function of one variable
  \begin{equation}
    f_i(\pi_i) = \mu_i-(\nu_i^2+\sigma_i^2)\pi_i^*+\lambda_i\alpha_i\left(\frac{1}{1+\pi_i^*\alpha_i}-1\right)+\lambda_0\beta_i\left(\frac{1}{1+\pi_i^*\beta_i}-1\right),
  \end{equation}
  for $i\in\mathcal{I}$, then the equation \eqref{eq.thm_log} can be written as $f_i(\pi_i)=0$. We compute its derivative
  \begin{equation}
    f'_i(\pi_i) = -(\nu_i^2+\sigma_i^2)-\frac{\lambda_i\alpha_i^2}{(1+\pi_i\alpha_i)^2}-\frac{\lambda_0\beta_i^2}{(1+\pi_i\beta_i)^2} < 0,
  \end{equation}
  on $\Omega:=\{\pi_i:1+\pi_i\alpha_i>0,1+\pi_i\beta_i>0\}$, because $\nu_i^2+\sigma_i^2>0$. The shape of the region $\Omega$ can take several possible forms, including:
  
  \vspace{0.5em}
  (a) $\alpha_i=\beta_i=0$: $\Omega=\R$, and combining with $f_i(-\infty)=+\infty,f_i(+\infty)=-\infty$ and the strictly monotonic decreasing property of $f_i(\pi_i)$ on $\Omega$, we can conclude that $f_i(\pi_i)=0$ has a unique solution in $\Omega$;

  (b) $\alpha_i=0$ or $\beta_i=0$: Without loss of generality, let's assume $\alpha_i=0$. Then if $\beta_i>0$, we have $\Omega=(-1/\beta_i,+\infty)$, $f_i(-1/\beta_i+)=+\infty,f_i(+\infty)=-\infty$, so $f_i(\pi_i)=0$ has a unique solution in $\Omega$; else $\beta_i<0$, we have $\Omega=(-\infty,-1/\beta_i)$, $f_i(-\infty)=+\infty,f_i(-1/\beta_i-)=-\infty$, so $f_i(\pi_i)=0$ also has a unique solution in $\Omega$;

  (c) $\alpha_i\neq0,\beta_i\neq0$:

  (c.1) $\alpha_i=\beta_i$: the derivation is same with case (b);

  (c.2) $\alpha_i\neq\beta_i$: Without loss of generality, let's assume $\alpha_i>\beta_i$. Then:
  
  (c.2.1) $\beta_i<\alpha_i<0$: $\Omega=(-\infty,-1/\beta_i)$, and combining with $f_i(-\infty)=+\infty,f_i(-1/\beta_i-)=-\infty$, we can conclude that $f_i(\pi_i)=0$ has a unique solution in $\Omega$;

  (c.2.2) $\beta_i<0<\alpha_i$: $\Omega=(-1/\alpha_i,-1/\beta_i)$, and combining with $f_i(-1/\alpha+)=+\infty,f_i(-1/\beta_i-)=-\infty$, we can conclude that $f_i(\pi_i)=0$ has a unique solution in $\Omega$;

  (c.2.3) $0<\beta_i<\alpha_i$: $\Omega=(-1/\alpha_i,+\infty)$, and combining with $f_i(-1/\alpha_i+)=+\infty,f_i(+\infty)=-\infty$, we can conclude that $f_i(\pi_i)=0$ has a unique solution in $\Omega$;

  \vspace{0.5em}
  Combining the above discussion, equation \eqref{eq.thm_log} has a unique solution $\pi_i^*$ that satisfies $1+\pi_i^*\alpha_i>0$ and $1+\pi_i^*\beta_i>0$.

\end{document}